\documentclass[tbtags,oneside]{article}
\usepackage{amsmath,amssymb,amsthm,amsfonts}
\usepackage[T1]{fontenc}
\usepackage{sectsty}
\usepackage{mathtools}
\usepackage{bm}
\usepackage{comment}
\usepackage[indent=20pt, skip=5pt]{parskip}
\usepackage{subcaption}
\DeclarePairedDelimiter{\norm}{\lVert}{\rVert}
\usepackage{xcolor}
\usepackage[colorlinks=true,linkcolor=blue,citecolor=blue]{hyperref}
\usepackage{graphicx}
\usepackage{upgreek, relsize}
\usepackage[shortlabels]{enumitem}
\usepackage[left=3.5cm,right=3.5cm]{geometry}
\usepackage{mathrsfs}
\usepackage{cite}
\usepackage{authblk}

\newtheorem{theorem}{Theorem}[section]
\newtheorem{prop}[theorem]{Proposition}
\newtheorem{lemma}[theorem]{Lemma}

\theoremstyle{definition}
\newtheorem{remark}[theorem]{Remark}
\newtheorem{definition}[theorem]{Definition}

\numberwithin{equation}{section}

\newtheorem{conj}{Conjecture}[section]
\def\R{\mathbb{R}}
\def\C{\mathbb{C}}

\def\Z{\mathbb{Z}}
\def\P{\mathbb{P}}
\def\E{\mathbb{E}\,}

\def\g{\gamma}

\def\d{\mathrm{d}}
\def\de{\delta}

\def\G{\Lambda}
\def\bG{\partial \G}
\def\D{D}

\def\Dd{\D_\de}

\def\w{\mathrm{w}}

\def\N{N}
\def\eps{\varepsilon}
\def\L{\mathcal{L}}

\def\n{\mathbf{n}}

\def\CLE{\textnormal{CLE}}
\def\GFF{h}
\def\IMF{\Phi}
\def\A{\mathbb{A}}
\def\CC{\mathbb{C}}
\def\P{\mathbb{P}}
\def\E{\mathbb{E}}
\def\w{\omega}
\def\wFK{\omega^{\text{FK}}}
\def\L{\Lambda}

\def\a{\alpha}
\def\Z{\mathbb{Z}}
\def\G{\Gamma}
\def\g{\gamma}
\def\Int{\mathrm{In}}

\def\F{\mathcal{F}}
\def\b{\partial}

\let\geq\geqslant                 
\let\leq\leqslant
\let\ge\geqslant

\def\Dd{D_\de}
\def\vp{\varphi}
\def\at{\textnormal{AT}}
\def\at{\textnormal{AT}}

\def\coin{\mathscr{S}}

\def\dual{\dagger}
\def\free{\mathrm{free}}

\def\n{\mathbf{{n}}}
\def\nh{\widehat{\n}}
\def\odd{\textrm{odd}}
\def\even{\textrm{even}}

\def\M{\mathrm{M}}

\def\C{\mathcal{C}}

\def\m{\bm{\mathrm{m}}}
\def\i{\mathbf{1}}
\def\f{\textnormal f}

\def\bG{{\b G}}
\def\bL{\b\L}
\def\Gd{G^\lozenge}

\title{The Ising magnetisation field \\ and the Gaussian free field}
\begin{document}

\author{
Tom{\'a}s Alcalde L{\'o}pez$^*$, Lorca Heeney$^*$, and Marcin Lis\thanks{ Technische Universität Wien.
	
Emails: \texttt{\{tomas.alcalde, lorca.heeney-brockett, marcin.lis\}@tuwien.ac.at}
}
}
	 
\maketitle
\begin{abstract}
	We construct a natural coupling between the continuum Gaussian free field (GFF) and the critical Ising magnetisation field (IMF) in a planar domain. In fact, we show that two independent IMFs with $+$ boundary conditions and two independent IMFs with free boundary conditions are a deterministic function of a single instance of the
	GFF together with a sequence of independent coin flips. This construction should be seen as an extension of the bosonisation phenomenon, and to the best of our knowledge its existence has not been predicted before.
	
	We arrive at our main result in the continuum by studying novel discrete structures. Our starting point is a coupling resembling the Edwards--Sokal coupling between the Ising model and the Fortuin--Kasteleyn random cluster model, though with role of the latter played by a different percolation model obtained from the double random current model. By taking a scaling limit of the coupling at criticality, we obtain a continuum Edwards--Sokal-like representation of the IMFs in terms of certain two-valued sets of the GFF introduced by Aru, Sep\'ulveda and Werner.
\end{abstract}

\tableofcontents

\section{Introduction}\label{sec:intro}

In this work we uncover an intrinsic relationship between two planar random fields that play central roles in random conformal geometry -- the continuum \emph{Gaussian free field} (GFF) and the \emph{Ising magnetisation field} (IMF).
The former is a universal object appearing as the (in many cases still conjectural) scaling limit of discrete height function models in their delocalized phase~\cite{Kenyon,BPR,GMT,BLR,DRC}.
In physics it is a model of a free (non-interacting) \emph{bosonic field theory}, and in mathematics it is the foundation on which numerous theories are built, e.g.~imaginary geometry~\cite{MilShe,MilShe3,MilShe2}, Gaussian multiplicative chaos~\cite{RhoVar,Ber,JSW,LRV} and Liouville quantum gravity~\cite{RhoVarIcm,DupShe}.
The IMF on the other hand arises as the scaling limit of the discrete magnetisation field of the celebrated Ising model at its critical point~\cite{mag-field,CHI1,CGN2}.

To be more precise, the GFF $\GFF$ with zero boundary conditions in a domain $D \subset \mathbb C$ is a random Gaussian field whose covariance kernel is the Green's function $G_D$ of the Brownian motion in $D$ killed upon hitting the boundary~$\partial D$, and as such is a conformally invariant object (see~\cite{BerPow} for an exhaustive account on the GFF).
Since $G_D$ diverges on the diagonal, the field is too rough to be made sense of as a random function and must be considered only as a random Schwartz distribution acting on test functions:
for smooth functions $f,g$ with support in $D$, the values of $\GFF$ tested against $f$ and $g$, denoted by $(\GFF,f)$ and $(\GFF,g)$ respectively, are centered Gaussians with covariance 
\[
\mathbb E[ (\GFF,f) (\GFF,g)]= \int_D \int_D G_D(x,y)f(x)g(y)\, dx\, dy.
\]
In beautiful works of Aru, Sep\'ulveda and Werner~\cite{ASW,AruSep}, that will be crucial for our considerations, it was shown that the GFF has a rich structure of built-in 
geometric objects called \emph{two-valued sets}. For $a,b>0$ and $a+b\geq2\lambda$, where $\lambda=\pi/2$
is (under the normalisation we use in this article) the \emph{height gap} of the GFF defined in~\cite{SS}, their boundaries $\mathcal L_{-a,b}$ are countable collections of non-crossing fractal-like loops in $D$, which may be heuristically thought of as the
{level lines} of the field at values $-a$ and~$b$, despite the fact that the field itself is not a function. 

The main novelty of our work is to show that two-valued sets with particular values of $a$ and~$b$ give rise to the IMF through a geometric representation, decomposing the field into a signed sum of the ``area'' measures of the two-valued sets, where the signs are given by independent coin tosses. In fact, we are able to couple \emph{four} (two pairs of two independent) IMFs with a single instance of a GFF. 
As a direct consequence the IMFs are measurable functions of the GFF and the signs. Moreover, we prove that two-valued sets with different values of $a$ and $b$ define a one-parameter family of continuum fields associated with the GFF, and conjecture them to be the continuum magnetisation fields of critical \emph{Ashkin--Teller} models. 

We achieve our main result by taking the scaling limit of a \emph{new} representation of the Ising model. This representation is analogous to the Edwards--Sokal coupling \cite{FK, EdwSok} between the Ising model and the FK-Ising model (also referred to as the Fortuin-Kasteleyn random cluster model with cluster weight $q=2$), but the role of the latter is played by a different percolation model defined in terms of the double random current (DRC) model. This is a classical graphical representation of the (products of) Ising correlation functions introduced by Griffiths, Hurst and Sherman~\cite{GHS}, and developed by Aizenman, Duminil-Copin and others
with great success in the study of both planar and higher-dimensional Ising (and related) models~\cite{Aiz82,AizBarFer,ADCS,DCT,ADTW,AizDC,DRC}.
It was recently shown by Duminil-Copin, Qian and the third author \cite{DRC, DRC2} that the scaling limit of a critical DRC is given by a specific iteration of two-valued sets. 
Our results build on this and may be seen as an extension, which adds the GFF into the picture, of the link developed by Camia, Garban and Newman~\cite{mag-field, CME} between the IMF and planar random geometry beyond the conformal loop ensembles with dual parameters $\kappa = 3$ and $\kappa = 16/3$. 

We highlight that our representation of the IMF in terms of the GFF is \emph{not} a local one like e.g.~the \emph{vertex operators} in the free-boson conformal field theory of the GFF.
Indeed, the \emph{bosonisation} framework in physics describes the Ising magnetisation field as a \emph{twist field} rather than a vertex operator of the GFF~\cite{ZubItz,MatLie,DVV}.
These two classes of fields belong to two different sectors in the \emph{orbifold CFT} description of bosonisation -- the twisted and untwisted one correspondingly. 
Unlike the vertex operators, the twist fields are not ``local functions'' of the GFF, and their construction proceeds via a change in the topology of the domain on which the GFF lives by considering a double cover branching around the points of insertion of the twist fields~\cite{DVV}.
In discrete terms, twist fields are therefore more like disorder operators in the Kadanoff-Ceva sense~\cite{KadCev} that change the state-space of the discrete model, whereas the 
vertex operators correspond to the true random degrees of freedom of the model.
What we achieve in this work is a natural probabilistic coupling of the twist fields in the {orbifold CFT}, realised as true random distributions, together with the GFF itself and its vertex operators (that include the scaling limit of the critical XOR-Ising model~\cite{JSW, XOR-exc}, to be defined later).

During the preparation of this article, we learned that the statement of our main Theorem~\ref{thm:main-decomp} 
was also conjectured by Aru and Lupu who provided evidence for it at the level of matching two-point functions. 
They build on the twist field description of the magnetisation field (see Section~\ref{sec:further}) and the recent interpretation of such fields through topological events in the GFF~\cite{lupu-twist}.
 We refer to their upcoming work \cite{AL26} for details.
 \bigskip

\subsection{Main results}
\def\free{\textnormal{free}}

Let us start by recalling the definition of the Ising model, whose history goes back more than a century to the foundational work of Lenz~\cite{Lenz}. Let $G=(V, E)$ be a finite graph. The Ising model with coupling constant $J\geq0$ and free boundary conditions, is a probability measure on spin configurations $\sigma\in \{\pm1\}^{V}$ given by
\begin{align}
	\mathbb P_{G,J} (\sigma)\propto \prod_{uv\in E} \exp\big(J \sigma_u \sigma_v\big).
\end{align}
The model with $+$ boundary conditions is defined by distinguishing a particular subset of the vertices on which $\sigma$ is required to be $+1$.

To define the IMF in a simply connected domain $D\subset \CC$, one starts by considering a sequence of discrete domain approximations $D_\delta \subset \delta \mathbb Z^2$. Here, we abuse notation and write $\Dd$ also for the vertex-set of these graphs. 
We consider the Ising model $\sigma_\de$ on $\Dd$ at the critical \cite{onsager} value of the coupling constants
\begin{equation}\label{eq:critical-Ising}
	J_c = \frac{1}{2} \log(1+\sqrt{2}),
\end{equation}
with either free boundary conditions or $+$ boundary conditions imposed on the boundary vertices of $\Dd$.
Each spin configuration $\sigma=\sigma_\de\in \{\pm1\}^{\Dd}$ can be thought of as a field (generalised function) acting on test functions $f: D\to \mathbb R$ by defining 
\begin{align} \label{eq:defgenfun}
	(\sigma, f) := \delta^2 \sum_{v\in\Dd} f(v)\sigma_v,
\end{align} 
where the vertices of $\Dd$ are identified with the corresponding points in the plane. The breakthrough result of Chelkak, Hongler and Izyruov \cite{CHI1} showed that the scaling limit $\de\to0$ of spin correlation functions, under free or $+$ boundary conditions, exists and is conformally covariant. In turn, it was shown in work of Camia, Garban and Newman \cite{mag-field} that this readily implies that the limit
\begin{align} \label{eq:IMFdef}
	\lim_{\de\to0}\delta^{-1/8} \sigma_\de =: \IMF_D
\end{align}
exist in law, when considered in proper spaces of generalised functions (and under very mild regularity assumptions on $\partial D$). The limit above is taken as the definition of the IMF.

We highlight that the methods in this paper do not rely on the existence of the scaling limit of the discrete IMF, but rather provide a new proof, which solely requires the convergence of the one-point function of the Ising model. This was already the case in the geometric construction of \cite[Section 2]{mag-field}, and indeed our approach is similar in spirit, yet with some key differences we discuss in-depth in Section \ref{sec:outline}.

\noindent\textbf{Statement of the continuum couplings.} We start by stating the existence of a coupling between one Gaussian free field (GFF) and four critical Ising magnetisation fields (IMFs). It is important to note that there is no canonical coupling of all four IMFs at once: there is always an arbitrary choice in how the signs for the $+$ and free boundary condition fields depend on each other --- see Remark \ref{rem:signs}.
\medskip

\begin{theorem}[Four IMFs from one GFF and independent coin tosses] \label{thm:main-coupling} 
Let $D \subset \mathbb C$ be a Jordan domain. Let $\GFF$ be a GFF in $D$ with zero boundary conditions,
and let $ \xi=(\xi_k)_{k\geq 0}$ be i.i.d.~symmetric $\{\pm1\}$-valued random variables. 
Then, there exist (explicit) deterministic measurable functions $\IMF^+$, $\tilde \IMF^+$, $\IMF^\textnormal{f}$, $\tilde \IMF^\textnormal{f}$ such that
\begin{itemize}
\item $\IMF^+(\GFF,\xi)$ and $\tilde\IMF^+(\GFF,\xi)$ have the law of two \textbf{independent} IMFs with $+$ boundary conditions,
\item  $\IMF^\textnormal{f}(\GFF,\xi)$ and $\tilde\IMF^\textnormal{f}(\GFF,\xi)$ have the law of two \textbf{independent} IMFs with free boundary conditions.
\end{itemize}
\end{theorem}

\def\C{\mathcal{C}}
\def\La{\mathcal{L}}

\noindent As already introduced, the explicit functions in Theorem \ref{thm:main-coupling} are given by sums of signed measures, each of which is a measurable function of the GFF. This is precisely the content of the next statement. For the sake of exposition, we explain how to sample the measures and their supports thereafter.

\begin{theorem}[Representation of IMFs via two-valued sets of the GFF]\label{thm:main-decomp} 
Let $D \subset \mathbb C$, $\GFF$ and $(\xi^+)_{k\ge1}$, $(\xi^{\textnormal{f}})_{k\ge1}$ be as above. There exist signs $(\tau_k^+)_{k\geq1}$, $(\tau_k^{\textnormal{f}})_{k\geq1}$ and measures $(\mu_k^+)_{k\geq0}$, $(\mu^{\textnormal {f}}_k)_{k\geq1}$ such that
	\begin{equation}\label{eq:IMF}
		\IMF^+ = \mu_0^+ + \sum_{k=1}^\infty\xi^+_k\mu_k^+,\quad
		\tilde\IMF^+ = \mu_0^+ + \sum_{k=1}^\infty\xi^+_k\tau_k^+\mu_k^+,\quad
		\IMF^{\textnormal{f}} =\sum_{k=1}^
		\infty\xi_k^{\textnormal{f}}\mu_k^{\textnormal{f}},\quad
		\tilde\IMF^{\textnormal{f}} = \sum_{k=1}^\infty\xi_k^{\textnormal{f}}\tau_k^{\textnormal{f}}\mu_k^{\textnormal{f}}, 
	\end{equation}
	are the IMFs in Theorem \ref{thm:main-coupling}. Moreover,
	\begin{itemize}
		\item The convergence of the sums holds a.s. in  the Sobolev space $H^s_{\textrm{loc}}(D)$ for any $s<-1$.
		\item The ordering of the sum is by decreasing diameter of the supports $(\C_k^+)_{k\geq0}$, $(\C^{\textnormal {f}}_k)_{k\geq1}$ of the respective measures.
		\item The measures are a measurable function of their supports, which in turn are a measurable function of $\GFF$.
		\item The signs $(\tau_k^+)_{k\geq1}$, $(\tau_k^{\textnormal{f}})_{k\geq1}$ are measurable functions of $\GFF$.
	\end{itemize} 
\end{theorem}

\begin{remark}\label{rem:signs}
	In the statement of Theorem \ref{thm:main-decomp}, we have chosen to use different coin tosses for different boundary conditions. This is an arbitrary choice, albeit natural, and the statement remains valid also when choosing e.g. the same coin tosses.
\end{remark}

\begin{remark}
	Decompositions of the form \eqref{eq:IMF} are known for fields other than the IMF. For instance, recent work of Aru, Lupu and Sep{\'u}lveda showed the GFF itself admits a similar decomposition, and further proved it arises as the scaling limit of the decomposition of the metric graph GFF  \cite{GFF-excur}. Similarly, in the work of the first author and Sep{\'u}lveda  \cite{XOR-exc}, it is shown that the real and imaginary parts of the complex multiplicative chaos of the GFF can be decomposed analogously. In fact, these fields are the (conjectural for the non-interacting case) scaling limit of the Ashkin-Teller polarisation field, see Section \ref{sec:conj}. 
\end{remark}

Let us now explain how to obtain the supports, which we refer to as (continuum) \emph{clusters}, as measurable functions of $\GFF$. The description we give now is a slightly simplified version of the iteration, where we only care about the law of the clusters as closed sets but do not make reference to the boundary values of $\GFF$ on them. We refer to Section~\ref{sec:proofs} and Figure \ref{fig:full-iter} for the full iteration, along with the definition of the signs $(\tau_k^+)_{k\geq1}$, $(\tau_k^{\textnormal{f}})_{k\geq1}$.

We denote by $\A_{-a,b}\equiv\A_{-a,b}(\GFF)$ the two-valued set \cite{ASW, AruSep} with boundary values $\{-a, b\}$ of a GFF $\GFF$ with zero boundary conditions in $D$. Recall that these are closed subsets of $\bar D$ connected to $\partial D$ such that the boundaries of the connected components of $D\setminus\A_{-a,b}$ define a countable union $\La_{-a,b}$ of loops (i.e. closed simple curves) in $D$. Heuristically, these loops should be thought of as the $\{-a,b\}$--level lines of the GFF and they are the two-dimensional analogue of the exit time of a standard Brownian motion from the interval $[-a,b]$. To be more precise, each $\ell\in\La_{-a,b}$ can be associated a label $c(\ell)\in\{-a,b\}$ such that the following strong Markov property holds: conditionally on $\ell$, the restriction of $\GFF$ to the domain $O(\ell)$ encircled by $\ell$ is given by 
\[	
h^{\ell} + c(\ell),
\]
where the law of $h^\ell$ is that of GFF \emph{with zero boundary conditions} in $O(\ell)$.

The existence and uniqueness of two-valued sets was established in \cite[Proposition 2]{ASW}, provided the boundary values are such that $a, b\geq0$ and $a+ b\geq2\lambda$, where we recall $\lambda=\pi/2$ is the height gap of the GFF \cite{SS, ASW}. Moreover, it was proved in \cite[Theorem 4.1]{AruSep} that the loops in $\La_{-a,b}$ touch each other if and only if $a+b<4\lambda$. Of special interest to us is the collection $\La_{-2\lambda, 2\lambda}$, which is known to have the same law  \cite[Proposition 1]{ASW} as the conformal loop ensemble $\CLE_4$ with parameter $\kappa=4$ defined in~\cite{SheTree,SheWer}. Equivalently, the two-valued set $\A_{-2\lambda, 2\lambda}$ has the same law as the \emph{carpet} of the $\CLE_4$ (i.e. the set of points not surrounded by a loop of the $\CLE_4$, or the closure of the union of all its loops).

Each cluster in the decompositions of Theorem \ref{thm:main-decomp} will be given by the two-valued set $\A_{-2\lambda, 2\lambda}(h^\ell)$ of a GFF in the Jordan domain defined by some loop $\ell$ (which in the case of $+$ boundary conditions may be the boundary $\partial D$ of the initial domain $D$), see Figure \ref{fig:cluster}. 
The exploration of the clusters goes as follows:

\begin{figure}[t]
	\centering
	\includegraphics[width=0.30\linewidth]{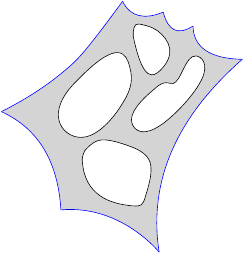}
	\caption{The two-valued set $\A_{-2\lambda, 2\lambda}$, shaded in grey, of a GFF with zero boundary conditions in the domain encircled by some loop, coloured in blue. Every cluster in the decompositions of Theorem \ref{thm:main-decomp} is of this form.
		\label{fig:cluster}}
\end{figure}

\paragraph{$\mathbf{(+)}$ For $+$ boundary conditions (Figure \ref{fig:decomp-plus}):}\hypertarget{pos}{}
\begin{enumerate}[align = left, labelwidth=\parindent, labelsep = 0pt]
	\item[(1.$+$)]\ The boundary cluster is $\C_0^+=\A_{-2\lambda, 2\lambda}(\GFF)$.
	\item[(2.$+$)]\ Let $\gamma\in\La_{-2\lambda, 2\lambda}(\GFF)$. Every loop $\ell\in\La_{-2\lambda, (2\sqrt{2}-2)\lambda}(\GFF^\gamma)$ is given a cluster $\C^+_{\ell}=\A_{-2\lambda, 2\lambda}(\GFF^\ell)$. 
	\item[(3.$+$)]\  Iteratively, let $\C^+_\ell=\A_{-2\lambda, 2\lambda}(\GFF^\ell)$ for some loop $\ell$. Apply (2.$+$) to every $\gamma\in\mathcal L_{-2\lambda, 2\lambda}(\GFF^\ell)$.
\end{enumerate}
Finally, reorder $\{ \C^+_\ell \}$ according to decreasing diameter to get $(\C^+_k)_{k\ge1}$.

\paragraph{$(\text{free})$ For free boundary conditions (Figure \ref{fig:decomp-free}):}\hypertarget{free}{}
\begin{enumerate}[align = left, labelwidth=\parindent, labelsep = 0pt]	
	\item[(1.\textup{f})]\ Every loop $\ell \in \La_{-\sqrt{2}\lambda, \sqrt{2}\lambda}(\GFF)$ is given a cluster $\C^{\textrm{f}}_\ell=\A_{-2\lambda, 2\lambda}(\GFF^\ell)$.
	\item[(2.\textup{f})]\ Iterate as in (3.$+$).
\end{enumerate}
Finally, reorder $\{ \C^{\textnormal{f}}_\ell \}$ according to decreasing diameter to get $(\C^\textnormal{f})_{k\ge1}$.	
\bigskip

\begin{figure}[h]
	\begin{subfigure}{.45\textwidth}
		\centering
		\includegraphics[width=.80\linewidth]{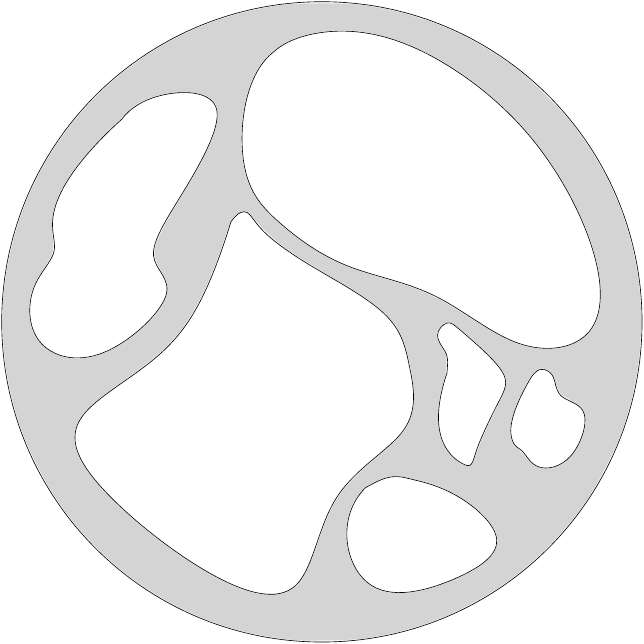}
	\end{subfigure}
	\begin{subfigure}{.45\textwidth}
		\centering
		\includegraphics[width=.80\linewidth]{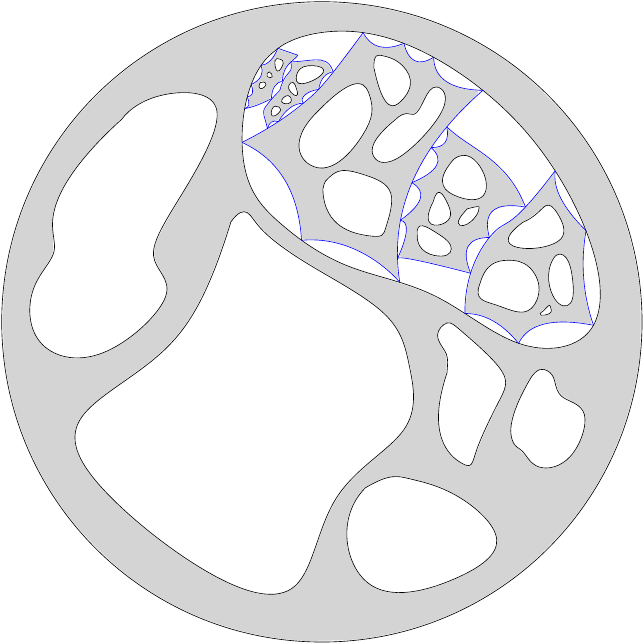}
	\end{subfigure}%
	\caption{Some of the clusters under $+$ boundary conditions. The left-hand depicts the boundary cluster $\C_0^+=\A_{-2\lambda, 2\lambda}(\GFF)$, shaded in grey. On the right-hand side, step (2.$+$) has been additionally performed inside \emph{only} one of the loops $\gamma\in\La_{-2\lambda, 2\lambda}(\GFF)$. Every  $\ell\in\La_{-2\lambda, (2\sqrt{2}-2)\lambda}(\GFF^\gamma)$, coloured in blue, is associated a unique cluster, shaded in grey, which has $\ell$ as its outer boundary.
		\label{fig:decomp-plus}}
\end{figure}

\begin{figure}[h]
	\centering
	\includegraphics[width=0.35\linewidth]{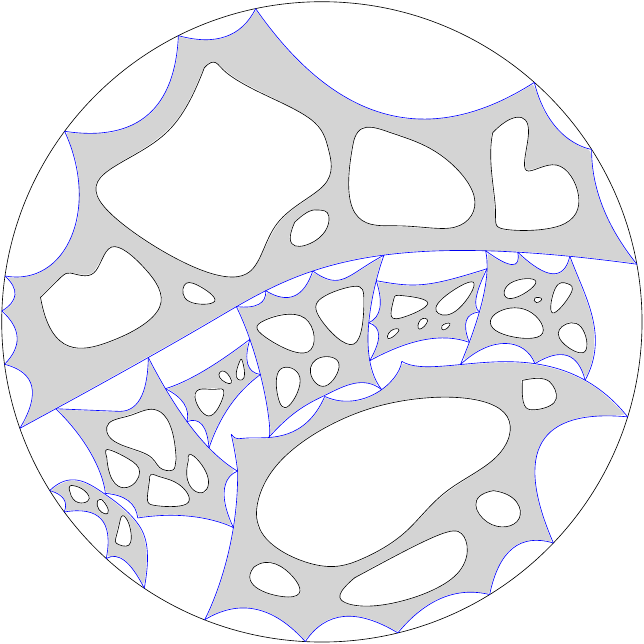}
	\caption{Some of the clusters under free boundary conditions. Only step (1.f) is depicted. Every loop $\ell\in\La_{-\sqrt{2}\lambda, \sqrt{2}\lambda}$, coloured blue, is associated a cluster, shaded in grey. 
		\label{fig:decomp-free}}
\end{figure}
\noindent The clusters $(\C_k^+)_{k\geq0}$ are pairwise disjoint, and similarly for the clusters $(\C^{\textnormal {f}}_k)_{k\geq1}$, but they may overlap from one collection to another.
\bigskip

\begin{remark}
	Note that in the iteration \hyperlink{pos}{$(+)$}, once the boundary cluster $\C_0^+$ has been obtained the iteration proceeds exactly as in \hyperlink{free}{$(\text{free})$} (after a constant shift of the values) within each loop of $\La_{-2\lambda,2\lambda}(\GFF)$. In particular, the $+$ boundary IMF $\Phi^+$ restricted to the interior of a loop $\ell \in \La_{-2\lambda,2\lambda}$ has the law of the free boundary IMF in this domain. Unlike the corresponding statement for the FK-Ising decomposition of the IMF, this spatial Markov property is emergent in the scaling limit and is not present in the discrete coupling.
\end{remark}

Our final result in the continuum establishes that the measures $(\mu_k^+)_{k\geq0}$ and $(\mu^{\textnormal {f}}_k)_{k\geq1}$ should be seen as the ``area'' measures of their (fractal) supports. 
Namely, we obtain a very direct definition of these measures in terms of box-counting, and use {\cite[Theorem 1.2]{CLE-meas}} to identify them as the unique conformally covariant measures in the carpet of $\CLE_4$. We stress that the existence of the limit in \eqref{eq:byboxcount} is already non-trivial, and will come as a consequence of some vital discrete inputs. 
\def\CR{\mathrm{CR}}

\begin{theorem}[Identification of the measures]\label{thm:unique-meas}
	Let $\C$ be any cluster in $(\C_k^+)_{k\geq0}$ or $(\C^{\textnormal {f}}_k)_{k\geq1}$, corresponding to the carpet of a $\CLE_4$ in some (random) Jordan domain $D$. Let $K\subset D$ be any compact subset and let $N_\eps(K\cap \C)$ be the number of boxes in $\eps\Z^2$ that intersect $K\cap \C$. Then, 
	\begin{equation}\label{eq:byboxcount}
		\mu_\C[K] = \lim_{\eps\to0} \beta\eps^{2-1/8}N_\eps(K\cap\C),
	\end{equation}
	where the limit exists in $L^2(\P)$ and $\beta\in(0,\infty)$ is a normalisation\footnote{This constant is such that $\E[\mu_\C[K]]=\mathfrak{C}\ 2^{1/4}\int_K\CR(z, D)^{-1/8}dz$, where $\mathfrak{C}=2^{5/48}e^{3/2\zeta'(-1)}$ is a lattice-dependent constant \cite{CHI1,CHI2}.} constant.
	
	\noindent Moreover, the measure is, up to a multiplicative constant, uniquely characterised by the Miller-Schoug axioms with exponent $d=\textnormal{dim}_{H}(\C)=2-1/8$. 
\end{theorem}

\noindent As explained in {\cite[Remark 1.3]{CLE-meas}}, it follows that if the $(2-1/8)$-dimensional Minkowski content of the carpet of $\CLE_4$ exists and has locally finite expectation, then it is a constant multiple of the measures in Theorem \ref{thm:main-decomp}.

\noindent\textbf{Statement of the discrete couplings and convergence.} While all the statements above concern only objects in the continuum, their proofs almost exclusively rely on their discrete counterparts. The heart of the coupling is the observation in Theorem \ref{thm:main-ES} of a new representation of the Ising model, of which we provide two different proofs in Section \ref{sec:defandprop}. 

Before stating it precisely, let us introduce some notation. Given a graph $G$, we write $V(G)$ and $E(G)$ for the sets of its vertices and edges respectively. For a subset $S \subset V(G)$, we denote by $E(S)$ the edges with both endpoints in $S$. When $G$ is a finite connected planar graph embedded in the plane, we let $G^*$ be its (strong) dual graph and $G^\dagger$ its \emph{weak} dual graph, i.e.~its dual graph with the vertex corresponding to the outer face removed. 
We also denote by $\partial G$ the set of vertices of $G$ lying on its unbounded face.

\begin{definition}[Coin tosses] \label{def:coins}
For any percolation configuration $\w \subset E(G)$, we denote by
\[
	\coin(\w) 
\]
the assignment of independent symmetric $\{\pm 1\}$-valued spins to each cluster (i.e. connected component) of $\w$. 
This gives rise to a spin configuration on $V(G)$ given by the sign of the cluster to which the vertex belongs. We write $\coin^+(\w)$ for the same procedure, with the exception that each cluster intersecting $\bG$ is always assigned the spin $+1$.
\end{definition}

\begin{definition}
	Given $\eta \subset E(G)$, we let $\eta^* \subset E(G^*)$ be its \emph{dual complement} defined by
\[
e\in \eta \qquad \Longleftrightarrow \qquad e^\dual \notin \eta^*,
\]
where $e^\dual$ is the dual edge of $e$. Note that $\eta \supset E(\bG)$ if and only if $\eta^* \subset E(G^\dual)$.
\end{definition}


\begin{theorem}[Alternative Edwards--Sokal coupling]\label{thm:main-ES}
	Let $G$ be a finite connected planar graph. Let $\widehat{\n}^\dual$ the trace of the sourceless double random current $\n^\dual$ on $G^*$ with free boundary conditions and with coupling constants 
	$J^\dual: E(G^\dagger)\to [0,\infty)$, as in Definition~\ref{def:drc}. Define $\w$ to be its \textbf{dual complement}
	\begin{align} \label{eq:dualcompcurr}
	\w = (\widehat{\n}^\dual)^*.
	\end{align}
	Then, $\coin(\w)$ has the law of the Ising model with free boundary conditions on $G$ 
	and coupling constants $J: E(G)\to [0,\infty)$ related to $J^\dual$ by the Kramers-Wannier duality \eqref{eq:Jdualdef}.

	\noindent Similarly, let $\nh^\dual$ be the trace of the sourceless double random current $\n^\dual$ now on the weak dual $G^\dual$ with free boundary conditions and coupling constants $J^\dual$. Again, define $\w = \nh^\dual$ to be its dual complement. Then, $\coin^+(\w)$ has the law of the Ising model on G with $+$ boundary conditions imposed on $\bG$ and the same coupling constants $J$ as above.
\end{theorem}

\noindent In fact, one does not require a planar setup in order to explicitly define the percolation model $\omega$ satisfying the above \emph{Edwards--Sokal property}, see Definition~\ref{def:one}. Moreover, this alternative and more general definition will prove to be more tractable when deriving all the necessary properties of the model, including FKG-type inequalities and RSW statements.

Throughout the paper we work with an extension of the so-called \emph{master coupling} introduced in \cite{DRC} (see also \cite{spins-perc-height}). Let us briefly describe its main properties, referring to Section~\ref{sec:defandprop} for more details and equivalent constructions. First, one couples a pair 
\begin{equation}\label{eq:DRCs-couple}
	(\n, \n^\dagger)
\end{equation}
of sourceless double random currents with wired boundary conditions in $G$ and free boundary conditions in $G^\dagger$, respectively, such that they do not cross each other (meaning that if a primal edge is present in $\n$, then its dual edge is absent in $\n^\dagger$, and vice versa). We postpone the proper definition of wired boundary conditions for the double random current (Definition \ref{def:drc}), but note that it involves dropping the requirement to be sourceless on the vertices of $\bG$.
Each of the two configurations defines a spin model
	\[
		\tau = \coin^+(\n) \quad \text{and} \quad \tau^\dagger=\coin(\n^\dagger),
	\]
with the law of a XOR-Ising model (i.e. with the law of the product of two independent Ising models). In turn, this defines a height function $H: V(G)\cup V(G^\dual)\to \mathbb Z \cup (1/2 + \mathbb Z)$ by imposing zero boundary conditions and defining its gradient to be
	\begin{align} \label{eq:defhintrod}
		H(u^\dual) - H(u)  = \frac{1}{2}\tau(u)\tau^\dagger(u^\dual),
	\end{align}
for $u\in V(G)$ and $u^\dual \in V(G^\dagger)$. This is the content of the master coupling from~\cite{DRC}. We extend it by taking the dual complements of the currents as in~\eqref{eq:dualcompcurr}, one can define the pair $(\omega, \omega^\dagger)$ and use Theorem \ref{thm:main-ES} to define two Ising models
	\[
		\sigma = \coin^+(\omega) \quad \text{and} \quad \sigma^\dagger=\coin(\omega^\dagger),
	\]
on $G$ with $+$ boundary conditions and on $G^\dual$ with free boundary conditions respectively.
Finally, one can recover two more Ising fields simply by setting
	\[
		\tilde\sigma =\sigma\tau\quad \text{and} \quad \tilde\sigma^\dagger=\sigma^\dagger\tau^\dagger,
	\]
which must be independent of $\sigma$ and $\sigma^\dagger$, respectively. This augmentation of the master coupling with Ising spins is the starting point of our considerations.

The following statement extends the convergence result in \cite{DRC, DRC2} and gives the joint convergence of the discrete models to their continuum counterparts. We do not prove, however, the joint convergence of both XOR-Ising models. This is done as part of the work of the first author with Sepúlveda~\cite{XOR-exc}, identifying them as the real and imaginary parts of the complex multiplicative chaos of the limiting GFF. Note that the GFF here is normalised differently from that in~\cite{DRC, DRC2}, changing the multiplicative constant appearing in the limit.

\begin{theorem}(Joint scaling limit at criticality) \label{thm:joint}
	Let $D\subset\mathbb C$ be a Jordan domain with discrete domain approximation $D_\de\subset\delta\Z^2$. At criticality, as $\de\to0$,
	\[
		\big(h_\de,\ \omega_\de,\ \omega^\dagger_\de,\ \delta^{-1 / 8}\sigma_\de,\ \delta^{-1 / 8} \tilde\sigma_\de,\ \delta^{-1 / 8}\sigma_\de^\dagger,\ \delta^{-1 / 8}\tilde\sigma_\de^\dagger\big) \overset{(d)}{\longrightarrow}\big(\frac{1}{\pi\sqrt{2}}\GFF,\ \omega,\ \omega^\dagger,\ \IMF,\ \tilde\IMF,\ \IMF^\dagger,\ \tilde\IMF^\dagger\big),
	\]
	where the discrete spin configurations $\sigma_\de, \tilde\sigma_\de, \sigma_\de^\dagger, \tilde\sigma_\de^\dagger$ are thought of as random distributions as in~\eqref{eq:defgenfun}, $\GFF$ is a GFF with zero boundary conditions in $D$, the clusters of $\omega$ and $\omega^\dagger$ are as described in \hyperlink{pos}{$(+)$} and \hyperlink{free}{$(\text{free})$}, and the IMFs admit the decompositions in Theorem \ref{thm:main-decomp}. 
	The precise topologies of convergence are given in Appendix \ref{app:spaces}.	
\end{theorem}

\begin{remark}
	We also mention that our methods allow to couple multiple IMFs with different boundary conditions (including Dobrushin boundary conditions) through a \emph{single} instance of the GFF. Moreover, Theorem~\ref{thm:joint} suggests 
	natural couplings (in the spirit of CLE percolations from~\cite{MSW}) between the GFF and (multiple) CLE$_3$ and SLE$_3$, which are the scaling limit of the interfaces between $+1$ and $-1$ spins in the critical Ising model~\cite{SLE3,BenHon}. 
	Both directions will be explored in a future work.
\end{remark}

\begin{figure}
	\centering
	\includegraphics[width=0.58\linewidth]{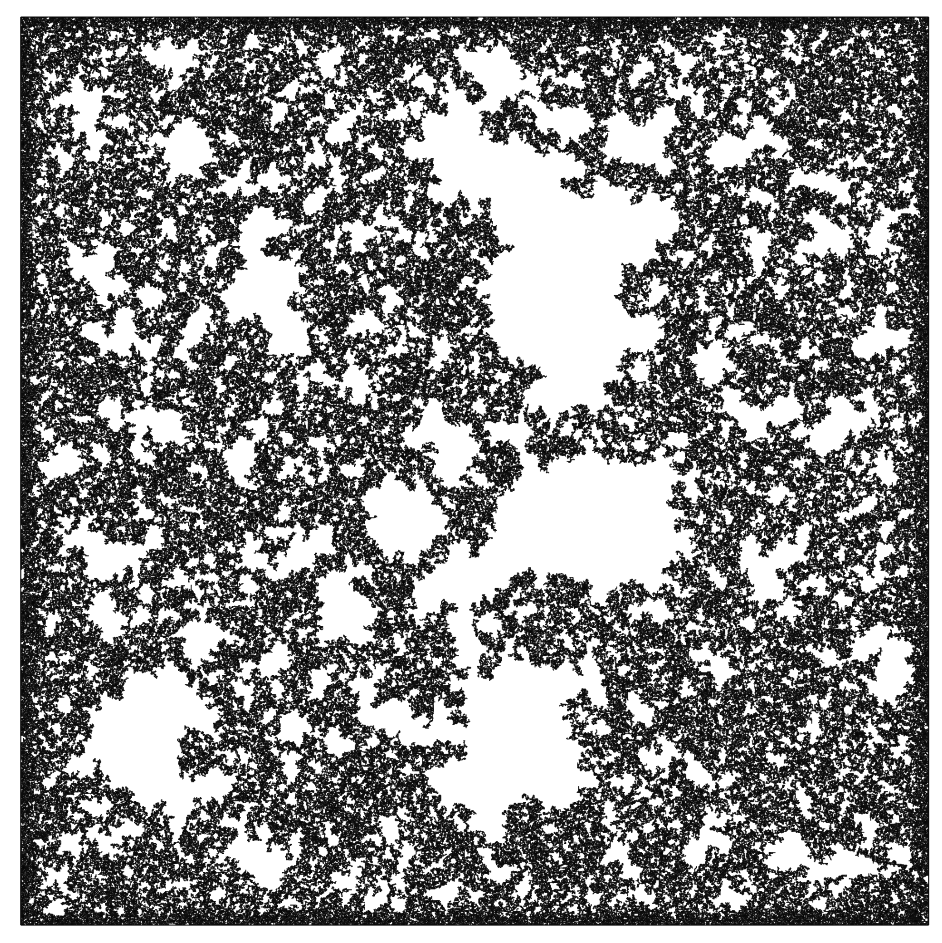}
	\caption{A simulation of the boundary cluster of the percolation $\omega$ on a 1000x1000 grid. The law of this boundary cluster converges to that of the carpet of $\CLE_4$. }
\end{figure}

\subsection{Further discussion and context}\label{sec:further}

The purpose of this section is two-fold. First, we introduce the Ashkin-Teller model and explain how our main results fit into a much broader picture, which motivates all of the conjectures we pose in Section \ref{sec:conj}. Secondly, we review some of the literature around previous approaches to the IMF \cite{mag-field, CME} and rigorous mathematical descriptions of bosonisation.

Let $G=(V, E)$ be a finite graph. The Ashkin--Teller (AT) model~\cite{AT} with coupling constants $J \geq0,\ U\in \mathbb R$, and free/free boundary conditions,
is a probability measure on pairs of spin configurations $(\sigma,\tilde \sigma)\in \{ -1,1\}^{V}\times \{ -1,1\}^{V}  $
given by
\begin{align} \label{eq:ATdef}
	\mathbb P_{G,J,U} (\sigma,\tilde \sigma)\propto \prod_{uv\in E} \exp\big(J (\sigma_u \sigma_v+\tilde \sigma_u\tilde \sigma_v)+U \sigma_u \sigma_v\tilde \sigma_u\tilde \sigma_v\big).
\end{align}
The measure under $+/+$ boundary conditions is defined as for the Ising model, namely by distinguishing a particular subset of the vertices on which $\sigma$ and $\tilde\sigma$ are both required to be $+1$. The single-spin configurations $\sigma$ and $\tilde\sigma$ are referred to as the \emph{magnetisation} spins, while the (coordinate-wise) product-spin configuration 
$\tau=\sigma\tilde \sigma$ is referred to as the \emph{polarisation} spin. On the \emph{free fermion line} $U=0$ the spins $\sigma$ and $\tilde \sigma$ become two independent Ising models with coupling constant $J$, and the polarisation spin is traditionally called the \emph{XOR-Ising model}.

The \emph{critical line}~\cite{Baxter} of the AT model is defined by 
\begin{align} \label{eq:critical}
	\sinh (2J) = e^{-2U}, \qquad J\geq U.
\end{align}
It has been predicted in the physics literature~\cite{DVV,Baxter,NienBook} that the model possesses explicit critical exponents that vary continuously on the critical line, and 
a conformally covariant scaling limit as $\de\to0$. In particular, given a domain $D\subset\CC$ and its discrete domain approximation $\Dd\subset\de\Z^2$, the following limits are predicted to exist and be conformally covariant:
\begin{align} \label{eq:ATexponents}
	\lim_{\delta \to 0}\ \delta^{-1/8}\ \langle \sigma_{x_\delta} \rangle^+_{\Dd,U} =: \langle \sigma_{x} \rangle^+_{D,U} \qquad \textnormal{and} \qquad
	\lim_{\delta \to 0}\ \delta^{-1/(2g)}\ \langle \tau_{x_\delta} \rangle^+_{\Dd,U} =: \langle \tau_{x} \rangle^+_{D,U}, \
\end{align}
where $\langle \cdot \rangle^+_{\Dd,U}$ denotes the expectation with respect to the critical model $\mathbb P_{\Dd,J_c(U),U}^+$
with $+/+$ boundary conditions and $x_\de\in\Dd$ is at distance $O(\de)$ from some fixed $x\in D$. The critical exponent of the polarisation is given by
\begin{align} \label{eq:defg}
	g:=\frac{8}{\pi}\sin^{-1}\left(\frac{1}{2}(1+\exp(4U))^{1/2}\right) \in (4/3, 4].
\end{align}
Interestingly, the critical exponent of the magnetisation field $\sigma$ is constant and equal $1/8$, whereas the polarisation exponent. 
As already discussed, for $U=0$, owing to its strong form of solvability, the convergence above was rigorously proved in recent breakthrough work of Chelkak, Hongler and Izyurov~\cite{CHI1}. Even though the AT model on the critical line is known to be tightly related to the \emph{six-vertex model} which is also exactly solvable~\cite{Baxter} (though in a weaker sense), the statements in~\eqref{eq:ATexponents} outside the critical free fermion point are still beyond the reach of rigorous proofs. 

An important feature of our predictions for the AT model is a conceptual explanation of the value $1/8$ on the whole critical line.  
Namely, we relate it to the fact that the Hausdorff dimension of the two-valued set $\A_{-2\lambda, 2\lambda}$ is equal to $2-1/8$ \cite{SSW,NacWer,dim-TVS}. 
This set appears in our conjectural geometric representation of the continuum AT fields \emph{irrespectively of the point on the critical line}. Moreover, the variation of the predicted exponent $g$ along the critical line corresponds to the fact that the two-valued sets with height gap $2\sqrt{2}$ in \hyperlink{pos}{$(+)$} and \hyperlink{free}{$(\text{free})$} are replaced by two-valued sets with height gap $2\sqrt{g}$, which have dimension $2-1/(2g)$ \cite{dim-TVS}. We refer to Conjectures \ref{conj:AT-currs}-\ref{conj:AT-mag-field}. for the precise statements.

Based on~\eqref{eq:ATexponents}, and viewing the discrete configurations as distributions as in \eqref{eq:defgenfun}, one expects that the limits 
\begin{align} \label{eq:fielddef}
	\lim_{\delta \to 0} \delta^{-1/8} \sigma_\de =: \IMF_D, \qquad \text{and} \qquad \lim_{\delta \to 0} \delta^{-1/(2g)} \tau_\de =: \mathlarger\uptau_D 
\end{align}
exist in law, again when considered in proper spaces of generalised functions. Beyond the case $U=0$, one of our contributions is an explicit conjecture describing the field $\IMF_D$ in terms of the appropriate two-valued sets. The identification of $\mathlarger\uptau_D$ with a vertex operator of the GFF (or \emph{imaginary multiplicative chaos} in the mathematics terminology~\cite{JSW}) follows directly from bosonisation and has been non-rigorously established in the physics literature~\cite{Baxter,NienBook,kadanoff-brown}. We restate this connection in Conjecture \ref{conj:AT-IGMC}, which has been already established for the Ising model~\cite{JSW}. 

In~\cite{mag-field,CamNew, CME} the Edwards--Sokal coupling between the Ising model and the FK-Ising model was employed to study the IMF from a geometric point of view. The coupling states that to sample the Ising model on $G$, one can first sample a configuration $\w^{\textup{FK}}\subset E(G)$ of the FK-Ising model and toss coins on each cluster. The critical FK-Ising model on the square lattice has been extensively studied in the last fifteen years. Thanks to its desirable features like exact solvability and positive association, a number of
breakthroughs were achieved including a complete RSW theory~\cite{RSW}, and proofs of conformally invariant scaling limits~\cite{Smir,FK-CLE}. 
In particular it was shown that the collection of all boundaries of the clusters on $\Dd$ 
 converges to the conformal loop ensemble CLE$_{16/3}$~\cite{FK-CLE} in $D$, which in turn allowed for a fully geometric description of the continuum Ising field~\cite{mag-field,CME}. We note that the dimension of the gasket of CLE$_{16/3}$ is also known to be $2-1/8$~\cite{MSWgasket}, thus the critical exponent of the magnetisation provides a nice conceptual explanation as for why this must equal the dimension of the carpet of $\CLE_4$.

The path to achieve our main results is indeed similar to~\cite{mag-field, CME}, but the study of the dual complement $\omega$ of the DRC differs in many aspects, as will become clear. It may even come as a surprise that two different laws on a percolation model can satisfy the Edwards--Sokal property. However, an observation to the same effect in a general planar setup can be found already in the work of the third author~\cite[Theorem 3.2] {LisT} and was explicitly stated in the master thesis of 
Ivan Hejn\'y~\cite{Hej}. 
In hindsight, our approach is closely related to that in~\cite{AT-6V}, and as mentioned therein it fits the framework of FK-repersentations
of the Ashkin--Teller model introduced by Pfister and Velenik~\cite{PfiVel} (though it falls outside the range of couplings constants considered there). We elaborate further on this point at the beginning of Section \ref{sec:conj}.

One can think of our coupling in the case of planar graphs as an enhancement of the bosonisation of the Ising model originating in the work of Zuber and Itzykson~\cite{ZubItz} 
with mathematical treatments among others in~\cite{dubedat-boson,BouTil,nesting-field}. In the continuum, this theory establishes exact identities between squares of Ising correlation functions (i.e. XOR-Ising correlation functions) and correlations of vertex operators of the GFF. Building on \cite{CHI1, CHI2}, the full picture for the continuum correlations has been completed in \cite{boson-cont}. In the discrete this is mirrored via a combinatorial 
relation between \emph{two independent} copies of the Ising model and a \emph{dimer model} on a \emph{bipartite} decoration of the original graph~\cite{dubedat-boson,BouTil,nesting-field}. 
Such dimer models can be thought of as models of random surfaces via their associated~\emph{height function}, where the XOR-Ising field is given by the parity of said function.
Since the groundbreaking work of Kenyon~\cite{Kenyon}, dimer height functions are known to converge to the GFF in the scaling limit. Intriguingly, the first dimer representation of the Ising model goes back to the work of Fisher~\cite{FisherDim}, who mapped a \emph{single} copy of the Ising model to a dimer model on a decorated graph. 
However, the decoration is \emph{nonbipartite} and as such does not correspond to a height function and hence does not fit in the framework of bosonisation.
Moreover, none of the aforementioned works makes a prediction about how a single Ising model could fit into the coupling between the XOR-Ising and the height function.
This is amended by our contributions. 

In terms of convergence of (dimer) height functions, of particular interest for the purposes of this paper is the convergence of the critical XOR-Ising height function to the GFF~\cite{DRC}, which justifies the name bosonisation used in reference to the combinatorial map between two Ising models and a dimer model. The work of Duminil-Copin and the third author~\cite{nesting-field} already made an implicit step in this direction by jointly coupling the DRC together with the
XOR-Ising model and the associated height function. This was in turn inspired by the \emph{alternating-flow} representation of the DRC model introduced in~\cite{LisT}.
In~\cite{DRC} the coupling from~\cite{nesting-field} was extended to the master coupling (which was briefly sketched in the previous section) of both primal and dual double random currents, XOR-Ising models and height functions. Moreover, it was realised in~\cite{DRC} that the DRC model itself (i.e.~not its dual complement) is coupled with the XOR-Ising model via the Edwards--Sokal procedure.
This is in contrast to the FK-Ising model whose dual complement is also an FK-Ising model, and hence both are coupled with a single Ising model 
An extension to the AT model (including the \emph{AT random currents} defined in~\cite{lis2022boundary}, the AT polarisation fields and the associated height functions) was simultaneously given in~\cite{spins-perc-height}. 

It was shown by Duminil-Copin, Qian and the third author~\cite{DRC,DRC2} that the critical master coupling has a joint scaling limit in which the height function converges to a GFF, and where the outer and inner boundaries of the 
clusters of the DRC converge to certain loops arising from two-valued sets of that GFF -- see Theorem \ref{thm:DRC}. We stress again that the couplings introduced in this paper take the setup of \cite{DRC, DRC2} as a starting point, and proceed by adding the magnetisation fields as explain in Theorem \ref{thm:main-ES}. It is precisely the duality relation in this statement that allows us identify the scaling limit of the clusters of our percolation model by invoking the results of \cite{DRC, DRC2}.

Lastly, we mention another approach to study planar spin systems that goes back to the work of Peierls~\cite{Peierls}, who interpreted a spin configuration by the collection of interfaces that separate the $+1$ and $-1$ spins. 
For the critical Ising model with so-called \emph{Dobrushin} boundary conditions, a single interface connecting two boundary points was shown to converge to Schramm--Loewner evolution SLE$_3$~\cite{SLE3}, whereas for $+$ boundary conditions the full collection of loops converges to CLE$_3$~\cite{BenHon}. 
The scaling limits of the analogous interfaces for the XOR-Ising model have been predicted by Wilson~\cite{Wil}, and should be described by the loops of certain two-valued sets, with different boundary values than those considered until now.
This is supported by the fact that in the discrete bosonisation picture the interfaces of the XOR-Ising model are actually level lines of the discrete height function~\cite{dubedat-boson,BouTil}. 

We recall that the conformal loop ensembles CLE$_\kappa$ for $\kappa\in (8/3,8)$ are countable collections of random fractal non-self crossing loops, which  are uniquely identified by conformal invariance and a type of spatial Markov property~\cite{SheWer}. Similarly, we already know that the loops of a two-valued set possess also a type of Markov property, now with respect to the underlying GFF. Conformal invariance and Markov property should usually be a feature of scaling limits of curves in critical models, and hence these two families yield natural candidates for the limits of the (XOR-)Ising interfaces.

On the other hand, the interfaces of the $\sigma$ and $\tilde \sigma$ configurations in the critical AT model away from the free fermion point do not possess any natural spatial Markov property even in the discrete. Indeed conditioning on the value of one spin along a circuit, still allows for information to flow from the outside of the circuit to the inside 
through the interaction with the other spin. For this reason there is no prediction on the scaling limit of these interfaces. 
We fill this gap by conjecturing (see Conjecture \ref{conj:AT-inter}) that the continuum curves corresponding to critical AT magnetisation interfaces are given by a coin tossing procedure on two valued-sets,
and then taking the interfaces between the two spin values. This is an analogue of the CLE percolation procedure described in~\cite{MSW} and we plan to explore this direction in a future work.

\subsection{Conjectures for the Ashkin--Teller model}\label{sec:conj}

 Based on the recent works of the third author~\cite{LisT,lis2022boundary,spins-perc-height}, it turns out that a completely analogous discrete coupling exists for the general Ashkin--Teller model (see Section \ref{sec:vargen} for the proofs).
Namely, the counterpart of the double random current on $G$ is the \emph{AT current} $\n=(\n_\odd,\n_\even)$ defined in~\cite{lis2022boundary}, whose law is given by
\begin{align} \label{eq:ATcurr}
\mathbb P_{\at}(\nh)\propto 2^{k(\nh)} x^{|\n_\odd|}y^{|\n_\even|},
\end{align}
where $k(\nh)$ is the number of clusters of $\nh = \n_\odd \cup \n_\even$, $\n_\odd$ and $\n_\even$ are disjoint and $\n_\odd$ is constrained to have even degree around every vertex in $V(G)$ (i.e.~$\n_\odd$ is an \emph{even} 
subgraph of $G$).
When compared to~\eqref{eq:ATdef}, the parameters satisfy
\[
x=e^{2U} \sinh(2J), \quad \textnormal{and} \quad y=e^{2U}\cosh(2J)-1.
\]
As was shown in~\cite{LisT} for 
$U=0$, the current $\n$ is the projection of the sourceless DRC (defined as a function from $E(G)$ to $\{0,1,2,\ldots \}$, see Definition \ref{def:drc}) onto the sets of edges where it takes odd and even positive values, resulting in $\n_\odd$ and $\n_\even$ respectively.
With each such current one can associate a height function $h$, called the \emph{nesting field}, which can be defined in the same way as the one for the noninteracting case introduced by Duminil-Copin and the third author in~\cite{nesting-field}. It is an integer-valued function defined on the faces of the graph that changes values by $\pm 1$ between neighbouring faces only when crossing an edge of $\n_\odd$, and moreover the 
increments across two edges are independent for edges in different clusters of $\n$. For a precise definition we refer to~\cite{nesting-field}, and for an alternative but equivalent one to~\eqref{eq:defhintrod} we refer to Section~\ref{sec:heightfunc} and Section~\ref{sec:vargen}.

A coupling between primal and dual AT currents, the corresponding primal and dual nesting fields and the two polarisation fields was given in~\cite{spins-perc-height}.
The extension to include the magnetisation fields follows from~\cite[Proposition 8.1]{AT-6V} for the
self-dual model on the square lattice, and from~\cite[Proposition 13]{lis2022boundary} in full generality of planar graphs and coupling constants.
As a result, all identities in law (in particular various spatial Markov properties that are crucial for the identification of scaling limits) at the discrete level described in this article for the case $U=0$ hold also for the corresponding objects in general AT models. The uncovering of this rich structure leads to new directions in the further study of the AT model described below. The proofs of precise statements can be either found in~\cite{lis2022boundary,spins-perc-height}
or are straightforward generalisations of formulas presented here. 

A particular feature of this coupling that instantiates the phenomenon of bosonisation is that the polarisation field is the complex exponential of the dual nesting field. This will be mirrored in one of the conjectures below.
Moreover, the joint low of the primal and dual height function is known to be the same as that of the height function in the associated \emph{six-vertex model} (see e.g.~\cite{AT-6V,spins-perc-height}).
Together with the physics prediction that the latter
converges to an explicit constant multiple of the Gaussian free field~\cite{Baxter,NienBook,kadanoff-brown}, 
it allows to make the following precise conjectures on the critical behaviour of the AT model. 

The first one is a reiteration of the convergence of the height function and the bosonisation relation for the polarisation field.

 \begin{conj}\label{conj:AT-IGMC} Let $D\subset\mathbb C$ be a Jordan domain with discrete domain approximation $D_\de\subset\delta\Z^2$. At each point of the critical line~\eqref{eq:critical}, as $\de\to0$, 
	the nesting field of the AT current on $D_\de$ under the $+/+$ or free/free boundary conditions converges to 
	\[
	\frac{1}{2\lambda\sqrt{g}}\ \GFF = \frac{1}{\pi\sqrt{g}}\ \GFF,
	\] 
	where $\GFF$ is a GFF with zero boundary conditions in $D$, and $g$ is defined in~\eqref{eq:defg}. 
	Moreover, the polarisation fields under $+/+$ and free/free boundary conditions, respectively, converge to
		\[
			\cos\big(h/\sqrt{g}\big) \quad \text{and} \quad	\sin \big(h/\sqrt{g}\big),
		\]
	corresponding to the \emph{imaginary Gaussian multiplicative chaos} with parameter $\alpha(g)=1/\sqrt{g}\in[1/2, \sqrt{3}/2)\subset[0, \sqrt{2}].$
\end{conj}

The second (and the first novel) conjecture is a direct analog of the main result of~\cite{DRC,DRC2}..
 \begin{conj}\label{conj:AT-currs}
 	The almost verbatim analog of Theorem~\ref{thm:DRC} holds for the Ashkin--Teller currents defined by~\eqref{eq:ATcurr} at each point of the critical line~\eqref{eq:critical}, the only difference being that every explicit occurrence of $\sqrt 2$ has to be replaced with $\sqrt g$ defined in~\eqref{eq:defg}. In particular, the appearance of $\A_{-2\lambda, 2\lambda}\equiv\CLE_4$ in the scaling limit is universal on this line, and the remaining 
	two-valued sets used in the iteration are of the form $$\A_{-2\lambda, (2\sqrt{g}-2)\lambda},\ \A_{(2-2\sqrt{g})\lambda,2\lambda}\ \ \text{or}\ \ \A_{-\sqrt{g}\lambda,\sqrt{g}\lambda}.$$ 
\end{conj}

We reiterate that the convergence of the outermost boundaries of the clusters to a CLE$_4$ is a geometric analog (that actually carries much more information)
of the classical conjecture that the exponent of the magnetisation fields are constant and equal to $1/8$ on the entire critical line (see e.g.~\cite{Baxter}).

\begin{remark}
Recall that at the endpoint $J=U$ of the critical curve we have $g=4$. In this case, all sets in the iteration are actually CLE$_4$.
This is consistent with the fact that the law of the AT becomes that of the 4-state Potts model, and the associated current $\nh$ becomes the critical FK random cluster model with cluster weight $q=4$. The former follows simply from the inspection of the Hamiltonian. The latter can be seen by noticing that the two weights $x$ and $y$ in~\eqref{eq:ATcurr} are both equal to $1$, which in turn implies that when one forgets which edges of $\n$ are in $\n_\odd$ and which in $\n_\even$, an additional combinatorial factor 
of $2^{|\n|+k(\n)}$ appears due to the total number of even subgraphs of $\n$ (see e.g.~\cite[Lemma 4.2]{LisT}). In particular, this yields a total weight proportional to $4^{k(\n)}2^{|\n|}$, which is the critical FK(4) random cluster model weight~\cite{BefDum}. As a consequence, our conjecture is a generalisation of the conjecture of Sheffield~\cite{SheTree} that the outer and inner boundaries of the clusters of the critical FK(4) model converge in the scaling limit to CLE$_4$'s.
\end{remark}

 
The next conjecture is a generalisation of our main results, namely Theorems~\ref{thm:main-coupling} and~\ref{thm:main-decomp}, and predicts a geometric decomposition of the continuum Ashkin-Teller magnetisation fields.

 \begin{conj}\label{conj:AT-mag-field}
	 The statement of Theorem~\ref{thm:main-decomp} under the almost verbatim analogue of the iterations \hyperlink{pos}{$(+)$} and \hyperlink{free}{$(\text{free})$} holds for the Ashkin--Teller magnetisation field defined by~\eqref{eq:fielddef} at each point of the critical line~\eqref{eq:critical}, the only difference being that every explicit occurrence of $\sqrt 2$ has to be again replaced with $\sqrt g$ defined in~\eqref{eq:defg}.
 \end{conj}

\noindent We believe that the methods we apply in this paper are robust, so that Conjecture \ref{conj:AT-mag-field} could be proven once Conjecture \ref{conj:AT-currs} and the convergence of the height function in Conjecture \ref{conj:AT-IGMC} are known.
Part of the statement above is that the relevant fields are well defined in the continuum, which we in fact prove.

 \begin{prop}\label{prop:AT-mag-field}
 	The analogous series to those in \eqref{eq:IMF}, but with $\sqrt 2$ replaced with $\sqrt g$ in the definition of the relevant two-valued sets, converge a.s. and in $L^2$ for $g\neq4$ (or, equivalently, $J\neq U$).
 \end{prop}
\medskip
 
We highlight that Conjecture \ref{conj:AT-mag-field} was independently discovered by Aru and Lupu, and will appear in their upcoming work \cite{AL26}, where they are able to compute the two-point function of the field. This suggests that a potential approach to rigorously establishing the convergence and conformal covariance of the spin-spin correlation function of the AT magnetisation field is to prove Conjectures \ref{conj:AT-IGMC}-\ref{conj:AT-currs} (which, again, we believe to be sufficient to establish Conjecture~\ref{conj:AT-mag-field}) and use their result thereafter.

As a final conjecture, we describe a continuum percolation procedure, in the spirit of \cite{MSW}, that should describe the scaling limit of the interfaces of the AT magnetisation field. This conjecture is simply the continuum analogue, given Conjecture \ref{conj:AT-currs}, of the discrete representation of the field in terms of the dual complement of the AT current. Even for $U=0$, this suggests a very interesting connection between the $\CLE_3$ and the two-valued sets of the GFF.

\begin{conj}\label{conj:AT-inter}
	Consider the clusters of the decomposition described in Conjecture \ref{conj:AT-IGMC} under $+/+$ boundary conditions. Toss a fair coin for each cluster to assign a value $+1$ or $-1$, except for the boundary cluster which is assigned $+1$. If two clusters with the value $+1$ (resp. $-1$) touch each other, we say that they belong to the same $\oplus$-cluster (resp. $\ominus$-cluster). Then, the law of the outermost interfaces of the Ashkin-Teller magnetisation field is given by the outermost outer boundaries of the $\ominus$-clusters. The full collection of interfaces is defined iteratively.
\end{conj}

\subsection{About the proof}
Let us now discuss all the inputs needed in the proofs of our main results above. We do so in the hope of identifying the minimal set of pre-existing results needed to generalise our arguments to Ashkin-Teller models.

The crucial first input in our proof is the identification of the continuum clusters as found in \cite{DRC, DRC2}. Simply put, the clusters of the DRC correspond to the ``complements'' of the clusters described in Theorem \ref{thm:main-decomp}. Namely, they are given by the two-valued sets with height gap $2\sqrt{2}\lambda$. We state the result directly in terms of clusters (rather than in terms of their outer and inner boundaries as found in \cite{DRC, DRC2}), see Appendix \ref{app:conv-ext}.

\begin{theorem}[Convergence of critical double random currents \cite{DRC, DRC2}]\label{thm:DRC}
	Let $D\subset\mathbb C$ be a Jordan domain with discrete domain approximation $D_\de\subset\delta\Z^2$. {Let $(\C_k^\de)_{k\geq0}$ and $((\C_{k}^\de)^\dagger)_{k\geq1}$ be the clusters of a pair of double random currents with wired boundary conditions in $\Dd$ and free boundary conditions in $\Dd^\dagger$, respectively, coupled as in \eqref{eq:DRCs-couple}. Let $H_\de$ be the associated height function defined in \eqref{eq:defhintrod}.} At criticality, as $\de\to0$, 
	\[
	\big(H_\de, (\C_k^\de)_{k\geq0}, ((\C_{k}^\de)^\dagger)_{k\geq1}\big) \overset{(d)}{\longrightarrow}\big(\frac{1}{\pi\sqrt{2}}\GFF, (\C_k)_{k\geq0}, (\C_{k}^\dagger)_{k\geq1}\big),
	\]
	{where $h$ is a Gaussian free field with zero boundary conditions and $(\C_k)_{k\geq0}, (\C_{k}^\dagger)_{k\geq1}$ are obtained reodering by decreasing size of diameter the following iterations:}
	
	\noindent For wired boundary conditions:
	\begin{enumerate}[align = left, labelwidth=\parindent, labelsep = 0pt]	
		\item[(1.\textup{w})]\ The boundary cluster is $\C_0=\A_{-\sqrt{2}\lambda, \sqrt{2}\lambda}(h)$.
		\item[(2.\textup{w})]\ Let $\ell\in\La_{-\sqrt{2}\lambda, \sqrt{2}\lambda}(h)$. Every loop $\gamma\in\La_{-2\lambda, 2\lambda}(\GFF^\ell)$ is given a cluster $\C_\gamma=\A_{-2\lambda, (2\sqrt{2}-2)\lambda}(\GFF^\gamma)$.
		\item[(3.\textup{w})]\  Iteratively, let $\C_\gamma=\A_{-2\lambda, (2\sqrt{2}-2)\lambda}(\GFF^\gamma)$ for some loop $\gamma$. Apply (2.\textup{w)}, but now to every $\ell\in\La_{-2\lambda, (2\sqrt{2}-2)\lambda}(\GFF^\gamma)$.
	\end{enumerate}	
	\noindent For free boundary conditions:
	\begin{enumerate}[align = left, labelwidth=\parindent, labelsep = 0pt]
		\item[(1.\textup{f})]\ Every loop $\gamma\in\mathcal L_{-2\lambda, 2\lambda}(h)$ is given a cluster $\C^\dagger_\gamma=\A_{-2\lambda, (2\sqrt{2}-2)\lambda}(\GFF^\gamma)$.
		\item[(2.\textup{f})]\ Iterate as in (3.\textup{w}).
	\end{enumerate}
\end{theorem}

As part of the second input, we stress once more that we do not use the convergence of all $k$-point spin correlation functions due to Chelkak, Hongler and Izyurov \cite{CHI1}. Instead, we only need the existence of the limit of the one-point function. We use its explicit expression to prove the uniqueness part of Theorem \ref{thm:unique-meas}.

\begin{theorem}[Convergence of one-point function \cite{CHI1}]\label{thm:CHI1}
	Let $D$ be a bounded, simply connected domain with discrete domain approximation $\Dd\subset\de\Z^2$. At criticality, for any $\eps>0$,
	\[
	\lim_{\de \to 0}\ \de^{ - 1 /  8} \E_{\Dd}^+ \left[ \sigma_\de(z) \right] = \mathfrak{C}\ 2^{1/4}\ \CR(z, D)^{-1/8}
	\] 
	uniformly over all points $z$ such that $d(z, \partial D)>\eps$, where $\CR(z, D)$ denotes the conformal radius of $D$ as viewed from $z$.
\end{theorem}

Finally, the uniqueness part of Theorem \ref{thm:unique-meas} is achieved by verifying the axioms posed in {\cite[Theorem 1.2]{CLE-meas}} for the uniqueness of measures supported on the carpet/gasket of some $\CLE_\kappa$ for $\kappa\in(8/3, 8)$. 
Below, we present a modified version of the statement in \cite{CLE-meas}. The key modification needed for our application comes from replacing the assumption on the conformal covariance of the measure by only the conformal covariance of the \emph{intensity measure}. This follows from essentially the same proof as in \cite{CLE-meas}. 

In terms of notation, let $\C$ be the carpet of the \CLE$_4$ in a simply connected domain $D$. Given an open, simply connected set $U \subset D$, we let $U^\C$ be the set obtained as the complement in $U$ of the closure of the union of the interior of the loops of $\C$ which intersect $U$ and $D\setminus U$. The Markov property of the $\CLE_4$ states that, for any connected component $V$ of $U^\C$, the conditional law of the restriction of $\C$ to $V$ given $\C\cap V^c$ is that of a $\CLE_4$ carpet in the new domain $V$. This is depicted in Figure \ref{fig:cle-markov}. We have stated the theorem in terms of a fixed domain $D$, which is an equivalent reformulation of the statement in \cite{CLE-meas}\footnote{Indeed, once a pair $(\C_D, \mu_D)$ satisfies the uniqueness axioms in $D$, it follows that for any conformal map $\psi:D\to\tilde\D$, the pair $(\C_{\tilde D}, \mu_{\tilde D})$ given by $\C_{\tilde D}=\psi(\C_D)$ and the pushforward $\mu_{\tilde D}=\psi_{*}(|\psi'|^d\mu_D)$ also satisfies the axioms in $\tilde\D$. This is immediate from the conformal invariance of the $\CLE_4$ and a simple chain rule computation. 
}.

	\begin{figure}[t]
			\centering
			\includegraphics[width=0.40\linewidth]{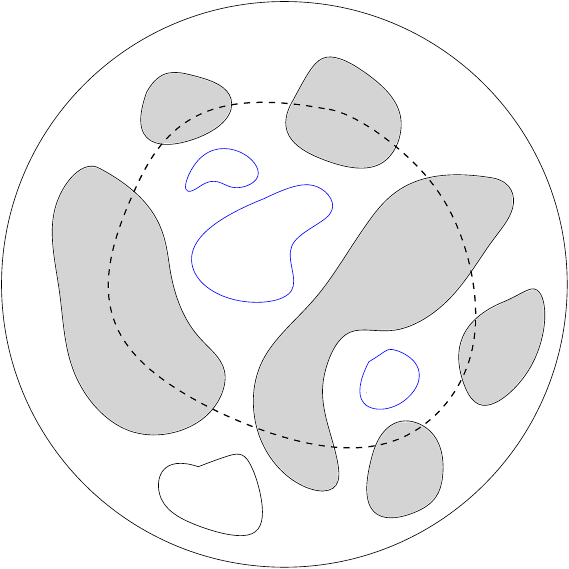}
			\caption{The dashed line is the boundary of a simply connected domain $U$. It defines a domain $U^\C$, with two connected components, by removing all the loops of $\C$ that intersect both the inside and the outside of $U$. The Markov property of the $\CLE_4$ says that the blue loops have the law of (independent) $\CLE_4$ in each of the connected components of $U^\C$.}
			\label{fig:cle-markov}
	\end{figure}

\begin{theorem}[Unique characterisation of measures]\label{thm:Miller-Schoug}
	Let $D$ be a simply connected domain. Consider a law on a pair $(\C,\mu)$ where $\C$ has the law of the carpet of the \CLE$_4$ in the domain $D$ and $\mu$ is a measure supported on the set $\C$. Suppose in addition that the joint law of $(\C,\mu)$ is such that:
	\begin{enumerate}
		\item (Measurability) The measure $\mu$ is measurable with respect to $\C$.
		\item (Local finiteness)  For each compact set $K \subset D$, $\E \left[ \mu(K) \right] < \infty$.
		\item (Markov property) For any open, simply connected set $U \subset D$ and any component $V$ of $U^\C$, the conditional law of $\mu|_{V}$ given $\C \cap V^c$ and $\mu|_{V^c}$ is only dependent on $V$. 
		\item (Conformal covariance of the intensity measure) Let $U, V$ be as above. For any $f:D\to\R$ with compact support in $V$ and any conformal map $\varphi:D\to V$,
		\[
		\E \left[ \mu[f\circ\varphi^{-1}] \mid V \right] = \int_D |\vp'(z)|^{d} f(z) \d  \E\left[ \mu \right](z).
		\] 
	\end{enumerate}
	Then, the law on $(\C,\mu)$ is uniquely determined (up to a multiplicative constant) and is equal to the measure constructed in \cite{CLE-meas}. In particular, the measure is conformally covariant.
\end{theorem}

Beyond these three inputs, the rest of the proof is essentially self-contained. As already hinted, the \emph{main difficulty} of the argument is to show that the measurability of the discrete IMFs with respect to the percolations $\omega_\de$ and $\omega_\de^\dagger$ is preserved upon taking the scaling limit $\de\to0$. Namely, we will define a collection of (appropriately renormalised) counting measures $\mu^\de_{\C^\de}$ on each discrete cluster $\C^\de$ and focus our efforts on proving that the limit $\mu_\C$ of these measures is measurable with respect to the continuum cluster $\C$.
 Whilst this may seem like a technical point, without it it would not be possible to translate the coupling in the discrete into our representation of the IMF in terms of two-valued sets of the GFF. The main novelties of the proof itself are as follows.
\begin{itemize}
	\item A new argument to prove measurability, based on the approaches of \cite{GPS, mag-field}. It uses a box-counting approximation as in \eqref{eq:byboxcount}, rather than an approximation by annuli crossings.

	\item We prove that the limit $\mu_\C$ is measurable with respect to the cluster $\C$ itself, not just the complete limit $\omega$. This is precisely why we can identify the limiting measures in Theorem~\ref{thm:unique-meas}.

	\item We develop new tools to circumvent the repeated issue that the law of $\w_\de$ is not known to satisfy the FKG inequality. We also construct the IIC measure of $\w_\delta$ by adapting the recent argument of \cite{hillairet2025short} to our setting.
\end{itemize}

\noindent \textbf{Outline.} The paper is organised as follows. In Sections \ref{sec:defandprop}--\ref{sec:heightfunc} we introduce several equivalent constructions of the discrete coupling and prove Theorem \ref{thm:main-ES}. We then discuss generalisations to the Ashkin-Teller model in Section \ref{sec:vargen}. Thereafter, we devote Sections \ref{sec:FKG}--\ref{sec:RSW} to deriving the most important properties of our percolation model: an alternative to the FKG inequality and standard RSW theory estimates. These properties are used in Section \ref{sec:IIC} to construct the incipient infinite cluster (IIC) measure of the model. Once all these features are in place, in Section \ref{sec:conv-meas}, we prove the key estimate (Theorem \ref{thm:boxcount}) which can be thought of as the discrete version of Theorem \ref{thm:unique-meas}. This will imply that the scaling limit of the area measures of the discrete clusters is measurable with respect to the continuum clusters. This is the crucial input needed to pass from the discrete to the continuum, which we finalize in Section \ref{sec:proofs} to prove Theorems \ref{thm:main-coupling}--\ref{thm:main-decomp}--\ref{thm:unique-meas}--\ref{thm:joint}.

\noindent \textbf{Acknowledgments. } We are grateful to Lukas Schoug for kindly discussing the modified statement in Theorem \ref{thm:Miller-Schoug} with us. We thank Christophe Garban for helpful discussions concerning the construction of the IIC measure in Section \ref{sec:IIC}. We also thank Nathana\"{e}l Berestycki for valuable comments on a first version of the paper. Finally, TAL would like to thank Avelio Sepúlveda for many helpful conversations. 

\noindent This research was funded in part by the Austrian Science Fund (FWF) grant P36298 ``Spins, Loops and Fields''.
The research of ML was also supported by the Austrian Science Fund (FWF) grant F1002 ``Discrete random structures: enumeration and scaling limits''. TAL acknowledges the support from a grant of the Vienna School of Mathematics (VSM).

\section{The discrete coupling}

\subsection{Definition and main properties}\label{sec:defandprop}

Let $G = (V(G),E(G))$ be a finite connected graph.
For any subset of vertices $S \subset V(G)$, we denote by $E(S)$ those edges with both endpoints contained in $S$.
We consider a family of ferromagnetic coupling constants $J : E(G) \to [0,\infty)$.

We also fix $V^+ \subset V(G)$ to be a (possibly empty) connected subset of the vertices and set $E^+ = E(V^+)$. These will be used to impose $+$ boundary conditions on $V^+$ for spin configurations and wired boundary conditions on $E^+$ for percolation configurations.

Below we will often assume that $G$ is planar and embedded in the plane, and we will write $\bG$ for the set of vertices of $G$ incident to the outer face.
We then denote by $G^*$ the dual graph of $G$ and by $G^\dual$ the weak dual of $G$, that is the (strong) dual graph with the vertex corresponding to the outer face of $G$ removed.
For those statements which rely on planarity, we will \emph{additionally assume} that $V^+$ is either empty or is equal to $\bG$.

Given any coupling constants $J$ there is a unique dual family of coupling constants $J^\dual$ on $G^\dual$ which are defined through the relation
\begin{equation}\label{eq:Jdualdef}
	\tanh( J^\dual_{(uv)^\dual}) = \exp( - 2J_{uv}),
\end{equation} 
{for every $uv \in E(G)$, where $(uv)^\dual$ is the dual edge of $uv$}.
We note that the mapping $J \mapsto J^\dual$ is an involution and the unique fixed point gives the self-dual coupling constants as in~\eqref{eq:critical-Ising}, which are critical on the square lattice.

Given $A \subset V(G)$, we define $\Sigma^{A}$ to be the set of spin configurations $\sigma : V \mapsto \left\{ \pm 1 \right\}$ which satisfy $\sigma_v = +1$ for all $v \in A$.

\begin{definition}[Ising model] \label{def:Ising}
{The \emph{Ising model} with $+$ boundary conditions on $V^+$ is the probability measure $\P^+_{G}=\P^+_{G, J}$ on spin configurations $\sigma: V(G) \to \left\{ \pm 1 \right\} $ given by
\[
	\P^+_G(\sigma) \propto \i( \sigma \in \Sigma^{V^+}) \exp\Big( \sum_{uv\in E(G)} J_{uv} \sigma_u \sigma_{v}\Big).
\] 	
When the set $V^+$ is empty, we instead say that the model has \emph{free} boundary conditions and we write $\P^{\f}_G$.
}
\end{definition}
We immediately extend the measure to include a second independent and identically distributed Ising model $\tilde{\sigma}$,
\begin{equation}\label{eq:xorham}
	\P^+_G(\sigma,\tilde{\sigma}) \propto \i(\sigma,\tilde{\sigma} \in \Sigma^{V^+}) \exp \Big ( \sum_{uv \in E} J_{uv} ( \sigma_u \sigma_v + \tilde{\sigma}_u\tilde{\sigma}_v) \Big).
\end{equation}
{This is the non-interacting version of the Ashkin--Teller model with $U=0$, and the product spin will be denoted by
\[
\tau = \sigma \tilde{\sigma}.
\] 
The marginal of $\tau$ under $\P^+_G$ (resp.~$\P^{\f}_G$) is referred to as the \emph{XOR-Ising model} with $+/+$ (resp.~free/free) boundary conditions.}

\begin{definition}[Double Random Current]\label{def:drc}
	A current is a map $\n : E(G) \mapsto \left\{ 0,1,2,\ldots \right\}$. We denote by $\b\n \subseteq V(G)$ the set of its \emph{sources}, 
	\begin{equation}\label{eq:sourcedef}
	\b\n = \left\{ u \in V(G) \setminus V^+ \ \mid \ \textstyle \sum_{v: \ uv \in E(G)} \n_{uv} \text{ is odd}  \right\}.
\end{equation} 
For $A\subseteq V(G)$, we write $\Omega_A$ for the set of all currents with sources $A$. 
	We introduce a weight $w = w_J$ on currents by setting
	\begin{equation}\label{eq:drcdef}
		w(\n) = \prod_{uv \in E(G)} \frac{J_{uv}^{\n_{uv}}}{\n_{uv}!}.
	\end{equation} 
	Whenever $A,B \subseteq V(G)$ are of even size, the \emph{double random current (DRC)} probability measure $\P^{A,B}_{G}=\P^{A,B}_{G,J}$ with sources $A,B$ is defined to be the law of $\n_1 + \n_2$, where $\n_1$ and $\n_2$ are independently sampled proportional to the weight $w_J$ from $\Omega_A$ and $\Omega_B$ respectively.
	
	The case of main importance for this article is when $A=B=\emptyset$, and if so we say that the currents are \emph{sourceless}. 
	The model is said to have free boundary conditions when $V^+$ is empty, otherwise it is said to have wired boundary conditions.
	The \emph{trace} $\nh$ of $\n$ is a percolation configuration given by
	\[
		\nh = \left\{ uv \in E(G) \mid \n_{uv} > 0 \right\}.
	\] 
	We define $\n_{\odd}$ and $\n_{\even}$ as the subsets of $\nh$ on which $\n$ takes odd and even values respectively.
\end{definition}

This definition for random currents with wired boundary conditions is not the most standard but it is natural. The set of sources \eqref{eq:sourcedef} arises from expanding in the usual manner the partition function of the Ising model with $+$ boundary conditions on $V^+$. The constraint $\b\n = \emptyset$ ensures that
\[
	\textstyle\sum_{u \in V^+} \textstyle \sum_{v : \ uv \in E(G)} \n_{uv} \ \text{ is even,} 
\] 
so that the current is completely sourceless when considered on the graph with the edges of $E^+$ contracted and all vertices of $V^+$ merged into one.

Given a spin configuration $\sigma : V(G) \to \left\{ \pm 1 \right\}$, we write
\[
\xi^\pm(\sigma) = \left\{ uv \in E(G) \mid \sigma_u = \sigma_v = \pm 1 \right\} 
\] 
and we define $\xi(\sigma) = \xi^+(\sigma) \cup \xi^-(\sigma)$. 
When $G$ is planar, the {dual complement} of $\xi(\sigma)$ is the \emph{contour (domain-wall)} representation of $\sigma$ and will be denoted
\begin{equation}\label{eq:domain-wall}
	\b \sigma = \xi(\sigma)^* = \left\{ (uv)^* \in E(G^*) \mid \sigma_u \neq \sigma_v \right\}. 
\end{equation} 

Before proceeding, we briefly recall the definition of the FK-Ising model and refer to \cite{RC} for a full account. We define $\wFK \in \left\{ 0,1 \right\}^E$ by setting
\begin{equation}\label{eq:fkdef}
\wFK = \xi(\sigma) \cap \eta',
\end{equation} 
where $\sigma \sim \P^+_G$ and $\eta'$ is sampled independently of $\sigma$ according to an (inhomogeneous) Bernoulli bond percolation with parameter
\[
	p'_{uv} = 1 - \exp(-2J_{uv})
\]
for $uv \in E(G) \setminus E^+$ and $p'_{uv} = 1$ (i.e.~always open) for $uv \in E^+$. The law of $\wFK$ will be denoted by $\phi^+_{G} = \phi^+_{G,J}$ and it has an explicit form given by
\begin{equation}\label{eq:fklaw}
	\phi^+_{G}(\wFK) \propto 2^{k(\wFK)} \prod_{uv \in \wFK} p'_{uv} \prod_{uv \not\in \wFK} (1-p'_{uv}),
\end{equation} 
where $k(\wFK)$ is the number of connected components of $(V(G),\wFK)$ including the isolated vertices.
As above, we will write $\phi^{\textnormal{f}}_{G}$ in the case that $V^+$ is empty and refer to the model as having \emph{free} boundary conditions. Otherwise, we say that the boundary conditions are \emph{wired} and note that $\wFK \supseteq E^+$.
Note that by construction, for wired boundary conditions, $V^+$ will belong to a unique cluster of $\wFK$.

We now state the theorem of Edwards and Sokal, relating the FK-Ising percolation model to the Ising spin model. Before doing so, we recall (and slightly generalise) the notation $\coin$ introduced in Definition \ref{def:coins}

\begin{definition}[Coin tosses] \label{def:coins2}
For any percolation configuration $\w \subset E(G)$, we denote by
\[
	\coin(\w) 
\]
the assignment of independent symmetric $\{\pm 1\}$-valued spins to each cluster of $\w$, and by $\coin^+(\w)$ the same procedure with the exception that each cluster intersecting $V^+$ is always assigned the spin $+1$.
\end{definition}

\begin{definition}[]
	We say that a probability measure $P$ on $\left\{ 0,1 \right\}^{E(G)}$ has the \emph{Edwards--Sokal property} (with $+$ boundary conditions) if $\coin^+(\w)$ has the law of the Ising model on $G$ with $+$ boundary conditions when $\w \sim P$.
\end{definition}

\begin{theorem}[Edwards--Sokal \cite{EdwSok}]\label{thm:ESfk}
	In the coupling $(\wFK, \sigma)$ defined in \eqref{eq:fkdef}, conditionally on $\wFK$ the law of $\sigma$ is uniform on the subset of $\Sigma^{V^+}$ where each cluster of $\w$ has a constant spin. In particular, the law $\phi^+_{G}$ has the Edwards-Sokal property. 
\end{theorem}

The use of the FK-Ising model and its percolative perspective to understand the behaviour of the Ising model is a recurring theme in many works. Part of the reason for this is that the measure $\phi^+_G$ enjoys many nice properties. We will focus on two here which will play an important role for us.
The first is the behaviour under duality. 
\begin{itemize}
	{\item \textbf{(Duality)} Let $G$ be a planar graph and let $\wFK\sim \phi^+_{G,J}$. Then $(\wFK)^*\sim \phi^\f_{G^\dual,J^\dual}$. }
\end{itemize}
On the square lattice with the self-dual coupling constants, the involution $\wFK \leftrightarrow (\wFK)^*$ therefore only induces a change in the boundary conditions of the measure.
The second important property is valid on any finite graph $G$.
\begin{itemize}
\item \textbf{(FKG)} The law $\phi^+_{G,J}$ satisfies the FKG inequality and is stochastically increasing in $J$. 
\end{itemize}

We will discuss the FKG inequality and its consequences more in Section \ref{sec:FKG}.
Together these two properties are very powerful and for instance allow one to prove that the phase transition of the Ising model on $\Z^2$ is continuous \cite{werner2009percolation}.

We now introduce our main definition. It is a coupling between the Ising model on $G$ and a percolation model on $E(G)$, which will be denoted $\w$. It starts by extending $\P^+_G$ to include an independent Bernoulli bond percolation $\eta \subset E(G)$ of (inhomogeneous) parameter
\[
	p_{uv} = 1 - \exp(-4J_{uv})
\]
for $uv \in E(G) \setminus E^+$ and $p_{uv} = 1$ (i.e.~always open) for $uv \in E^+$.

\begin{definition}\label{def:one}
	Define the measure
	\[
		\P^+_{G}(\sigma,\tilde{\sigma},\eta) \propto \i(\sigma,\tilde{\sigma} \in \Sigma^{V^+}) \exp \Big(\sum_{uv \in E(G)} J_{uv} (\sigma_u \sigma_v + \tilde{\sigma}_u \tilde{\sigma}_v)\Big)\prod_{uv \in \eta} p_{uv} \prod_{uv \not\in \eta} (1-p_{uv} ).
	\]
	on the set $\left\{ \pm 1 \right\}^V \times \left\{ \pm 1 \right\}^V \times \left\{ 0,1 \right\}^{E(G)}$ and construct $\w \subset E(G)$ by setting
	\begin{equation}\label{eq:wdef}
	\w = \xi(\sigma) \cap \xi(\tilde{\sigma}) \cap \eta.
	\end{equation} 
	The free and wired measures are defined as before, noting again that in the wired model $\omega \supseteq E^+$. 
\end{definition}

The first observation is that \emph{both} spins $\sigma$ and $\tilde{\sigma}$ must be constant on any cluster of $\w$, and so too the product $\tau$. This motivates us to introduce a decomposition based on the sign of~$\tau$,
\[
\w = \w^+ \cup \w^-,
\] 
where we set $\w^\pm = \w \cap \xi^\pm(\tau)$. Note that $\w^+$ and $\w^-$ are actually vertex-disjoint so we may consider this as a decomposition of the clusters of $\w$. The pair $(\w^+,\w^-)$ will play an important role in Section \ref{sec:propofrep}. 

We are now in a position to state the main theorem of this section: the pair $(\sigma, \omega)$ is coupled with the Edwards--Sokal property, just like the FK-Ising model in Theorem \ref{thm:ESfk}.

\begin{theorem}[New Edwards--Sokal coupling]\label{thm:edwardsokal}
	In the coupling $\P^+_G$ of Definition \ref{def:one} and conditionally on $\w$, the law of $\sigma$ is uniform on the subset of $\Sigma^{V^+}$ where each cluster of $\w$ has a constant spin.
	In particular, the law of $\w$ has the Edwards--Sokal property. 
\end{theorem}

At first sight, it may be surprising that the Ising model admits more than one percolation measure with the Edwards-Sokal property.
This was first observed in \cite{steif2019generalized} and further studied in \cite{forsstrom2022color}\footnote{We propose to use the term Edwards-Sokal property, differentiating them from the term `divide-and-colour representation' of \cite{steif2019generalized} by requiring that the partitions they induce are connected subsets of the graph. }. 
Our construction gives an explicit family of such measures for any finite graph $G$ (generically different from the FK-Ising model) which continuously depend on the coupling constants $J$. See Remark \ref{remark:sources} for an even larger class of examples.

A classical consequence of the relation $\sigma \sim \coin^+(\w)$ are the identities relating connection events for $\w$ to spin correlations in $\sigma$. In particular, we will make frequent use of 
\begin{equation}
	 \E^+_{G} \left[ \sigma_x \right] =\P^+_{G}( x \overset{\w}{\longleftrightarrow} \bG) 
\end{equation}
and
\begin{equation}	
\label{eq:sokalsitetosite}
 \E^+_{G}\left[ \sigma_x \sigma_y \right]=\P^+_{G}( x \overset{\w}{\longleftrightarrow} y) ,
\end{equation} 
for any $x,y \in G$. Here the standard notation $ \overset{\w}{\longleftrightarrow}$ indicates that both sides belong to the same connected component of $\omega$.
Before going further, we introduce another result which may again be surprising to the reader. Its simple proof is postponed to the end of this section. 

\begin{prop}\label{prop:dualiden}
	Let $G$ be a planar graph. 
	\begin{itemize}
		\item (Wired boundary conditions) Let $\n \sim \P^{\emptyset,\emptyset}_{G^\dual,J^\dual}$ be the double random current on the \textbf{weak dual} graph with free boundary conditions. Then, the dual complement $\nh^*$ of $\nh$ has the same distribution as $\w \sim \P^+_G$.
		\item (Free boundary conditions) Let $\n \sim \P^{\emptyset,\emptyset}_{G^*,J^\dual}$ be the double random current on the \textbf{strong dual} graph with free boundary conditions. Then, the dual complement $\nh^*$ of $\nh$ has the same distribution as $\w \sim \P^\f_G$.
	\end{itemize}
\end{prop}

For a planar graph, Theorem \ref{thm:edwardsokal} can be summarised thus: \emph{tossing coins on the clusters of the {dual complement} of the double random current results in an Ising model.}
This should not be confused with the fact that applying the same procedure to the double random current itself (i.e. \emph{not} its dual complement) will result instead in the law of the XOR-Ising model \cite{spins-perc-height}.

The next result is a link between $\w$ and the double random current \emph{which holds on any finite graph $G$}. It will be particularly useful in Section \ref{sec:heightfunc} to extend the coupling further. We state the result as an another definition of $\w$.

\begin{definition}[]\label{def:two}
	Let $G$ be any finite graph. We construct a coupling $\P_G^+$ of $(\n,\omega,\tau,\sigma)$, where $\n\in \Omega_\emptyset$, $\omega \in \{0,1\}^{E(G)}$ and $\tau,\sigma \in\{ \pm 1 \}^{V(G)}$, as follows.
	\begin{enumerate}[(1)]
	\item Sample the double random current $\n \sim \P^{\emptyset,\emptyset}_{G}$ with wired boundary conditions.
	\item Sample $\tau \in \left\{ \pm 1 \right\}^{V(G)}$ by setting $\tau = \coin^+(\nh)$.
	\item Sample an independent bond percolation $\eta$ with parameters $1 - \exp(-2J_{uv})$.
	\item Set 
		\[
			\w = \nh \cup (\eta \cap {\xi(\tau)}),
		\]
	\item Sample $\sigma \in \left\{ \pm 1 \right\}^{V(G)}$ by setting $\sigma = \coin^+(\omega)$.
\end{enumerate}
\end{definition}

The key point is that the coupling of $(\w,\sigma,\tau\sigma)$ in Definition \ref{def:two} then exactly coincides with the coupling of $(\w,\sigma,\tilde{\sigma})$ in Definition \ref{def:one}. This is not difficult to check using arguments as in Proposition \ref{prop:couplingstruc} and we leave it to the interested reader. 

\begin{remark}
	This description of the law of $\w$ above should be compared to the known coupling between the single random current and the FK-Ising model \cite{lupu2016note,ADTW}. Similar to there, exactly the same procedure as in Definition \ref{def:two} except starting instead with the sourced double random current measure
	\[
		\n \sim \P^{A,\emptyset}_G 
	\] 
	will yield $\w$ sampled according to the conditioned measure $\P_G(\cdot \mid \w \in \F_A)$, where $\F_A$ is the event that the clusters of $\w$ induce an even partition on the vertices of the (even) subset $A$.
\end{remark}
We will now give two proofs of Theorem \ref{thm:edwardsokal}. The first applies on any finite graph $G$ and will along the way introduce further aspects of the coupling which we will later need. 
The second uses Proposition \ref{prop:dualiden} to prove the result on a planar graph by appealing to the switching lemma for the random current model. 
{A version of the latter can be found in the master's thesis of Ivan Hejn\'y~\cite{Hej} which was 
directly based on~\cite{LisT}. A statement to the same effect appears also in~\cite{klausen2022monotonicity} (also using~\cite{LisT}).}

In the first proof, the product $\tau$ will be fundamental to understanding the coupling between $\sigma$ and $\w$. We start with an observation which concerns only the pair $(\sigma,\tau)$.

\begin{lemma}[]\label{lemma:favlemma}
	Let $(\sigma,\tilde{\sigma})$ be two independent Ising models as in \eqref{eq:xorham} and set $\tau = \sigma \tilde{\sigma}$. Then conditionally on $\tau$, $\sigma$ is still distributed as an Ising model on $G$ except now with $\tau$-dependent coupling constants $J^\tau$ given by
	\[
	J^{\tau}_{uv} = 2 \times \i( uv \in \xi(\tau) )  J_{uv} 
	\] 
	for each  $uv \in E(G).$
\end{lemma}
\begin{proof}
	We may simply rearrange the Hamiltonian appearing in \eqref{eq:xorham},
	\begin{equation}\label{eq:Jtaudef}
		J_{uv}( \sigma_u \sigma_v + \tilde{\sigma}_u \tilde{\sigma}_v) = 2 \times \i(\tau_u = \tau_v)J_{uv} \sigma_u \sigma_v,
\end{equation} 
from which the result is self-evident.
\end{proof}

Note that by the simple form of \eqref{eq:Jtaudef} we can instead interpret it as the Ising model restricted to the subgraph $\xi(\tau)$ and with coupling constants $2J$. This implies for instance that the state of $\sigma$ restricted to $\left\{ \tau = -1 \right\}$ and to $\left\{ \tau = +1 \right\}$ become conditionally independent. With this lemma in hand, we can now clearly elucidate the structure of the coupling.

\begin{prop}[]\label{prop:couplingstruc}
	Consider the \emph{conditional} law of the coupling $\P^+_G$  given $\tau$:
	\begin{enumerate}
		\item The (conditional) law of the pair $(\omega, \sigma)$ is that of \eqref{eq:fkdef} between the Ising model and the FK-Ising model with coupling constants $J^\tau$. In particular, the conditional law of $\w$ is $\phi^+_{G,J^\tau}$. 
		\item The two percolations $\w^+$ and $\w^-$ are (conditionally) independent and distributed according to the measure $\phi^+_{G,J^\tau}$ restricted, respectively, to the subgraphs $\xi^+(\tau)$ and $\xi^-(\tau)$. 
	\end{enumerate}
\end{prop}

\noindent Using Proposition \ref{prop:couplingstruc}, we can also give an explicit density for the pair $(\tau,\w)$ in the coupling~$\P$,
\begin{align}\label{eq:fulldef}
	\P^+_G( (\tau,\w)) \propto \i(\tau \in \Sigma^{V^+}, \w \subseteq \xi(\tau)) \times 2^{k(\w)} \prod_{uv \in \w} p_{uv} \prod_{uv \in \xi(\tau) \setminus \w} (1-p_{uv}) \\
	\propto \i(\tau \in \Sigma^{V^+}, \w \subseteq \xi(\tau)) \times \exp( - \sum_{uv \in E(G)} J\tau_u \tau_v) \times 2^{k(\w)} \prod_{uv \in \w} p_{uv} \prod_{uv \notin \w} (1-p_{uv}).\nonumber 
\end{align} 
This will be used in Section \ref{sec:FKG}.
Before proving the result, we show how it implies Theorem \ref{thm:edwardsokal}. 

\begin{proof}[First proof of Theorem \ref{thm:edwardsokal}]
	By Proposition \ref{prop:couplingstruc} and Theorem \ref{thm:ESfk}, we see that conditionally on both $\w$ and $\tau$, $\sigma$ is uniformly distributed among spin configurations in $\Sigma^{\bG}$ which are constant on clusters of $\w$. Since this law is independent of $\tau$, it must also hold without conditioning on $\tau$, which proves the result. 
\end{proof}

\begin{proof}[Proof of Proposition \ref{prop:couplingstruc}] 
	Note that the definition \eqref{eq:wdef} may be equivalently written as
	\begin{equation}\label{eq:wdefprime}
	\w = \xi(\sigma) \cap \xi(\tau) \cap \eta,
	\end{equation} 
	and recall that $\eta$ is open independently on each edge $uv$ with probability $1 - \exp(-4J_{uv})$. In particular, we see that $\xi(\tau) \cap \eta$ may be interpreted (conditional on $\tau$) as a Bernoulli bond percolation with parameter
	\[
		1 - \exp(- 2J^{\tau}_{uv}).
	\] 
	Using Lemma \ref{lemma:favlemma}, we can now exactly identify \eqref{eq:wdefprime} as the definition of the FK-Ising model we gave in \eqref{eq:fkdef}. The independence of $\w^+$ and $\w^-$ is the property of the FK-Ising model, seen evidently from \eqref{eq:fklaw}, that the restrictions to different connected components of the domain are independent.
\end{proof}

\def\Ev{\mathcal{E}}
We now turn to our second proof of Theorem \ref{thm:edwardsokal} and restrict to the case that $G$ is planar. We prove the result for $+$ boundary conditions on $\bG$, though the argument is almost identical for free boundary conditions. There are two properties of the double random current model we will use. 

The first is {very closely related to} the \emph{switching lemma}, which has been instrumental in the study of random currents and the Ising model since its introduction in \cite{GHS}. The slight generalisation that we use may be found in \cite[Lemma 2.1]{ADTW}\footnote{Note that there is stated that the conditional law of $\n_1$ is invariant under $\n_1 \mapsto \n_1 \triangle \mathbf{k}$ for any fixed $\mathbf{k} \in \Ev(\n)$ which, this map being a bijection on $\Ev(\n)$, is equivalent to the version we use here.}.
For any graph $H$, we let $\Ev(H)$ be the set of \emph{even} (or \emph{Eulerian}) subgraphs (subgraphs whose induced vertex degrees are all even) of $H$. When $\n$ is a current we identify it with {the 
multigraph having $\n_e$ copies of each edge $e\in E(G)$}, and consider $\Ev(\n)$ to be the set of even submultigraphs of $\n$. We take $\n = \n_1 + \n_2$ to be sampled according to $\P_{G,J}^{\emptyset,\emptyset}$. 
\begin{itemize}
	\item \textbf{(Switching lemma)} Conditionally on $\n$, the law of $\n_1$ is uniform on $\Ev(\n)$.
\end{itemize}
Our second property concerns the Kramers-Wannier duality of the Ising model. When $G$ is planar, there is a one-to-one correspondence between the set $\Sigma^{\bG}$ of spin configurations {with $+$ boundary conditions on $\bG$} and the set $\Ev(G^\dual)$ of even subgraphs of the weak dual $G^\dual$.
The bijection is given by matching each $\sigma \in \Sigma^{\bG}$ to its domain-wall expansion $\b\sigma$ (recall \eqref{eq:domain-wall}). Let $\n$ be a single random current on $G^\dual$ sampled proportional to the weight $w_{J^\dual}$.
\begin{itemize}
	\item \textbf{(Odd part of a single current)} The law of $(\n_1)_\odd$ is that of $\b\sigma$ under $\sigma \sim \P_G^+$.
\end{itemize}
This result can be found as part of \cite[Theorem 3.2]{ADTW}.

\begin{remark}[]\label{remark:XORdual}
We will later use the analagous fact that the odd part of the entire double current $\n_\odd$ is the domain-wall expansion of the XOR-Ising model $\tau$. This due to the equivalence betweeen
\[
\n_\odd = (\n_1)_\odd \, \triangle\, (\n_2)_\odd \quad\text{ and }\quad \tau = \sigma \tilde{\sigma}.
\] 
\end{remark}

\begin{proof}[Second proof of Theorem \ref{thm:edwardsokal}]
	We take $\n = \n_1 + \n_2 \sim \P^{\emptyset,\emptyset}_{G^\dual,J^\dual}$. 
	We then set $\w = (\nh)^*$, which by Proposition \ref{prop:dualiden} has the same law as under $\P_G^+$. Our goal is to identify the law of $\coin^+(\w)$.

Under the correspondence described above, the subset of spin configurations $\sigma \in \Sigma^{\bG}$ which are constant on the clusters of $\w$ is equivalent to the set $\Ev(\n)$. As such, using the identification of $(\n_1)_\odd$, it suffices to prove that conditionally on $\n$, the law of $(\n_1)_\odd$ is uniform on $\Ev(\nh)$. 
We will show that this follows from the property that the law of $\n_1$ is uniform on $\Ev(\n)$ (i.e. the switching lemma) once we have the following simple combinatorial fact.  
\begin{itemize}
	\item Fix any current $\n$. The map $\pi : \Ev(\n) \to \Ev(\nh)$ projecting any even submultigraph onto its set of odd edges $\pi(\m) = \m_\odd$ has the same number of preimages for any element of $\Ev(\nh)$. 
\end{itemize}
Indeed, for any $\eta \in \Ev(\nh)$ the number of $\pi$-preimages factorises over each edge and is given by
\[
	\prod_{e \in \eta} \left( \sum_{\substack{1 \leq m \leq \n_e \\ m \text{ odd}}} {\n_e \choose m}\right) \prod_{e \not\in \eta} \left( \sum_{\substack{0 \leq m \leq \n_e \\ m \text{ even}}} {\n_e \choose m}\right) = 2^{\sum_{e} ( \n_e - 1)},
\] 
which is independent\footnote{Here, we used the classical identity that $1 - 1 = 0$.} of $\eta$. This is then exactly the right condition to ensure that the conditional law of $(\n_1)_\odd$ is still uniformly distributed, now on the set $\Ev(\nh)$, completing the proof.
\end{proof}

\begin{remark}
	Under the guise of computing correlation functions for fermionic disorder operators (which correspond to order operators on the dual), the observation that tossing coins on the clusters of the dual complement of the double random current leads to the law of an Ising model can be deduced from \cite{lis2022boundary, ADTW,DRC}, thus providing a third proof of Theorem \ref{thm:edwardsokal}.
\end{remark}

\begin{remark}\label{remark:sources}
	The careful reader will have noticed that the above proof actually extends to the double random current $\n = \n_1 + \n_2$ sampled according to the measure $\P^{A,\emptyset}_{G^\dual}$ (i.e. with sources $\b\n_1 = A$ and $\b\n_2 = \emptyset$) for any subset $A \subset V$. That is, it still has the property that conditionally on $\n$, the law of $(\n_2)_\odd$ is uniform on $\Ev(\nh)$. This leads to a large class of examples of measures with the Edwards-Sokal property. Interestingly, the FK-Ising measure may be shown to be \emph{non-extremal} in this collection, using the self-duality of $\phi_{G,J}$ and an explicit convex decomposition
	\begin{equation}
		\phi^\f_{G,J} = \sum_{A \subset V} \alpha_A \times \P^{A,\emptyset}_{G,J}\left( \nh \in \cdot \right),
	\end{equation}
	where the coefficients satisfy $\sum_A \alpha_A = 1$ and $\alpha_A \propto \E^\f_{G,J}[\prod_{v \in A} \sigma_v]$. These observations are the starting point to couple Ising models with \emph{different} boundary conditions, which will appear in future work.
\end{remark}

\begin{proof}[Proof of Proposition \ref{prop:dualiden}]
	We let $\n = \n_1 + \n_2$ be the double random current model on the weak dual $G^\dual$. Upon taking the dual of \eqref{eq:wdef},
	\[
	\w^* = \b\sigma \cup \b\tilde{\sigma} \cup \eta^*.
	\] 
	We identify each of these terms with the decomposition
	\[
	\nh = (\n_1)_\odd \cup (\n_2)_\odd \cup ((\n_1)_\even \cap (\n_2)_\even).
	\] 
	Recall that $(\n_1)_\odd$ is distributed as $\b\sigma$ so we may focus on the last term. From the definition \eqref{eq:drcdef} we see that conditionally on $(\n_1)_\odd$, the edges belonging to $(\n_1)_\even$ are an independent sprinkling (i.e.~ Bernoulli bond percolation) in the complement of $(\n_1)_\odd$ with parameter
	\[
		\tanh J^\dual_{uv} = \exp( - 2J_{uv}).
	\] 
	Thus, conditionally on both $(\n_1)_\odd$ and $(\n_2)_\odd$, those edges belonging to $(\n_1)_\even \cap (\n_2)_\even$ are an independent sprinkling in the complement of $(\n_1)_\odd \cup (\n_2)_\odd$ with parameter
	\[
		(\tanh J^\dual_{uv})^2 = 1 - (1 - \exp(-4J_{uv})).
	\] 
	By \eqref{eq:Jdualdef} this exactly coincides with the law of $\eta^*$, completing the proof.
\end{proof}

\subsection{The height function}\label{sec:heightfunc}
\def\blackdot{\bullet}
\def\whitedot{\circ}

In this section we fix $G$ to be planar.
Our goal is to explain how to extend the coupling of the Ising model to include a discrete height function $H$, which will be the discrete analogue of the coupling between the IMF and the GFF.
The height function is equivalent to the six-vertex height function at the free-fermion point.
Proving delocalisation of the six-vertex height function in the regime $1 \leq c \leq 2$ was a major open problem which has now been settled through numerous works. The expected convergence to the GFF in bounded domains is currently only known at the free-fermion point by the work of \cite{DRC,DRC2}.
We will touch more on the relationship with the six-vertex model in Section \ref{sec:vargen}.

The height function that we use was introduced in \cite{nesting-field} and the convergence in finite domains proved in \cite{DRC}.
We give a step-by-step construction of the coupling which does not involve combinatorial mappings to the dimer model as in \cite{DRC} (although the reader should note that the dimer model is fundamental to proving the required convergence to the Gaussian free field).

We start by observing that when $G$ is planar, the construction given in Definition \ref{def:two} may be entirely dualised. From now on, whenever we refer to a random current $\n$ we will consider only its projection onto $(\n_\odd,\n_\even)$ since the exact value of the current will never play a role. We also always impose $+$  boundary conditions on the primal graph $G$ and free boundary conditions on the weak dual $G^\dual$\footnote{This is the opposite convention to \cite{DRC}}.

\begin{definition}[]Let $G$ be a planar graph.
	\begin{enumerate}[(1)]
		\item Construct the coupling $(\n, \w, \tau, \sigma)$ exactly as in Definition \ref{def:two}.
		\item Define, up to a global spin flip which we choose independently,  $\tau^\dual\in\left\{ \pm 1 \right\}^{V(G^\dagger)} $ by 
			\[
				\b\tau^\dual = \n_\odd.
			\]
		\item Set $\w^\dual = \nh^*$ and define $\n^\dual$ by the two conditions that
			\[
			\nh^\dual = \w^* \quad\text{ and }\quad \n^\dual_\odd = \b\tau.
			\]
		\item Define $\sigma^\dual\in \left\{ \pm 1 \right\}^{V(G^\dual)}$ by setting $\sigma=\coin(\omega^\dual)$, and define
			\[
				\tilde{\sigma}^\dual = \sigma^\dual\tau^\dual.
			\]
			
	\end{enumerate}
\end{definition}

There are a few technicalities with this definition: 
\begin{itemize}
	\item In step (3), the two conditions are compatible since $\b\tau$ is an even subgraph and also $\b\tau \subset \w^*$ (as $\tau$ is constant on the clusters of $\w$).
\item In step (4), it is important that the signs on each cluster of $\w^\dual$ are independent from each other and from $\n$. 
However, there is an ambiguity in that it will not be important that these signs are independent from the signs used to construct $\sigma$. 
In this sense, \emph{there is no coupling of all four Ising models which is canonical}. Since it is the most simple, we may choose to take these signs completely independently from everything else; however, it will play no role.
\end{itemize}

It remains to show that the introduced objects on the dual graph have the desired marginals.
\begin{itemize}
\item The marginal of $\n^\dual$ is the double random current on $G^\dual$ with free boundary conditions and coupling constants $J^\dual$.
This is the content of Proposition \ref{prop:dualiden} and Remark \ref{remark:XORdual}. 
\item The marginal of $\tau^\dual$ is the XOR-Ising model on $G^\dual$ with free boundary conditions and coupling constants $J^\dual$, again by Remark \ref{remark:XORdual}.
\item The marginals $\sigma^\dual$ and $\tilde{\sigma}^\dual$ are independent Ising models on $G^\dual$ with free boundary conditions and coupling constants $J^\dual$. This can be seen from the second proof of Theorem \ref{thm:edwardsokal}. Indeed, conditionally on $\n$, the law of $\b\sigma^\dual$ (which determines $\sigma^\dual$ up to a global sign) is equal to the conditional law on $(\n_1)_\odd$ i.e.
	\[
		\b\sigma^\dual \sim (\n_1)_\odd.
	\]
	Since $\b\tau^\dual = \n_\odd$, as in Remark \ref{remark:XORdual}, the conditional law of $\b \widetilde{\sigma}^\dual$ is then $(\n_2)_\odd$, i.e.
	\[
		\b \widetilde{\sigma}^\dual \sim (\n_2)_\odd.
	\]
Therefore, the unconditional joint law $(\b\sigma^\dual,\b \widetilde{\sigma}^\dual)$ is the same as $((\n_1)_\odd,(\n_2)_\odd)$, yet $\n_1$ and $\n_2$ were independent by construction which proves the claim.
\end{itemize}

\noindent We note that the coupling of $(\n^\dual, \tau^\dual, \w^\dual, \sigma^\dual)$ is another instance of the coupling in Definition~\ref{def:two}, now under free boundary conditions. In particular, conditionally on $\n^\dual$, the law of $\tau^\dual$ is that of $\coin(\nh^\dual)$.

A particularly important consequence of this construction is that the resulting pair $(\tau,\tau^\dual)$ obeys the consistency condition: \emph{never} do both
\begin{equation}\label{eq:taucons}
	\tau_u \neq \tau_v \quad\text{ and }\quad \tau^\dual_{u^\dual} \neq \tau^\dual_{v^\dual}
\end{equation} 
occur for a pair $uv \in E(G)$ and $u^\dual v^\dual \in E(G^\dual)$ of edges dual to one another. To see this, note that in the construction $\nh \subset \w$ and so $\nh$ and $\nh^\dual$ (and in particular $\n_\odd$ and $\n^\dual_\odd$) will never cross. 

We may now give an explicit definition of the height function $H$. 
Let $\Gd$ be the diamond graph of $G$. This is the planar bipartite graph with vertex set $V(G) \cup V(G^\dual)$ and an edge between $u$ and $u^\dual$ if and only if $u \in V(G)$ is a vertex incident to the face $u^\dual \in V(G^\dual)$. 

\begin{definition}[]
	We define $H$ to be the unique function $H : V(\Gd) \to \frac{1}{2}\Z$ such that $H|_{\bG} = 0$ and for every $u \in V(G)$ incident to a face $u^\dual \in V(G^\dual)$ we have
	\[
		H_{u^\dual} - H_u = \frac{1}{2} \tau_u \tau_{u^\dual}^\dual.
	\] 
\end{definition}

It is easy to verify that this local rule for defining $H$ is globally consistent. Indeed, it suffices to check that the relevant one-form is closed around the faces of $\Gd$, i.e. around the four-cycles $uu^\dual v v^\dual$ for each pair of edges $uv \in E(G)$ and $u^\dual v^\dual \in E(G^\dual)$ which are dual to each other. 
We only have to notice that
\begin{equation}
	\frac{1}{2} ( \tau_u \tau_{u^\dual}^\dual - \tau^\dual_{u^\dual} \tau_v + \tau_v \tau^\dual_{v^\dual} - \tau^\dual_{v^\dual} \tau_u) 
	= \frac{1}{2} (\tau_u - \tau_v) (\tau^\dual_{u^\dual} - \tau^\dual_{v^\dual}),
\end{equation} 
which must vanish by the consistency condition \eqref{eq:taucons}.
It is also immediate from the definition that $H$ will take integer values on $G$ and half-integer values on $G^\dual$.
Moreover, we have 
\[
	\tau_u = (-1)^{H_u} \quad\text{ and }\quad \tau^\dual_{u^\dual} = (-1)^{H_{u^\dual} - 1 / 2}.
\] 
This is the discrete analagoue of the realisation of the two continuum polarisation fields (with $+$ and free boundary conditions respectively) as the cosine and sine of an appropriately scaled GFF (see Conjecture \ref{conj:AT-IGMC} and \cite{XOR-exc}).

We have now defined the entire coupling of all the discrete objects appearing in the statement of Theorem \ref{thm:joint}:
\[
	(\sigma,\tilde{\sigma},\sigma^\dual,\tilde{\sigma}^\dual, \tau, \tau^\dual, \n,\n^\dual, \w,\w^\dual, H).
\] 
The coupling of $(H,\tau,\tau^\dual,\n,\n^\dual)$ (i.e. without any of the Ising models) appeared in \cite{DRC} where it was referred to as the \emph{master coupling}. We stress that various partial aspects of this coupling have appeared before in the literature and in Remark \ref{remark:history} we include a list of such instances.

\begin{remark}[]
	We point out that when $G$ and $G^\dual$ are suitable subsets of the square lattice, the graph $\Gd$ is another subset of the square lattice and it is the dual of this graph on which the six-vertex arrow configurations are defined. The definition of $H$ given here differs by a factor of $1/2$ from the usual definition of the six-vertex height function (e.g. as described in \cite{GL} where it instead always takes even and odd values respectively on subsets of the bipartite decomposition of $\Z^2$, corresponding here to $V(\Gd) = V(G) \cup V(G^\dual)$). 
\end{remark}

\subsection{Extension to the Ashkin--Teller model}\label{sec:vargen}

We now introduce the generalisation of the coupling $\P^+_G$ to the interacting version of the Ashkin-Teller model.
We will restrict ourselves to the case where 
\begin{equation}\label{eq:ATreg}
	J_{uv} \geq 0 \quad\text{ and }\quad \cosh(2J_{uv}) \geq e^{-2U_{uv}}
\end{equation}
for every edge $uv\in E(G)$. The reader should note that the second inequality is often written in the equivalent form $$\tanh(U _{uv})\geq - (\tanh(J_{uv}))^2.$$
Actually neither the definition of the coupling nor Theorem \ref{thm:ATedwardsokal} will require this restriction. However, it is only in this regime that the Ashkin-Teller model has a duality relation \cite{PfiVel}. 
This is a mapping from $(J,U)$ to a pair $(J^\dual,U^\dual)$ which also satisfy \eqref{eq:ATreg}. We note that the self-dual line
\[
	\sinh(2J_{uv}) = e^{-2U_{uv}}
\] 
always satisifies \eqref{eq:ATreg} and is the unique line of fixed points to the duality map.
This is also the same regime where correlation inequalities are known for the Ashkin-Teller spin model \cite{PfiVel}, and our proof of Proposition \ref{prop:FKG}, which we choose to give in this general setting, will rely on this restriction. The following is the extension of Definition~\ref{def:one}.
\begin{definition}[]\label{def:three}
	Define the measure
	\[
		\P^+_{G}(\sigma,\tilde{\sigma},\eta) \propto \i(\sigma,\tilde{\sigma} \in \Sigma^{V^+}) \exp(\sum_{uv \in E(G)} J_{uv}(\sigma_u \sigma_v + \tilde{\sigma}_u \tilde{\sigma}_v) + U_{uv} \tau_u \tau_v )\prod_{uv \in \eta} p_{uv} \prod_{uv \not\in \eta} (1-p_{uv})
	\]
	on the set $\left\{ \pm 1 \right\}^V \times \left\{ \pm 1 \right\}^V \times \left\{ 0,1 \right\}^{E(G)}$ and construct $\w \subset E(G)$ by setting
	\begin{equation}\label{eq:omega-def}
	\w = \xi(\sigma) \cap \xi(\tilde{\sigma}) \cap \eta.
	\end{equation} 
\end{definition}

\noindent Note that the parameter $p = 1 - \exp(-4J)$ remains the same as before and does not depend on the value of $U$. The statement of Theorem \ref{thm:edwardsokal} remains true in the interacting model.

\begin{theorem}[]\label{thm:ATedwardsokal}
	In the coupling $\P^+_G$ and conditionally on $\w$, the law of $\sigma$ is uniform among spin configurations $\sigma \in \Sigma^\bG$ which are constant on clusters of $\w$. In particular, $\sigma \sim \coin^+(\w)$.
\end{theorem}

This result can be proved in a similar way to Theorem \ref{thm:edwardsokal}. The key point is that we have only changed the measure by a factor depending on $\tau$, and so properties of the law conditional on $\tau$ (in particular Lemma \ref{lemma:favlemma}) are unaffected. 

If the graph $G$ is planar and \eqref{eq:ATreg} is satisified, then there is also an analogous result to Proposition \ref{prop:dualiden}. Indeed, one can identify $\w^*$ with the same law of the Ashkin-Teller random current with weights $(J^\dual,U^\dual)$ introduced in Section \ref{sec:conj}. We leave the exact details of this correspondence to the reader.

\begin{remark}
	As explored in \cite{spins-perc-height,GL}, the behaviour of the percolation model changes along the self-dual line. 
	\begin{enumerate}
		\item 
	In the section $J \geq U$, there is a \emph{super-duality} at play. That is, a coupling between $\w \sim \P^+_G$ and $\w^\dual \sim \P^\f_{G^\dual}$ such that
	\[
	\w \subset (\w^\dual)^* \quad\text{ and }\quad \w^\dual \supset \w^*
	\] 
	i.e. every primal-dual edge pair is covered by at least one of $\w$ or $\w^\dual$. 
	\item At the point $J = U$, the law of $\w$ is equal to the critical $\textrm{FK}(4)$ random cluster model and is exactly self-dual. 
	\item In the other regime $J \leq U$ of the self-dual line, the model instead exhibits sub-duality. Moreover, the percolation model $\w$ is actually known to satisfy the FKG inequality when $J \leq U$, which is unavailable when $J > U$ (see Section \ref{sec:FKG}). 
	\end{enumerate}
	It is interesting that this change exactly occurs as the six-vertex model itself changes from delocalised to localised behaviour.
\end{remark}

\begin{remark}\label{remark:history}
	We now list various places where aspects of this coupling have appeared before.
	\begin{itemize}
		\item To our knowlege, the coupling of $(\sigma,\tilde{\sigma},\tau,\w)$ first appeared in the work of \cite[Equation (7)]{chayes1998percolation} in the $U \geq 0$ regime of the Ashkin-Teller model.
		\item The coupling of $(\tau,\tau^\dual,\w^*,(\w^\dual)^*,H)$, i.e. taking the dual of each percolation, appeared in the work \cite{spins-perc-height}.
		\item 
			The coupling of $(\sigma,\tilde{\sigma},\tau,\w, 2H)$ also appears explicitly in the work \cite{AT-6V} in the setting of six-vertex model on $\Z^2$ (corresponding to the self-dual line of the Ashkin-Teller model). 
	\item The coupling of $(\tau,\tau^\dual,\w,\w^\dual,2H)$ is exactly the coupling introduced in \cite[Definition 4.5]{GL} to study the delocalisation of the six-vertex model on the square lattice. The correspondence between the notation is given by $(\tau,\tau^\dual,\w,\w^\dual) \leftrightarrow (\sigma^\blackdot,\sigma^\whitedot,\xi^\blackdot,\xi^\whitedot)$. There, the property of super-duality, mentioned above, is crucially used to conclude delocalisation of the six-vertex height function.
\end{itemize}
\end{remark}

{
\begin{remark}
Generalisations of this coupling to other situations, including Dobrushin and more general boundary conditions, will appear in a future work. 
\end{remark}
}

\section{Properties of the representation}\label{sec:propofrep}

\subsection{Correlation inequalities}\label{sec:FKG}

In this section, we will study the FKG properties of the measure $\P^+_G$ in Definition \ref{def:three}. A recurring theme of our argument is the issue that the law of $\w$ is not known to satisfy the FKG inequality. In fact, it is generally expected that the double random current (equivalently $\w$) does not satisfy the FKG inequality, although ruling it out completely is still seemingly open \cite{klausen2022monotonicity}. 

The results of this section may be seen as one approach to overcome this hurdle. The first result recovers a specific type of FKG inequality for the decomposition $(\w^+,\w^-)$. It originates from the beautiful work of \cite{GL}. After this, we prove a technical extension so that we can apply it to the construction of the IIC in Section \ref{sec:IIC}. We then state and prove a simple correlation inequality which will allow us to decorrelate an increasing event from a decreasing event. 

We turn first our attention to the FKG-properties of the measure $\P^+_G$.
We briefly recall that a random variable $X \sim \mu$ taking values in a (finite) partially-ordered set $(\mathcal X,\leq)$ has the \textbf{FKG property} if for any two increasing functions $f,g : \mathcal X \mapsto \R$,
\[
	\mu \left[ f(X) g(X) \right] \geq \mu[ f(X)] \E[ g(X)].
\] 
The only partially-ordered sets we will consider are the spaces of spin configurations $\left\{ \pm 1 \right\}^V$ and bond configurations $\left\{ 0,1 \right\}^E$ with their natural order, and various products of these spaces. We use the notation $\tau_S$ or $\w^+_H$ to denote the restriction $\tau|_S$ or $\w^+|_H$ for any subset $S \subset V$ or $H \subset E$.

The most important result of this section concerns the distribution of the triple $(\tau,\w^+,\w^-)$ under the law $\P^+_G$. We use the notation $(\tau,\w^+,-\w^-)$ to denote that we consider the triple living in the partially-ordered set 
\[
	\left\{ \pm 1 \right\}^V \times \left\{ 0,1 \right\}^{E} \times \left\{ 0,1 \right\}^E,
\] 
with the order $(\tau_1,\w^+_1,\w^-_2) \leq (\tau_2,\w^+_2,\w^-_2)$ if and only if $\tau_1 \leq \tau_2$, $\w^+_1 \leq \w^+_2$ and $\w^-_1 \geq \w^-_2$. 
That is, we have \emph{reversed the order on the third component} and a function $f$ is considered increasing in $(\tau,\w^+,-\w^-)$ if it is increasing in both $\tau$ and $\w^+$ but decreasing in $\w^-$.

\begin{prop} \label{prop:FKG} The coupling $\P^+_G$ of Definition \ref{def:three} enjoys the following stochastic properties:
	\begin{enumerate}
		\item The triple $(\tau,\w^+,-\w^-)$ has the FKG property.
		\item The boundary cluster $\C_0$ of $\w$ has the FKG property.
	\end{enumerate}
\end{prop}

	The first property should be understood as saying that individually $\w^+$ and $\w^-$ have the FKG property but that they are \emph{negatively correlated} with each other. This replaces the lack of a traditional FKG inequality for the union $\w = \w^+ \cup \w^-$.  
	The second property in Proposition~\ref{prop:FKG} is an immmediate consequence of the first (since the boundary cluster $\C_0$ is an increasing function of $\w^+$) but is remarkably useful and worth highlighting by itself. 

	The result of Proposition \ref{prop:FKG} is contained in \cite{GL}, where it stated in terms of the six-vertex spin representation of the square lattice. There, this intricate FKG structure was key to deriving delocalisation for the six-vertex height function.
	We review the proof here for completeness and because it is an essential component in what is to come. It is a consequence of three important facts: 
	\begin{enumerate}[(1)]
		\item The marginal law on $\tau$ 
		has the FKG property, see Lemma~\ref{lemma:taufkg}.
		\item The percolations $\w^+$ and $\w^-$ are conditionally independent given $\tau$.  
		\item Given $\tau$, the conditional law of $\w^+$ (resp. $\w^-$) is increasing (resp. decreasing) in~$\tau$.
	\end{enumerate}
	The first of these facts was proved in \cite[Lemma 1]{chayes1998percolation} when $U \geq 0$. It was also proved in \cite{AT-6V} and \cite{GL} along the entire self-dual curve of the square lattice. These later proofs only superficially use planar duality and their adaptation to a general setting is the proof we give here. 
	The second and third facts rely on the structure of the coupling elucidated in Proposition \ref{prop:couplingstruc}. Indeed, they follow from general properties of the FK-Ising model $\phi_{G,J}$ on a graph: independence across connected components of the $G$ and monotonicity of the law in $J$ . Together, these three facts may be chained to conclude that the triple $(\tau,\w^+,-\w^-)$ satisifies the FKG inequality.
	
	We also observe that the first and third facts above imply that, under the measure $\P^+$ (or trivially $\P^\f$), there is a stochastic domination
	\begin{equation}\label{eq:stocdom}
	\w^+ \succeq \w^-.
	\end{equation} 

\begin{remark}[]\label{remark:antiferro}
	The primary reason why $\w$ does not immediately satisfy FKG is that in the regime $J > U$, the $\tau$-spin of each cluster is distributed, conditionally on $\w$, as an \emph{antiferromagnetic} 
	Ising model conditioned to be constant on each cluster (see \eqref{eq:fulldef}). It should be noted that the statement of Proposition \ref{prop:FKG} does not immediately preclude $\w$ from having FKG. As an example, if $(\w^+,\w^-)$ represent the two signs of clusters of the FK-Ising model then both the pair $(\w^+,-\w^-)$ and the union $\w^+ \cup \w^-$ satisfy the FKG inequality. 
\end{remark}

This next result will be used to show that the the configuration $(\tau,\w^+)$ may be \emph{explored} in a suitable manner whilst preserving the FKG inequality at each stage. The combination of exploration algorithms (sometimes referred to as decision trees) and the FKG inequality has recently been useful in various approaches to study percolation models \cite{duminil2019sharp,duminil2022planar,gladkov2024percolation}. 
Applications so far have generally been restricted to \emph{monotonic measures}, a strictly stronger property than satisifying the FKG inequality. 
A measure $\mu$ on $\left\{ 0,1 \right\}^E$ is monotonic if it satisfies the following condition: 
\begin{enumerate}[(1)]
	\item Let $X \sim \mu$. For each subset $H \subset E$, the conditional measure 
		\[
		\mu(X \in \cdot \mid X_H)
		\] 
		is stochastically increasing in $X_H \in \left\{ 0,1 \right\}^H$.
\end{enumerate}
Whenever the measure $\mu$ is strictly positive\footnote{The reader should beware that various measures we consider will not always be positive due to hard-core constraints.} this condition is equivalent to another condition:
\begin{enumerate}[(1)]
	\setcounter{enumi}{1}
	\item Let $X \sim \mu$. For each subset $H \subset E$ and $X_H \in \left\{ 0,1 \right\}^H$, the conditional measure
		\[
		\mu(X \in \cdot \mid X_H)
		\] 
		has the FKG property.
\end{enumerate}
We refer to \cite{RC} for a proper introduction to the notion of a monotonic measure and other equivalent characterisations.
The combination of the two conditions (1) and (2) will play the important role in our argument in Section \ref{sec:IIC}.
They should be understood as stating that it is possible to partially explore $X$ whilst ensuring that various correlation inequalities remain valid.

Unforunately, we are unable to prove that the law of $\w^+$ is monotonic.
However, the following result will suffice for us. It states that it is possible to explore the \emph{joint pair} $(\tau,\w^+)$ in a specific manner that preserves the FKG inequality. 

\begin{lemma}[]\label{lemma:conFKG}
	Let $S \subset V$ and $H \subset E(G)$ satisfy the constraint $H \subset E(S)$. Then, the conditional measures 
	\[
	\P^+_G( (\tau,\w^+) \in \cdot \mid \tau_S, w^+_H ),
	\] 
satisfy the FKG-inequality and\footnote{Note that without the constraints on $H$ and $S$, the FKG part of the statement would also imply the stochastic ordering. With the constraint, this implication is no longer clear.} are stochastically increasing in $(\tau_S,\w^+_H)$. 
\end{lemma}

\begin{remark}
	The constraint that $H \subset E(S)$ is crucial to our argument and we do not know how to prove the result (it may not even be true) otherwise. The implication for our exploration procedure is that when discovering the state of a new edge, we must first ensure that we have discovered the $\tau$-spin on both endpoints of the edge before proceeding.
\end{remark}

We now turn to a correlation lemma for $\w$ which, whilst not strictly necessary, is helpful to streamline the proof of Theorem \ref{thm:boxcount} and may be useful more generally.
We introduce the following definition.
\begin{definition}[Primitive events]
	An increasing event $A$ is primitive for $\w$ if for any $\w \in A$ and any colouring of the clusters of $\w$ there is still a monochromatic subset which satisifies $A$.
\end{definition}

The definition should be interpreted that there is a single cluster `witnessing' the event $A$. Most increasing events related to the standard RSW theory of planar percolation (for instance the event $A \leftrightarrow B$ for two sets $A$ and $B$) are seen to be primitive events. The usefulness of this definition is as follows. For any increasing primitive event $A$, if we set $A^\zeta = \left\{ \w^\zeta \in A \right\}$ for each $\zeta \in \left\{ \pm 1 \right\}$, then there is simple inclusion
\[
A \subset A^+ \cup A^-.
\] 
Note that for events depending only on the cluster of a single vertex $x$ say, then is the exact decomposition $A = A^+ \sqcup A^-$.
Overall, if $A$ is any increasing primitive event, then using also \eqref{eq:stocdom}
\begin{equation}\label{eq:primitive}
\P(\w^+ \in A) \leq \P(\w \in A ) \leq 2 \P(\w^+ \in A).
\end{equation}
This turns out to be a very useful comparison tool when combined with the joint FKG property of $\w^+$ and $\w^-$.

\begin{lemma}[Anti-correlation lemma]\label{lemma:decorrlemma}
	Let $A$ be an increasing primitive event and $B$ a decreasing event. Then,
	\[
	\P(\w \in A \cap B) \leq 2  \P(\w^- \in B)  \P(\w \in A).
	\] 
\end{lemma}
\begin{proof}
	Using a union bound and the inclusions $A \subset A^+ \cup A^-$ (since $A$ is primitive) and $B \subset B^\pm$ (since $B$ is decreasing), 
	\[
	\P(\w \in A \cap B) \leq \P(A^+ \cap B^+) + \P(A^- \cap B^-).
	\] 
	We may now apply the FKG property in Proposition \ref{prop:FKG} along with the stochastic domination $\w^+ \succeq \w^-$ to deduce the result. 
\end{proof}

\def\DL{\mathcal{E}}
\def\S{\mathcal{S}}
The rest of this section is devoted to the proofs of Lemma \ref{lemma:conFKG} and Lemma \ref{lemma:taufkg} , which are slightly involved. We have chosen to adopt the general setting of the Ashkin-Teller coupling from Definition \ref{def:three}, where it is crucial that the coupling constants satisfy \eqref{eq:ATreg}.

In terms of notation, given any $\w^+_H \in \left\{ 0,1 \right\}^H$, we set $V(\w^+_H)$ to be the set of vertices incident to an open edge in $\w^+_H$. Moreover, we let $\nu$ be the marginal law on $\tau$ under $\P^+_G$. Recall that a measure $\nu$ on $ \left\{ \pm 1 \right\}^V$ satisifies the FKG--lattice condition if 
\begin{equation}\label{eq:fkglattice}
\nu(\tau \vee \tau) \nu(\tau \wedge \tau') \geq \nu(\tau) \nu(\tau')
\end{equation}
for each $\tau, \tau' \in \left\{ \pm 1 \right\}^V$. When $\nu$ is positive, this is sufficient for the FKG property~\cite{fortuin1971correlation,RC}.

\begin{lemma}\label{lemma:taufkg}
	The marginal $\nu$ satisifes the FKG-lattice condition.
\end{lemma}
\begin{proof}
	Our first aim is to  give an explicit definition of the law of $\nu$.
	We consider the map $x = (x_e)_{e \in E} \mapsto Z_G(x)$, with $0 \leq x_e \leq 1$, given by
	\[
		Z_G(x) = \sum_{\sigma \in\left\{ \pm 1 \right\}^V} \i(\sigma \in \Sigma^{V^+}) x(\sigma),
	\] 
	where, for any spin $\sigma \in\left\{ \pm 1 \right\}^V$ and coupling constants $x : E \to \left[ 0,1 \right]$, we write
	\[
		x(\sigma) = \prod_{\substack{uv \in E \\ \sigma_u \neq \sigma_v}} x_{uv}. 
	\] 
	As can be seen from Lemma \ref{lemma:favlemma}, the marginal $\nu$ is explictly given by
	\begin{equation}\label{eq:nulaw}
		\nu(\tau) \propto \i(\tau \in \Sigma^{V^+} ) \times y(\tau) \times Z_G(x^+_\tau) \times Z_G(x_{\tau}^-).
	\end{equation} 
	where
	\begin{align} \label{eq:defy}
	y_{uv} = \exp(-2(J_{uv} + U_{uv}))
	\end{align} 
	and 
	\begin{equation}\label{eq:xplusdef}
	x_\tau^\pm = \begin{cases}
		x_{uv} = \exp(-4J_{uv}) \quad \text{ if } \tau_u = \tau_v = \pm 1,\\
		1 \quad \text{ otherwise,}
	\end{cases}
	\end{equation} 
	for every $uv \in E(G)$. We will consider the lattice condition \eqref{eq:fkglattice} for each of the three factors in \eqref{eq:nulaw} seperately.

	\noindent \textbf{(1) The term $y(\tau)$}. It is classical that whenever $0 \leq y \leq 1$ (i.e. the model is ferromagnetic), the FKG-lattice condition is satisfied:
	\begin{equation}\label{eq:ytlatweak}
	y(\tau \vee \tau') y(\tau \wedge \tau') \geq y(\tau) y(\tau').
	\end{equation} 
	Alone this will not suffice to prove the result and we will have to be more precise. We introduce the set of edges
	\[
		\DL_{\tau,\tau'} = \left\{ uv \in E(G) \mid uv \not\in \xi(\tau), uv \not\in \xi(\tau'), \tau_u \neq \tau'_u \right\}.
	\] 
	There is then an exact identity,
	\begin{equation}\label{eq:ytaulat}
		y(\tau \vee \tau') y(\tau \wedge \tau') = \left( \prod_{uv \in \DL_{\tau,\tau'}} \frac{1}{y_{uv}^2} \right) y(\tau) y(\tau').
	\end{equation} 

	\noindent \textbf{(2) The term $Z_G(x^+_\tau)$. } We need two properties of the partition function $Z_G(x)$. 

	The first is the FKG-lattice condition \cite[Lemma 6.1]{lammers2024delocalisation}, 
	\begin{equation}\label{eq:Zlat}
	Z_G(x \vee x') Z_G(x \wedge x') \geq Z_G(x) Z_G(x'),
	\end{equation} 
	for any $0 \leq x,x' \leq 1$. The second is the inequality
	\begin{equation}\label{eq:Zmon}
		Z_G(x') \geq \prod_{uv \in E} \frac{1 + x'_{uv}}{1 + x_{uv}} \times Z_G(x)
	\end{equation} 
	for $0 \leq x' \leq x \leq 1$\footnote{Note that this inequality goes against the trivial mononoticity of the map $x \mapsto Z_G(x)$. In the planar setting, \eqref{eq:Zmon} follows immediately from the Kramers-Wannier duality and said monotonicity but for the dual partition function.}. To prove \eqref{eq:Zmon}, we can argue by integrating the inequality
	\begin{equation}\label{eq:Zdiff}
		\frac{\b}{\b x_{uv}} \log Z_G(x) =  \frac{1}{2x_{uv}} (1 - \left< \sigma_u \sigma_v \right>^+_{x}) \leq \frac{1}{2x_{uv}} \left( 1 - \left< \sigma_u \sigma_v \right>_{x'}^+ \right)  \leq \frac{1}{1+x_{uv}},
	\end{equation} 
	where $x'_{uv} = x_{uv}$ but otherwise $x'_{zw} = 1$. The first inequality in \eqref{eq:Zdiff} follows from the monotonicity of Ising spin correlations and the second by a direct computation.

	\noindent Whilst it is true that $x^+_{\tau \wedge \tau'} = x^+_\tau \vee x^+_{\tau'}$, there is only in general the inequality 
	\[
		x^+_{\tau \vee \tau'} \leq x^+_{\tau} \wedge x^+_{\tau'}.
	\] 
	Applying \eqref{eq:Zmon}, we see that
	\[
		Z_G(x^+_{\tau \vee \tau'}) \geq \left( \prod_{uv \in \DL_{\tau,\tau'}}  \frac{1+x_{uv}}{2} \right) Z_G(x^+_\tau \wedge x^+_{\tau'}).
	\]
	We may now use \eqref{eq:Zlat}, leading to 
	\begin{equation}\label{eq:Zpluslat}
		Z_G(x_{\tau \vee \tau'}^+) Z_G(x_{\tau \wedge \tau'}^+) \geq \left( \prod_{uv \in \DL_{\tau,\tau'}} \frac{1 + x_{uv}}{2} \right) \times Z_G(x_\tau^+) Z_G(x_{\tau'}^-). 
	\end{equation}
	This same inequality also holds for the term $Z_G(x^-_\tau)$ by identical reasoning (with the roles of $\wedge$ and $\vee$ exchanged).

	\noindent We can now put \eqref{eq:ytaulat} and \eqref{eq:Zpluslat} together to show that
	
	\[
		\nu(\tau \vee \tau') \nu(\tau \wedge \tau') \geq \left( \prod_{uv \in \DL_{\tau,\tau'}}  \frac{1+x_{uv}}{2 y_{uv}} \right)^2 \nu(\tau) \nu(\tau'). 
	\] 
	The condition is thus satisfied whenever $1+x_{uv} \geq 2 y_{uv}$, which is equivalent to \eqref{eq:ATreg}.
\end{proof}

\begin{proof}[Proof of Lemma \ref{lemma:conFKG}]
	By Proposition \ref{prop:couplingstruc}, the conditional law of $\w^+$ under $\P^+_G$ given $(\tau, \w^+_H)$ is simply the FK-Ising measure
	\[
		\P_G^+(\w^+ \in \cdot \mid \tau, \w^+_H) = \phi^+_{\xi^+(\tau),2J}( \cdot \mid \w^+_H),
	\] 
	which still satisfies the FKG-inequality and is stochastically increasing in both $\w^+_H$ and $\tau$. 
	Thus, it remains to prove that the conditional laws
	\[
		\P_G^+(\tau \in \cdot \mid \tau_S, \w^+_H)
	\]
	satisfy the FKG-inequality and are increasing in the pair $(\tau_S, \w^+_H)$. To do so, we introduce a collection of auxiliary measures on $\left\{ \pm 1 \right\}^V$ and parameterised by $\w^+_H \in \left\{ 0,1 \right\}^H$ given by
	\[
		\mu^{\w^+_H}(\tau) \propto \i(\tau \in \Sigma^{V^+ \cup V(\w^+_H)}) \times y(\tau) \times Z(x^+_{\tau,\w^+_H}) \times Z(x^-_{\tau}).
	\]
	Here, 
	\[
	 (x^+_{\tau,\w^+_H})_{uv} = \begin{cases} \i(uv \not\in \w^+_H)\quad \text{if}\ uv \in H,\\
	 	(x^+_\tau)_{uv}\quad \text{otherwise,}
	 	\end{cases}
	\]
	where $x^+_\tau$,$x^-_\tau$ are defined as in \eqref{eq:xplusdef} and $y$ is as in~\eqref{eq:defy}. We proceed in two steps.

	\noindent \textbf{Step 1. } We claim the following. Given $\tau_S$ and $\w^+_H$, the conditional law coincides with the auxiliary measure above, i.e.
	\begin{equation}\label{eq:mulaweq}
	\P_G^+( \tau \in \cdot \mid \tau_S, \w^+_H) =  \mu^{\w^+_H}(\ \cdot \mid \tau_S),
	\end{equation} 
	for every $\tau_S \in \left\{ \pm 1 \right\}^S$ and $\w^+_H \in \left\{ 0,1 \right\}^H$.
	Note that this is \emph{not true} without the conditioning on $\tau_S$ as we will make no statement about the distribution (conditionally on $\w^+_H$) of $\tau_S$ itself. It is also the place where the restriction $H \subset E(S)$ is relevant. The claim follows from an inspection of the density in \eqref{eq:fulldef}.

	\noindent We now argue that it is enough to verify the following (generalised) FKG-lattice condition: For any $\eta,\eta' \in \left\{ 0,1 \right\}^H$ such that $\eta \geq \eta'$, the inequality
	\begin{equation}\label{eq:mulat}
		\mu^{\eta}(\tau \vee \tau') \mu^{\eta'}(\tau \wedge \tau') \geq \mu^{\eta}(\tau) \mu^{\eta'}(\tau')
	\end{equation} 
	holds for all $\tau, \tau' \in \left\{ \pm 1 \right\}^V$. Indeed, this is sufficient to imply that the law 
	\[
		\mu^{\w^+_H}(\ \cdot \mid  \tau_S)
	\] 
	satisifies the FKG inequality and is increasing in both $\w^+_H$ and $\tau_S$ \cite{RC}, which together with the identication \eqref{eq:mulaweq} is exactly what is desired.

	\noindent \textbf{Step 2.} We now verify the inequality \eqref{eq:mulat} by arguing similarily to the proof of Lemma~\ref{lemma:taufkg}. A simple first observation is that if $\tau \in \Sigma^{V^+ \cup V(\eta)}$ and $\tau' \in \Sigma^{V^+\cup V(\eta')}$ then it must be that
	\[
		\tau \wedge \tau' \in \Sigma^{V^+ \cup V(\eta')} \text{ and } \tau \vee \tau' \in \Sigma^{V^+ \cup V(\eta)},
	\] 
	because of the ordering $\eta \geq \eta'$. We now turn to the coupling constants $(x^+_{\tau,\eta})$ which we recall are independent of $\tau$ whenever $uv \in H$. 
	The important observation is that
	\[
		(x^+_{\tau \vee \tau',\eta})_{uv} = (x^+_{\tau,\eta} \wedge x^+_{\tau',\eta'})_{uv} \quad \text{ and } (x^+_{\tau \wedge \tau', \eta'})_{uv} = (x^+_{\tau, \eta} \vee x^+_{\tau',\eta})_{uv}
	\]
	when $uv \in H$, due to the ordering $\eta \geq \eta'$.
	Thus, we have that (arguing exactly as in \eqref{eq:Zmon} on $ E(G) \setminus H$ and then using \eqref{eq:Zlat})
	\begin{equation}\label{eq:Zpluswlat}
		Z_G(x_{\tau \vee \tau',\eta}^+) Z_G(x_{\tau \wedge \tau',\eta'}^+) \geq \left( \prod_{uv \in \DL_{\tau,\tau'} \setminus H} \frac{1 + x_{uv}}{2} \right) \times Z_G(x^+_{\tau, \eta}) Z_G(x_{\tau',\eta'}^+). 
	\end{equation}
	The corresponding inequality for $Z(x^-_\tau)$ is exactly as before. 
	All together, we deduce that
	\begin{equation}\label{eq:mulatineq}
		\mu^{\eta}(\tau \vee \tau') \mu^{\eta'}(\tau \wedge \tau') \geq \prod_{uv \in \DL_{\tau,\tau'} \setminus H} \left( \frac{1+x_{uv}}{2y_{uv}} \right)^2 \prod_{uv \in \DL_{\tau,\tau'} \cap H} \left( \frac{1+x_{uv}}{ 2 y^2_{uv}} \right) \mu^{\eta}(\tau) \mu^{\eta'}(\tau'),
	\end{equation} 
	so \eqref{eq:mulat} is satisfied whenever we are in the regime \eqref{eq:ATreg}.
\end{proof}

\begin{remark}\label{remark:altFKG}
	We would now like to mention an alternative approach to above, that we originally intended to take. The definition \eqref{eq:omega-def} suggests that at first glance it is most obvious to decompose $\w$ into \emph{four} vertex-disjoint subsets based on the joint spin of $(\sigma,\tilde{\sigma})$ on each cluster:
	\[
		\w = \w^{++} \cup \w^{--} \cup \w^{+-} \cup \w^{-+}.
	\] 
	It is now immediate that $\w^{++}$ say satisifies the FKG-inequality since it is a directly increasing function of the triple $(\sigma,\tilde{\sigma},\eta)$. With this approach, it is possible to simply derive the FKG properties of each subset as well as various correlations between them. One can continue on with this decomposition instead and still prove all of our results (see Remark \ref{remark:altIIC} for instance). In the end, it was simpler to work with the decomposition $(\w^+,\w^-)$ which had already appeared in \cite{GL}, but we believe that this other approach (together with inequalities such as Lemma \ref{lemma:decorrlemma}) may be useful in other situations e.g. studying the dual of a single random current. 
\end{remark}

\subsection{RSW Theory}\label{sec:RSW}

In this section, we state and prove the minimal RSW theory needed to prove Theorem \ref{thm:boxcount}. Let us note that much stronger crossing results for the double random current have been established in \cite{DRC2} but we do not rely on these. We also mention that RSW results were also developed in \cite{GL} for the percolations $\w^+,\w^-$ on the section of the self-dual line when $U \leq J$, which may be very useful in generalising our work in the direction of Conjecture \ref{conj:AT-mag-field}.

We introduce the notation we use that will be specific to the square lattice.
We will restrict ourselves always to \emph{domains} $\L \subset \Z^2$ that are a finite connected union of faces of $\Z^2$. 
We take $\b\L$ to be the set of vertices $u \in \L$ incident to a vertex $v \not\in \L$.
We further impose that the boundary of such a domain $(\bL,E(\bL)$ has the property that it is always connected. 
We say that $\G$ is a \emph{circuit} of $\Z^2$ if it is the possible edge-boundary of such a domain (i.e. a vertex-self-avoiding closed path of edges in $\Z^2$) and we denote by $ \overline{\G}$ the unique domain with $E(\b \overline{\G}) = \G$.
The most common examples of domains we will use are the boxes $\L_n$, defined for each $n \geq 1$ by
\[
\L_n = \left\{ (x,y) \in \Z^2 \mid |x|,|y| \leq n-1 \right\}.
\] 

Unless otherwise stated, the measure $\P_\L^+$ will always refer to the coupling on the graph $(\L,E(\L))$ with $+$ boundary conditions on $\b\L$ and the coupling constants fixed to take their critical value given in \eqref{eq:critical-Ising}.

Before turning to crossing estimates we will need two further properties of the measure $\P^+_\L$. The first is perhaps the most important tool in the study of planar percolation.
\begin{itemize}
	\item (\textbf{Markov property}) Let $\G$ be any circuit in $\b\L$. Under the conditional law
		\[
			\P^+_\L(\cdot \mid \G \text{ is }\w\text{-open}),
		\]
		the triplet $(\sigma,\tilde{\sigma},\w)$ restricted to the vertices and edges of $ \overline{\G}$ is given by $\P_{ \overline{\G}}^+$, except that possibly the global sign of $\sigma$ and $\sigma'$ must be flipped (which does not affect $\w$) to ensure they are $+/+$ on $\G$.
\end{itemize}
In particular, the the state of $\w|_{ \overline{\G}}$ and $\w|_{\L \setminus \overline{\G}	}$ are, conditionally on $\G$ being $\w$-open, independent. We will often use this to explore an circuit from the outside so that no information about inside the circuit is revealed. Note that (unlike $\phi_G$) there is no Markov property with respect to $\w$-closed cutsets (e.g. after discovering the entirety of a $\w$-cluster).
The other property we will need is a spatial mixing estimate. We will use this only in the proof of Lemma \ref{lemma:dualonearm}.
\begin{itemize}
	\item (\textbf{Spatial mixing}) There are constants $c, C > 0$ such that for any $n \geq 1$, $\L \supset \L_{2n}$ and events $A$ (resp. $B$) depending only on the state in $\L_n$ (resp. $\L \setminus \L_{2n}$) we have
		\begin{equation}\label{eq:mixing}
		c \times \P^+_\L(A) \P^+_\L(B) \leq \P^+_\L(A \cap B) \leq C \times \P^+_\L(A) \P^+_\L(B).	
		\end{equation}
\end{itemize}
To prove these two properties we take a common approach and rely on analagous properties of the triple $(\sigma,\tilde{\sigma},\eta)$ and then extend them to $\w$ using the simple relation $\w = \xi(\sigma) \cap \xi(\tilde{\sigma}) \cap \eta$.
For the first property, we can use the Markov property of the Ising model together with the observation that if $\G$ is $\w$-open then necessarily both $\sigma$ and $\tilde{\sigma}$ are constant on $\G$. 
Similarly, the second property immediately reduces to a mixing estimate for the critical Ising model which are well known (see e.g. \cite{tassion2025noise} or \cite{duminil2022planar,RSW} for mixing statements for the FK-Ising model).	

\begin{remark}
	We remark that a similar strategy can be used to prove mixing statements for the (sourceless) critical random current $\n_1$ on the square lattice: It suffices to prove mixing for $(\n_1)_\odd$ (since $(\n_1)_\even$ is a sprinkling on top) and this may be derived by the identification $(\n_1)_\odd = \b\sigma$ and mixing of the Ising model.
	Other stronger mixing properties of the planar random current model are proved in \cite{DRC2}.
\end{remark}
\begin{remark}
	We could improve the result to a polynomial-mixing statement since this is also known for the critical Ising model (see \cite{duminil2022planar,RSW} for analagous proofs for the FK-Ising model), though we have no use for this improvement in this paper. 
\end{remark}

We define $A(n,N)$ to be the event that there is an open circuit encircling the annulus $\L_N \setminus \L_n$. 
It is known (it can be derived from the results of \cite{DRC2}) that
\begin{equation}\label{eq:exactonearm}
	\P_{\Z^2}\left( \w \not\in A(n,N)  \right) = \P^{\emptyset, \emptyset}_{(\Z^2)^*}\left( \L^\dual_n \overset{\n}{\longleftrightarrow} \b\L_{N}^\dual \right) \asymp \left( \frac{n}{N} \right)^{1 / 4}.
\end{equation}
Notice that here $1 / 4$ is the one-arm exponent of the double random current, whilst the one-arm exponent of $\w$ is instead $1 / 8$ owing to its connection with the Ising model. The double random current is then an interesting example of a (critical) planar model which is not self-dual, and moreover, whose one-arm and dual one-arm exponents \emph{differ}. We do not rely on \eqref{eq:exactonearm} or any other results from \cite{DRC2} in this work.

Instead, we state and prove the following result which we will then bootstrap to a polynomial (but non-explicit) bound on the probability of there not existing any circuits in $\w^+$ or $\w^-$ in a large annulus. This dual one-arm exponent of $\w^+$, which we do not know how to identify explicitly (see Remark \ref{remark:1arm}), will be the main error term in Theorem \ref{thm:boxcount}.
In light of Lemma \ref{lemma:decorrlemma} and the lack of an FKG inequality for $\w$, it will also be very useful to derive a bound on the dual one-arm exponent for $\w^-$ as well.

\begin{lemma}[Existence of circuits] \label{lemma:onecircuit}There exists a constant $c > 0$ such that for any $n \geq 1$ and any domain $\L$ containing $\L_{2n}$,
	\[
		\P_{\L}^+(\w^+ \in A(n,2n)) \geq c.
	\] 
	Moreover, $\P_{\L}^+(\w^- \in A(n,2n)) \geq c$ whenever $\L \supset \L_{4n}$.
\end{lemma}
\begin{proof}
	Observe that the event $A(n,2n)$ is primitive. Thus,
	\[
		\P_{\L}(\w^+ \in A(n,2n)) \geq \frac{1}{2} \P_{\L}(\w \in A(n,2n)).
	\]
	Let $\P^\f_{\L_{2n} \setminus \L_n}$ be the free measure on the domain $\L_{2n} \setminus \L_n$. 
	It follows from known monotonicity properties of the double random current \cite[Lemma 3.2]{DRC2} that 
	\[
		\P_{\L}^+(\w \in A(n,2n)) \geq \P^\f_{\L_{2n} \setminus \L_n}(\w \in A(n,2n)).
	\] 
	Finally, there is an exact identity relating this probability to the free-boundary FK-Ising measure $\phi^\f_{\L_{2n} \setminus \L_n}$,
	\[
		\P^\f_{\L_{2n} \setminus \L_{n}}(\w \in A(n,2n)) = 1 - (1 - \phi^\f_{\L_{2n} \setminus \L_n}(A(n,2n)))^2 \geq \phi^\f_{\L_{2n} \setminus \L_{n}}(A(n,2n)).
	\] 
	Indeed, from the duality relation between $\w$ and $\n$,
	\[
		1 - \P^\f_{\L_{2n} \setminus \L_n}(\w \in A(n,2n)) = \P^{\emptyset,\emptyset}_{(\L_{2n} \setminus \L_{n})^*}\left( \L_n \overset{\n}{\longleftrightarrow} \L_{2n}^c \right) 
	\] 
	where $(\L_{2n} \setminus \L_n)^*$ is the full (i.e. not weak) dual and the event is that $\n$ connects the inner face to the outer face. Then, by writing the connection event as a spin correlation(see \ldots ) and using the self-duality of the FK-Ising model, we have
	\[
	\P^{\emptyset,\emptyset}_{(\L_{2n} \setminus \L_{n})^*}\left( \L_n \overset{\n}{\longleftrightarrow} \L_{2n}^c \right)  = (1 - \phi^\f_{\L_{2n} \setminus \L_n}(A(n,2n)))^2.  
	\] 
	The result is then immediate from the RSW property of the FK-Ising model \cite{RSW}. The statement for $\w^-$ follows by applying the mixing result in \eqref{eq:mixing} which is why we impose the boundary of $\L$ not to be too close.
\end{proof}

\begin{remark}
	It is possible to prove a corresponding upper bound whenever $\bL$ is sufficiently far from $\b\L_{2n}$. The subtelty here is that the planar double random current does not satisfy RSW results which are uniform in the domain. In particular, the probability 
	\[
		\P_{\L}^{\emptyset,\emptyset}(\L_n \overset{\n}{\longleftrightarrow} \b\L_{2n})
	\] 
	is \emph{not} uniformly positive for all $n \geq 1$ when $\L = \L_{2n}$, but is when $\L \supset \L_{(2+\eps)n}$ for some fixed $\eps > 0$ (i.e. only when the free-boundary is macroscopically far away) \cite{DRC2}. This is also the expected behaviour of the random cluster model with $q = 4$, since the scaling limit of the outer boundaries of large clusters in both models is the $\CLE_4$, whose loops are known not to touch the boundary of the domain.
\end{remark}

\begin{lemma}[Dual one-arm exponent] \label{lemma:dualonearm}
There exists $\alpha, C > 0$ such that for all $\zeta \in \left\{ \pm 1 \right\}$, $N \geq n \geq 1$ and any $\L \supset \L_N$,
\[
	\P_\L^+\left( \w^{\zeta} \not\in A(n,N) \right)  \leq C \left( \frac{n}{N} \right)^{\alpha}.
\] 
\end{lemma}
\begin{proof}
	The argument is standard and consists of applying Lemma \ref{lemma:onecircuit} and the spatial mixing property of $\P_\L^+$ to show that the set of scales on which there is a $\w^\pm$-circuit stochastically dominates a point process of uniformly positive density. 
\end{proof}

\begin{remark}\label{remark:1arm}
	It is not difficult to derive the required quasi-multiplicativity to prove the existence of the dual one-arm exponent $\a$ for $\w^+$. The exact value of $\a$ is then an interesting question, not least because it represents the main error term in Theorem \ref{thm:boxcount}. Using the relation with $\tau$ and the FKG properties (Proposition \ref{prop:FKG}), one can show that it must satisfy $\a \leq 1 / 8$. If the value $\alpha = 1 / 8$ were to be true, then the model $\w^+$ would be `approximately self-dual' in the sense that the one-arm and dual one-arm exponents would coincide. However, it is generally the case that there are polynomial improvements to applications of the FKG inequality for arm-exponents \cite{beffara2011monochromatic,radhakrishnan2024strict,gassmann2024comparison}, suggesting that possibly $\alpha < 1 / 8$. 
\end{remark}

Suppose that for some integer $n \geq 1$ and two vertices $a,b \in \L$ are such that the two boxes $a + \L_{2n}$ and $b + \L_{2n}$ are both contained in $\L$ and are disjoint. In this case, we define
\[
A = a + \L_n \quad \text{ and } \quad B = b + \L_n,
\] 
two boxes of size $n \geq 1$. The next proposition is concerned with the probability of the connection event $\left\{ A \leftrightarrow B \right\}$. 
Its proof is a standard application of the FKG inequality and the bounds on existence of circuits above. It will allow us to take sharp estimates on site-to-site connection probabilities (derived from \eqref{eq:sokalsitetosite}) and translate them into estimates for box-to-box connections. We write $X \lesssim Y$ if there is a universal constant $C > 0$ such that $X \leq CY$ and $X \asymp Y$ if $X \lesssim Y$ and $Y \lesssim X$.

\begin{prop}[Box-to-box connection estimate]\label{prop:boxtobox} There exists a constant $C > 0$ such that the following is true. Let $\L$ be a domain containing two boxes $A$ and $B$ satisfying the conditions above. Then,
	\[
		\P_{\L}\left( A \leftrightarrow B \right) \leq C n^{1 / 4 - 4} \E_{\L} \left[ \sum_{\substack{x \in A, y \in B}} \sigma_x \sigma_y \right].
	\] 
\end{prop}
\begin{proof}
	We define also
	\[
		A' = a + \L_{2n} \quad \text{ and } \quad B' = b + \L_{2n},
	\] 
	Let $\G_A$ and $\G_B$ be the outermost $\w^+$-open circuits in the annuli $A' \setminus A$ and $B' \setminus B$ respectively (if they exist). Let $C_{A,B}$ be the event that $\G_A$ and $\G_B$ do both exist and that
	\[
		\G_A \overset{\w^+}{\longleftrightarrow} \G_B.
	\]
	It follows from \eqref{eq:primitive}, the FKG inequality and Lemma \ref{lemma:onecircuit} that
	\[
		\P( A \leftrightarrow B) \leq 2\P(A \overset{\w^+}{\longleftrightarrow} B) \lesssim \P(C_{A,B}). 
	\] 
	By conditioning on the state of $\w$ outside $\G_A$ and $\G_B$,
	\[
		\E\left[ \sum_{x \in A, y \in B} \i(x \overset{\w^+}{\longleftrightarrow} y) \right] \geq \E \left[ \i(C_{A,B}) \E\left[ \sum_{x \in A} \i(x \overset{\w^+}{\longleftrightarrow} \G_A) \mid \G_A \right] \E \left[ \sum_{y \in B} \i(y \overset{\w^+}{\longleftrightarrow} \G_B) \mid \G_B \right] \right].
	\] 
	Then, by the FKG property of the boundary cluster (or in this case simply by monotonicity of Ising spin correlations in the domain),
	\[
		\E\left[ \sum_{x \in A} \i(x \overset{\w^+}{\longleftrightarrow} \G_A) \mid \G_A \right] \geq \E^+_{\L_{2n}}\left[ \sum_{x \in \L_n} \sigma_x \right] \gtrsim n^{2 - 1/8},
	\] 
	where the last estimate is contained in Lemma \ref{lemma:Ising2}. Applying the same for the box $B$ and rearranging all the terms leaves the desired inequality.
\end{proof}

\subsection{Construction of the IIC measure}\label{sec:IIC}

In this section, we construct the \emph{incipient infinite cluster} (or IIC) measure of $\w$, by which we mean we prove that in an appropriate the sense the limit 
\[
	\lim_{N \to \infty} \P_{\L_N}^+( \cdot \mid o \overset{\w}{\longleftrightarrow} \b \L_N)
\]
exists. This a classical problem in percolation theory first solved by Kesten \cite{kesten1986incipient} in the case of Bernoulli percolation. Results of this kind for dependent percolation models are not extensive but there has recent progress in a variety of models \cite{panis2024incipient}. 
The construction of a suitable IIC measure is an integral aspect to the $L^2$ argument introduced in \cite{GPS}. It is used to argue that the behaviour of macroscopic clusters, when viewed at a mesoscopic scale, concentrates sharply around the IIC measure. It is this concentration which will play the key role in the proof of Theorem \ref{thm:boxcount}.  
In \cite{GPS}, the analysis is mostly focused on constructing the pivotal measure of Bernoulli site percolation. For the pivotal measure, the analagous input is to construct an infinite-volume measure conditioned on an alternating \emph{four-arm} event. 
There, a scale-by-scale coupling argument was used to prove the existence of such a measure. The analysis is intricate and requires certain seperation of arm estimates, however, a simpler approach for the case of a single arm (which is relevant to the cluster measures) was also suggested \cite[Proposition 5.2]{GPS}. Whilst it is possible to adapt this to $\w$ using the RSW theory in Section \ref{sec:RSW}, we take another approach.
We utilise an argument recently introduced in \cite{hillairet2025short}. It uses a simple and explicit coupling to prove a quantitative convergence towards the IIC measure based on the dual-one-arm exponent. The argument there was written for Bernoulli site percolation but it immediately extends to any monotonic measure on bond configurations. After overcoming a slight difficult due to not knowing $\w^+$ is monotonic, our argument will follow \cite{hillairet2025short} almost exactly. The result we will prove is the following. 

\begin{prop}[Existence of IIC measure]\label{prop:IIC}
	There exists $C,\alpha > 0$ such that the following holds. For each $N \geq k \geq n \geq 1$ and each domain $\L \supset \L_N$, the measures
	\[
		\w^{(1)} \sim \P^+_{\L_N}(\cdot \mid \L_n \overset{\w}{\longleftrightarrow} \L_N) \quad \text{and} \quad \w^{(2)} \sim \P^+_{\L}(\cdot \mid \L_n \overset{\w}{\longleftrightarrow} \L)
	\]
	may be coupled so that with probability $1 - C (\frac{k}{N})^{\alpha}$ they satisfy:
	\begin{itemize}
		\item The restriction of $\w^{(1)}$ and $\w^{(2)}$ agree inside $\L_k$ and furthermore, the restrictions of the boundary clusters $\C_0(\w^{(1)})$ and $\C_0(\w^{(2)})$ agree inside $\L_k$.
	\end{itemize}
\end{prop}

As a consequence of Proposition \ref{prop:IIC}, we may speak of the limiting measure\footnote{We note that other definitions of the IIC (which we will not use) can be easily shown to exist and coincide with this one by combining this result with the RSW estimates in Section \ref{sec:RSW}.} as $N \to \infty$,
\[
	\P^+_{\Z^2}(\cdot \mid \L_n \overset{\w}{\longleftrightarrow} \infty).
\]
In Section \ref{sec:conv-meas}, we will use this measure to define a quantity $\beta$ which should be interpreted as the `density' of the infinite cluster $\C_0$ under this measure.
To be more precise, we expect that since the infinite cluster $\C_0$ is conditioned to intersect the box $\L_n$, then $|\C_0 \cap \L_n|$ will be of the order of $n^{2 - 1 / 8}$. 
In fact, we will prove that
\[
	\E^+_{\Z^2}\left[ |\C_0 \cap \L_n| \mid \L_n \leftrightarrow \infty  \right] = (\beta + o(1)) n^{2 - 1 / 8},
\] 
for some constant $\beta > 0$.
The existence of this limit will use the following lemma, which itself is a simple consequence of Proposition \ref{prop:IIC}.

\begin{lemma}[Cluster-box intersection estimate]\label{lemma:Xnestimate}
	There exist constants $C, \a > 0$ such that for any $N' \geq N \geq n \geq 1$,
	\[
		\left| \E^+_{\L_{N'}}[|\C_0 \cap \L_n| \mid \L_n \leftrightarrow \b \L_{N'}] - \E^+_{\L_N}[|\C_0 \cap \L_n| \mid \L_n \leftrightarrow \b \L_N] \right|  \leq C n^{2- 1 / 8} ( \frac{n}{N})^{\alpha / 2}.
	\] 
	\smallskip
\end{lemma}
\begin{proof}
	By using the coupling of Proposition \ref{prop:IIC} and the Cauchy-Schwarz inequality, we deduce 
	\begin{multline}
		| \E_{\L_{N'}}[|\C_0 \cap \L_n| \mid \L_n \overset{}{\longleftrightarrow} \b\L_{N'}] - \E_{\L_N}[|\C_0 \cap \L_n| \mid \L_n \overset{ }{\longleftrightarrow} \b \L_N] | 
		\\
		\lesssim (\frac{n}{N})^{\a_1 / 2} \sqrt{ \E_{\L_{N'}}[|\C_0 \cap \L_n|^2 \mid \L_n \overset{}{\longleftrightarrow} \b\L_{N'}] + \E_{\L_{N}}[|\C_0 \cap \L_n|^2 \mid \L_n \overset{ }{\longleftrightarrow} \b\L_N]}.
	\end{multline} 
	Then, by the FKG property of the boundary cluster (Proposition \ref{prop:FKG}), we can compare to when the boundary conditions are wired on $\b\L_n$, 
	\[
		\E_{\L_N}\left[ |\C_0 \cap \L_n|^2 \mid \L_n \leftrightarrow \b\L_N \right] \leq \E_{\L_n} \left[ |\C_0 \cap \L_n|^2 \right] \leq \sum_{x,y \in \L_n} \P_{\L_n}(x \leftrightarrow y),
	\] 
	and so we can use Lemma \ref{lemma:Ising1} to conclude.
\end{proof}

Let us now explain the details of the coupling argument which is based on an exploration algorithm. The basic strategy is simple to understand but formalising the details will be a little fiddly. We recommend the exposition in \cite{hillairet2025short} and also \cite[Proposition 2.6]{duminil2022planar} for an argument of a similar flavour. The main additional complication to these proofs is that we do not know $\w^+$ is monotonic so must instead use Lemma \ref{lemma:conFKG} for the pair $(\tau,\w^+)$. 

\begin{figure}[t]
	\centering
	\includegraphics[width=0.5\linewidth]{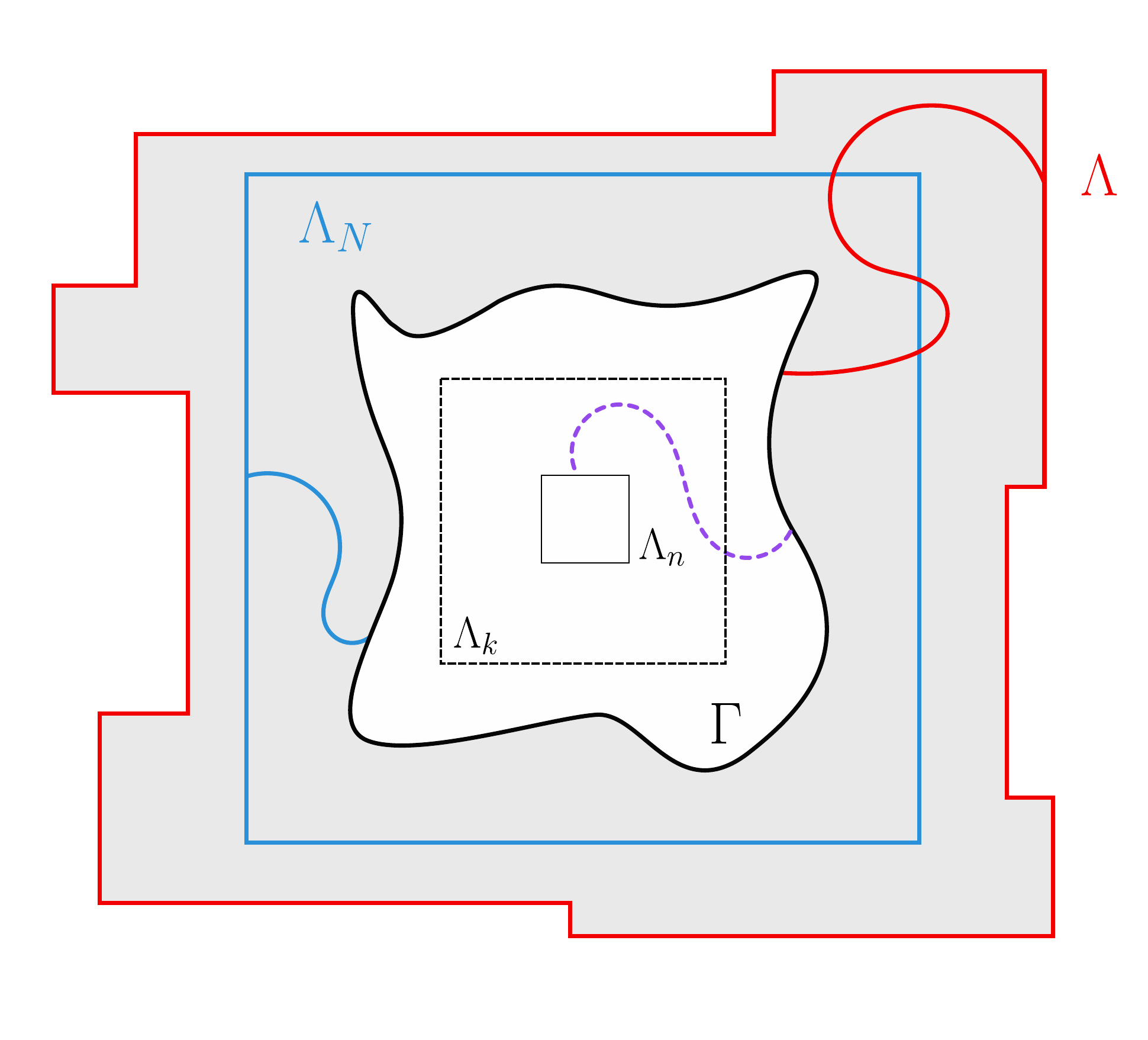}
	\caption{A schematic of the proof of Proposition \ref{prop:IIC}. The exploration of $X^{(3)}$ (blue) and exploration of $X^{(1)}$ (red) is shown up until the outermost $(\w^+)^{(3)}$-open circuit $\G$ (black) in $\L_N \setminus \L_k$ is discovered. At this time, the state of each configuration is known only on the shaded region. The purple dashed arm signifies that from this time, the two configurations $X^{(1)}$ and $X^{(2)}$ can be sampled to coincide inside $\Gamma$. }
	\label{fig:IIC}
\end{figure}

We fix $N \geq k \geq n \geq 1$ as in the statement. A first observation is that for each domain $\L$ and under the measure $\P_\L^+$, the event $\L_n \overset{\w}{\longleftrightarrow} \b\L$ is both measurable with respect to $\w^+$ and an increasing event for $\w^+$.  We will let 
\[
X = (\tau,\w^+)
\] 
and focus solely on the law of $X$ for the time being. We will actually couple not two but three configurations $X^{(1)},X^{(2)},X^{(3)}$ with the marginals
\begin{align*}
	X^{(1)} = (\tau^{(1)}, (\w^+)^{(1)}) &\sim \P^+_{\L_N}( \cdot \mid \L_n \overset{\w}{\longleftrightarrow} \L_N) \\
	X^{(2)} = (\tau^{(2)},(\w^+)^{(2)}) &\sim \P^+_{\L}( \cdot \mid \L_n \overset{\w}{\longleftrightarrow} \L) \\
	X^{(3)} = (\tau^{(3)},(\w^+)^{(3)}) &\sim \P^+_{\L}( \cdot).
\end{align*}
Actually, we will sample $X^{(1)}$ according to the equivalent measure
\[
	\P^+_{\L}(\cdot \mid \L_n \overset{\w}{\longleftrightarrow} \L_N , \w^+_{E(\b\L_N \cup \L\setminus\L_N)} = 1)
\] 
so that it is extended to the domain $\L$ and can still be viewed as the law $X^{(3)}$ conditioned on an increasing event.
The coupling will have the property that
\begin{equation}\label{eq:xiorder}
	X^{(1)} \geq X^{(3)} \quad\text{and}\quad X^{(2)} \geq X^{(3)}.
\end{equation}
The existence of such a coupling is guaranteed by the FKG inequality (Proposition \ref{prop:FKG}). To give an explicit coupling we will explore the simultaneuous state of $X^{(1)},X^{(2)},X^{(3)}$ starting from the boundary $\b \L$ and proceeding inwards.
That we have the desired ordering \eqref{eq:xiorder} will follow from inductively applying Lemma \ref{lemma:conFKG}. 
They key step will be to to ensure with high probability $X^{(1)}$ and $X^{(2)}$ agree exactly before the box $\L_k$ is reached, and so our coupling is as desired.

Let us now say a bit more about this key step. The plan will be to find a circuit encircling the annulus $\L_N \setminus \L_k$ which is entirely $(\w^+)^{(3)}$ open, which by Lemma \ref{lemma:dualonearm} will happen with high probability. Due to \eqref{eq:xiorder}, it must also be open in both $(\w^+)^{(1)}$ and $(\w^+)^{(3)}$. 
Then, we can argue that if such a circuit $\Gamma$ is discovered from the outside to be open in both $(\w^+)^{(1)}$ and $(\w^+)^{(2)}$, then, the law of $X^{(1)}$ and $X^{(2)}$ inside $\Gamma$ is precisely the same. 
Indeed, together with the Markov property of $\P_\L^+$, there is a \emph{decoupling} of the conditioning upon finding $\Gamma$ so that is now equivalent to $\L_n \leftrightarrow \Gamma$. 
Thus, the coupling can be completed whilst ensuring that $\w^{(1)}$ and $\w^{(2)}$ agree exactly inside $\Gamma$. 

We now spell out exactly how Lemma \ref{lemma:conFKG} can be used to explore the configurations whilst preserving the order. We focus our attention on just showing $X^{(2)} \geq X^{(3)}$ with the other case treated similarily. Suppose that we have explored the state of both $\tau^{(2)}$ and $\tau^{(3)}$ on a common set of vertices $S$ and also the state of $(\w^+)^{(2)}$ and $(\w^+)^{(3)}$ on a common set of edges $H$ with $H \subset E(S)$. As our inductive hypothesis, we may assume that so far the ordering is as desired,
\begin{equation}\label{eq:IH}
	\tau^{(2)}_S \geq \tau^{(3)}_S \quad\text{and}\quad (\w^+)^{(2)}_H \geq (\w^+)^{(3)}_H.
\end{equation} 
We will now sample the state of an edge $e \notin H$. For simplicitly, we will assume that already  $H \cup \left\{ e \right\} \subset E(S)$ (see Stage 2 of the algorithm otherwise).
To continue the coupling whilst maintaining $X^{(2)} \geq X^{(3)}$, we need to show that the conditional probability (given what we have observed so far) of the edge being open in $(\w^+)^{(2)}$ is at least that of it being open in $(\w^+)^{(3)}$. 
Indeed, we have
\begin{align}\label{eq:weone}
	\P_{\L}^+( \w^+_e =1  \mid \L_n \overset{\w^+}{\longleftrightarrow} \b\L, \tau_S &= \tau^{(2)}_S,\w^+_H =  (\w^+)^{(2)}_H) \nonumber \\
	\geq \P_{\L}^+( \w^+_e =1  \mid \tau_S &= \tau^{(2)}_S, \w^+_H = (\w^+)^{(2)}_H) \nonumber \\
	\geq \P_{\L}^+( \w^+_e =1  \mid \tau_S &= \tau^{(3)}_S, \w^+_H = (\w^+)^{(3)}_H). 
\end{align} 
Here, we have applied both aspects of Lemma \ref{lemma:conFKG} (using our assumption on $S$ and $H$). In the first inequality, we used that the conditional measure 
\[
	\P^+_{\L}( X \in \cdot  \mid \tau_S = \tau^{(2)}_S, \w^+_H = (\w^+)^{(2)}_H)
\] 
has the FKG property and that the events $\left\{ \w^+_e = 1 \right\}$ and $\{ \L_n \overset{\w^+}{\longleftrightarrow} \b\L \}$ are increasing. In the second, we applied the ordering of the conditional measures and \eqref{eq:IH}.

\begin{remark}
	The introduction of the third configuration and the stochastic domination is a particularly nice approach to circumvent the incomplete argument given in \cite[Lemma 2.9]{mag-field}, where it is claimed that one can directly order the two conditioned measures in different domains. We thank Christophe Garban for discussions on this point.  
\end{remark}

\begin{remark}
	As mentioned in \cite{hillairet2025short}, this strategy is only suited to the planar case where `open surfaces' which decouple the event exist with high probability. In dimension $d \geq 4$, it is possible to use the coupling in Definition \ref{def:two} and mixing results for random currents \cite{AizDC} to construct a version of the IIC (this time the limit $\P_{\Z^d}(\cdot \mid o \overset{\w}{\longleftrightarrow} x)$ as $|x| \to \infty$) in a similar manner to \cite{panis2024incipient}. 
\end{remark}

\begin{remark}[]\label{remark:altIIC}
	In the spirit of Remark \ref{remark:altFKG}, we now sketch an alternative approach. The event $\L_n \leftrightarrow \b \L$ must be realised by the configuration $\w^{++}$. In particular, it is an increasing event for the triple $(\sigma,\tilde{\sigma},\eta)$ whose law we do know to be monotonic (it is the product of three positive measures satisfying the FKG-lattice condition). Thus, we could instead simulatenously explore $(\sigma,\tilde{\sigma},\eta)$ in each of the three measures in an ordered fashion. The failure rate of the coupling is then controlled by the probability that $(\w^{++})^{(3)} \notin A(k,N)$ which is larger but can still be easily shown to decay polynomially.
\end{remark}

\begin{proof}[Proof of Proposition \ref{prop:IIC}]
	We now give the necessary details required to implement the strategy described above.
	We can consider only the task of coupling the laws on $(\tau,\w^+)$ and postpone sampling $\w^-$ to the end. 

	\noindent We work on an enlarged probability space endowed by an independent collection of random variables $(U_e)_{e \in E(\L)}, (U_v)_{v \in \L}$ taking values uniformly in $[0,1]$, and indexed by both each edge and vertex of $\L$. The entire coupling will be a deterministically constructed from this initial randomness. To ease notation we define $\mu^{(i)}$ to be the marginal law of each $X^{(i)}$ for each $1 \leq i \leq 3$. 
	To describe the state of the algorithm we will use a sequence $(S_t,H_t)_{t \geq 0}$, where $S_t \subset \L$ and $H_t \subset E(\L)$ for each $t \geq 0$. We explicitly set 
	\[
	S_0 = \b\L \quad \text{ and } \quad H_0 = E(S_0).
	\] 
	It will be important that we will always have $H_t \subset E(S_t)$. We now describe each step of the algorithm. At each stage, we will have discovered the state of $\tau^{(i)}$ only on $S_t$ and the state of $(\w^+)^{(i)}$ only on $H_t$. 

	\noindent Let $t \geq 0$. Each step of the algorithm will be split into three stages. The first stage will be to choose the next edge $e_{t+1}$ to add to $H_t$. We would like to do this in a manner so that we will first completely reveal the outermost $(\w^+)^{(3)}$-open circuit in $\L_N \setminus \L_n$ before entering the circuit. It is best if the reader first convinces themselves that this is certainly possible (one informal but natural strategy is to explore the clusters of $(\w^+)^{(3)}$ using the `right hand rule'), but we formalise one option below.

	\noindent \textbf{Stage 1.}  We first choose $e_{t+1}=uv \in E(\L) \setminus H_t$ (if the set is empty then the algorithm is considered over) so that at least one endpoint (say $u$) belongs to $S_t$ according to the following rule:
	We choose such an edge $e_{t+1}$ in any manner except for prioritising (i.e. always selecting first) any edge in which at least one endpoint of the dual edge $e_{t+1}^*$ is connected to $\L^* \setminus \L_N^*$ in the set 
	\[
		((\w^+)_{H_t}^{(3)})^* = \left\{ e^* \in E(\L^*) \mid e \in H_t, (\w^+)^{(3)}_{e} = 0 \right\}.
	\]
	Note that this condition is always satisified for any edge in $E(\L) \setminus E(\L_N)$. 
	We set $H_{t+1} = H_t \cup \left\{ e \right\} $.

	\noindent \textbf{Stage 2. } This stage is only relevant if the other endpoint $v$ of $e_{t+1}$ did not already belong to $S_t$, otherwise we can skip directly to the next stage. We set $S_{t+1} = S_t \cup \left\{ v \right\}$ and sample 
	\[
		\tau^{(i)} = 
		\begin{cases}
			+1 \quad\text{if}\quad U_{v} \leq \mu^{(i)}( \tau_v = +1 \mid \tau_{S_t} = \tau^{(i)}_{S_t}, \w^+_{H_t} = (\w^+)^{(i)}_{H_t}),\\
			-1 \quad\text{else.}
		\end{cases}
	\] 
	for each $1 \leq i \leq 3$. We now see that given $\tau^{(j)}_{S_t} \geq \tau^{(3)}_{S_t}$ and $(\w^+)^{(j)}_{H_t} \geq (\w^+)^{(3)}_{H_t}$, then we must have
	\[
		\tau^{(j)}_v \geq \tau^{(3)}_v
	\] 
	for each $j \in \left\{ 1,2 \right\}$. Since $H_t \subset E(S_t)$, this follows from using Lemma \ref{lemma:conFKG} to argue as in \eqref{eq:weone}. 
	Note that if $j=1$ and $v \in \L \setminus \L_N$, then $\tau^{(1)}_v = +1$ and there is nothing to prove. 

	\noindent \textbf{Stage 3. } We set 
	\[
		(\w^+)^{(i)}_e = 
		\begin{cases}
			1 \quad\text{if}\quad U_e \leq \mu^{(i)}( \w^+_e = 1 \mid \tau_{S_{t+1}} = \tau^{(i)}_{S_{t+1}}, \w^+_{H_t} = (\w^+)^{(i)}_{H_{t}}),\\
		0 \quad\text{else.}
		\end{cases}
	\] 
	for each $1 \leq i \leq 3$. We know that $H_t \subset E(S_{t+1})$ because of Stage 2 and so again we may again argue as in \eqref{eq:weone} to deduce that 
	\[
		(\w^+)^{(j)}_e \geq (\w^+)^{(3)}_e. 
	\]
	As before, if $j = 1$ and $e \notin E(\L_N)$ then $(\w^+)^{(1)}_e = 1$ and there is nothing to prove.
	We can proceed to a new step of the algorithm, safe in the knowledge that we still have the ordering 
	\[
		(\tau^{(j)}_{S_{t+1}}, (\w^+)^{(j)}_{H_{t+1}}) \geq (\tau^{(3)}_{S_{t+1}}, (\w^+)^{(3)}_{H_{t+1}}).
	\] 
	
	\noindent We suppose now that the algorithm has finised. Firstly, that each $X^{(i)}$ has the desired marginal may be easily checked (see \cite[Lemma 2.1]{duminil2019sharp} for example) and so we have constructed a bona fide coupling. We now consider whether the coupling has succeeded or not. We introduce the event $G$ that there exists $\Gamma$ encircling the annulus $\L_N \setminus \L_n$ and $\Gamma$ is entirely $(\w^+)^{(3)}$ open. For the outermost such circuit $\Gamma$, there must exist a time $T \geq 0$ such that $\Gamma \subset H_{T}$ but
	\[
		V(\overline{\Gamma}) \cap S_T = V(\Gamma) \quad \text{ and }\quad E(\overline{\G}) \cap H_T = \Gamma.
	\] 
	That is, at time $T$ the algorithm had not explored the state of $\tau^{(i)}$ (and so $(\w^+)^{(i)}$) inside the circuit $\Gamma$. 
	If the circuit is $(\w^+)^{(3)}$ open then it is also open in $(\w^+)^{(1)}$ and $(\w^+)^{(2)}$.
	As argued before, it must be the case that the conditional law of \emph{both} $X^{(2)}$ and $X^{(3)}$ in $ \overline{\G}$ given the exploration at time $T$ is exactly
	\[
		\P^+_{ \overline{\G}}( \cdot \mid \L_n \overset{\w^+}{\longleftrightarrow} \G).
	\] 
	As explained before, here we have used the Markov property of the unconditioned measure $\P_G^+$ and that the existence of an open circuit $\G$ decouples the conditioning in both $\mu^{(1)}$ and $\mu^{(2)}$. Therefore, it follows (since we the random variables $U$ are common for each configuration) that we have
	\[
		X^{(1)}|_{ \overline{\G}} = X^{(2)}|_{ \overline{\G}}.
	\] 
	At the end of the algorithm, we may consider (for each configuration) the law of $\w^-$ given $X = (\tau,\w^+)$. This is simply the random-cluster measure $\phi_{\xi^-(\tau),2J}$, and so it follows immediately that on the event $G$, the conditional distributions of $(\w^-)^{(1)}$ and $(\w^-)^{(2)}$ (when projected onto $\L_k$) are identical. Therefore, they too may be sampled to coincide. 

	\noindent It only remains to say that the probability the event $G$ does not occur is given by 
	\[
		\P\left( (\w^+)^{(3)} \not\in A(k,N) \right) =  \P^+_{\L}(  \w^+ \not\in A(k,N)),
	\] 
whence we may apply the bound in Lemma \ref{lemma:dualonearm} to conclude the proof of the Theorem, noting that our coupling has all the desired properties.
\end{proof}

\section{Convergence of cluster area measures}\label{sec:conv-meas}
\def\Cdp{\mathfrak{C}^\de(\rho')}
\def\C{\mathcal{C}}
\def\Cp{\mathfrak{C}(\rho')}
\def\Cd{\C^\de}
\def\mudC{\mu_\C^\de}
\def\Q{\mathcal{Q}}
\def\A{\mathcal{A}}
\def\Ade{\mathcal{A}^\de_\eps}
\def\Ader{\mathcal{A}^\de_{\eps,r}}
\def\Cdk{\mathfrak{C}^\de_N}
\def\CCd{\mathfrak{C}^\de}
\def\N{N}
\def\Pd{\P_{\de}}
\def\Ed{\E_{\de}}
\def\ot{\frac{1}{3}}
\def\th{\frac{2}{3}}

Let $D$ be a Jordan domain. Let $\Dd\subset\Z^2$ be a discrete domain approximation of $D$ as defined in \eqref{eq:dom-approx}. For instance, a natural choice is to take $\Dd$ to be the induced subgraph of $\de\Z^2$ where every vertex corresponds to a dual face intersecting $D$. Throughout this section, we assume the domains have non-fractal boundary in the following sense:
\begin{definition}\label{def:non-fractal}
	A Jordan domain $D$ is said to be non-fractal if the number of $\eps$-boxes (i.e. squares of sidelength $\eps$) intersecting $\partial D$ is of order $O(\eps^{-1})$.
\end{definition}
\noindent Actually, the arguments we give can be changed only slightly to work with boundary dimensions strictly less $2 - 1 / 8$, which is optimal, but not necessary for our purposes.

\noindent Moreover, we only consider critical coupling constants $J_c$ as in \eqref{eq:critical-Ising} and $+$ boundary conditions, using the shorthand notation
	\[
	\P_\de\equiv\P_{\Dd}^+.
	\]
We explain the case of free boundary conditions in Remark \ref{rem:free-bc} and the case of fractal domains in Remark \ref{rem:fractal-bdy}. Recall that the discrete IMF is defined as 
	\begin{equation}\label{eq:def-ising-delta}
		\IMF_\de := \de^{2-1/8}\sum_{v\in\Dd}\sigma^\de_v\boldsymbol{\updelta}_v, 
	\end{equation}
where $\boldsymbol{\updelta}_v$ is the Dirac mass at $v$, and the field should be interpreted either as a measure or as a distribution in some negative-index Sobolev space (see Appendix \ref{app:spaces} for some background). Theorem \ref{thm:main-ES} allows us to rewrite the IMF as\footnote{Recall that the boundary cluster is deterministically assigned the label $\xi=+1$.}
	\[
		\IMF_\de = \sum_{\C^\de}\xi_{\C^\de}\mu^\de_{\C^\de},
	\]
where the sum is over all clusters $\C^\de$ of $\omega_\de$ on $\Dd$ and
	\[
		\mu^\de_{\C^\de} = \de^{2-1/8}\sum_{v\in\C^\de}\boldsymbol{\updelta}_v
	\]
is the (renormalized) discrete ``area'' measure of the cluster $\C^\de$. Abusing notation, we also write $\C^\de$ for the $\de$-enlargement of the cluster, i.e. the closed subset of $D$ obtained as the union of all dual faces corresponding to the vertices belonging to the cluster. This allows us to view clusters as random variables taking values in the collection of all closed subsets of $D$, which we can equip with the standard Hausdorff metric and the associated
Borel $\sigma$-algebra. In particular, we can write $\F_{\C}$ for the $\sigma$-algebra generated by some cluster $\C$. We refer to Appendices \ref{app:spaces}-\ref{app:conv-ext} for more details. Finally, we fix an ordering of the clusters by decreasing size of their diameter, indexing them as elements of the collection $(\C^\de_k)_{k\geq0}$. 

The goal of this section is to prove the following theorem.

\begin{theorem}[Convergence of Area Measures]\label{thm:convergencearea} 
	Let $D$ be any Jordan domain with non-fractal boundary. Fix $k \geq 0$. As $\delta\to0$, 
	\[
		(\C^\de_k, \mu^\de_{\C^\de_k}) \overset{(d)}{\longrightarrow}(\C_k, \mu_{\C_k}),
	\]
	where the clusters converges with respect to the Hausdorff metric and the measures converge weakly. Moreover, the limit measure $\mu_{\C_k}$ is a measurable function of $\F_{\C_k}$. 
\end{theorem}

The first step is to prove the existence of subsequential limits for the measures. 

\begin{prop}\label{prop:tight}
	Let $D$ be a Jordan domain with non-fractal boundary. Fix $k \geq 0$. For every $\eps>0$, there exists a compact subset $K\subset D$ such that
	\[
		\limsup_{\de\to0}\E[ \mu^\de_{\C^\de_k}(D\setminus K)^2 ]< \eps.
	\] 
	In particular, there exists a subsequential weak limit in distribution.
\end{prop}
\begin{proof}
	Choose any $K$ such that $d(K, \partial\D)\geq\eps>0$. A straightforward expansion gives that
	\begin{align*}
		\Ed [ \mu^\de_{\C^\de_k}(D\setminus K)^2 ]  &= \de^{4 - 1 / 4} \sum_{x,y \in \Dd\setminus K} \Pd(x \overset{\C^\de_k}{\longleftrightarrow} y) \\
		&\leq  \de^{4 - 1 / 4} \sum_{x,y \in \Dd\setminus K} \Pd(x \overset{\omega^\delta}{\longleftrightarrow} y)  \\
		& \leq \de^{4 - 1 / 4} \sum_{\substack{x,y \in \Dd \\ \d(x,\b\Dd), d(y, \b\Dd) \leq \eps}}\E^\de[\sigma_x\sigma_y],
	\end{align*} 
	whence we can use Lemma \ref{lemma:Ising2} to see that
	\[
		\limsup_{\de\to0}\Ed [ \mu^\de_{\C^\de_k}(D\setminus K)^2 ] = O(\eps^{2-1/4}).
	\]
\end{proof}

\def\NN{\mathbb{N}}

\subsection{Outline and approximation scheme}\label{sec:outline}

The proof of Theorem \ref{thm:convergencearea} is strongly inspired by the $L^2$-approximation argument introduced in \cite{GPS}, which was motivated by the study of the scaling of measures supported on the set of pivotal points of Bernoulli percolation. This same approach was used throughout \cite{mag-field, CME} to prove an analagous geometric decomposition of the IMF in terms of measures supported on the clusters of a nested CLE$_{16/3}$. However, the proof presented here features some key differences to \cite{mag-field, CME} which lead to an improved understanding of the limiting measures. We take the time now to highlight these differences.

In \cite{mag-field}, following the \cite{GPS}, the main idea is to introduce the notion of an \emph{annulus measure}. For a given annulus, this is the (appropriately renormalised) number of points inside the inner face of the annulus that are connected in $\w_\de$ to the exterior. The first step of the argument in \cite{mag-field} is to show that the annulus measure of a given annulus converges as $\delta \to 0$ and that the limit is measurable with respect to the scaling limit of crossing events\footnote{Note that in the paper \cite{mag-field} they work with the Schramm-Smirnov crossing topology, whereas here we work with the Hausdorff metric on the clusters, viewed as closed sets.}
in $\w_\de$.

To do so, they tile the interior of the annulus into boxes of size $\eps_1$, with $\eps_1\gg\de$, and consider a coarse-grained variant of the annulus measure: the (appropriately renormalised) number of $\eps_1$-boxes which intersect a cluster crossing the annulus, which is then multiplied by a factor $\beta(\eps_1)$. We will discuss the relevance of this factor shortly, see \eqref{eq:beta-intu}.
Once it is proved that this is indeed a good approximation as $\de, \eps_1\to0$, since the approximant depends only on mesoscopic crossing events, it follows that the limit of the annulus measures must be measurable with respect to the limit of the crossing events. It also implies that the IMF is measurable with respect to crossing events in the scaling limit, but does not yet prove a decomposition of the IMF in terms of the clusters of CLE$_{16 / 3}$.

The existence of the decomposition\footnote{Note that in \cite{mag-field} the IMF in a bounded domain is defined as we do here. However, in \cite{CME} they consider only the full-plane IMF and then restrict to bounded domains.} was finalized in \cite{CME}. For a fixed cluster $\C^\de$, the cluster measure of a set $K$ is shown to be well-approximated by considering a tiling of the set $\C^\de \cap K$ into annuli of fixed ratio and size $\eps_2$, with $\eps_2\gg\eps_1\gg\de$, and then summing the annulus measure of each annulus in the tiling. This is the content of \cite[Proposition 8.4]{CME}. 
Using this approximation and the convergence of annulus measures in \cite{mag-field}, they prove that the cluster measure itself does indeed converge and that the limit is measurable with respect to the full scaling limit of crossing events.

Our approach \emph{does not use annulus measures}. Instead, we take a more direct route. We consider only \emph{one} mesoscopic scale $\eps\gg\de$ and tile the plane into $\eps$-boxes. Given a cluster $\C^\de$ and a set $K$, we approximate its cluster measure by just the (appropriately renormalised) number of $\eps$-boxes intersecting $\C^\de \cap K$, multiplied by a factor $\beta$. We use the convergence of each cluster in the Hausdorff metric to show that this approximant converges exactly to the corresponding number of $\eps$-boxes intersecting $\C\cap K$, which is clearly measurable with respect to the limiting cluster. This is the first upshot of our approach:
\begin{itemize}
	\item We prove that the continuum are measure $\mu_\C$ is measurable with respect to $\C$ itself, not just the scaling limit of the entire collection of clusters.
\end{itemize}
This is crucial in our identification of the cluster measure as the measures constructed in \cite{CLE-meas}.
We are able to prove this because our approximant depends directly on the cluster $\C$ itself, whereas with annulus measures the underlying event is a crossing event that must be defined without reference to the cluster. 
Another improvement is the following:
\begin{itemize}
	\item We give an explicit expression for the limiting measure (recall Theorem \ref{thm:unique-meas}). 
\end{itemize}
One reason for this is that the factor $\beta$ does not depend on the scale $\eps$, although it is possible that \cite[Lemma 4.7]{GPS} could be adapted to show that $\beta(\eps)$ converges as $\eps \to 0$, which would be in effect the same. Here, we give a new proof for the existence of $\beta$ (Proposition \ref{prop:betalimit}), coming as a simple combination of the existence of the IIC measure and the convergence of single-spin correlation functions. To prove the existence of a corresponding $\beta(\eps)$ in \cite[Equation (2.23)]{mag-field}, they must adapt the `ratio-limit theorem' approach of \cite{GPS} and rely on half-plane exponents for the FK-Ising model.


Finally, we believe that our direct approach leads to the argument being simpler at many points. As evidence for this, we do not rely on any quasi-multiplicavity estimates for arm events, which are used in both \cite{mag-field, CME}. At various stages of the $L^2$ estimate, we also considerably simplify the analysis -- for instance, in handling boxes which do not satisfy a certain local uniqueness event we utilise the simple correlation inequality in Lemma \ref{lemma:decorrlemma}. 

Let us now give some more details for the $L^2$-approximation we use, especially hoping to shed some light on the crucial role of $\beta$. To simplify the exposition, we focus on applying the strategy to study the size of the boundary cluster. 
Namely, we look at
\[
	X = X^{\de} = \de^{2 - 1 / 8} |\C^\de_0|.
\]
For some $\eps\gg\de$, fix a tiling of the plane into $\eps$-boxes and decompose
\[
	X^{\de} = \de^{2 - 1 / 8} \sum_{A^\eps} |\C^\de_0 \cap A^\eps|,
\] 
where the sum runs over all boxes in our tiling. The coarse-grained approximant, which only depends on the large-scale geometry of the cluster $\C^\de_0$, is given by the number of $\eps$-boxes it intersects. That is, we set
\[
	Y = Y^{\de,\eps} = \beta \eps^{2 - 1 / 8} \sum_{A^\eps} \i(A^\eps \cap \C^\de_0),
\]
where we also scale by a fixed positive constant $\beta$. See the start of Section \ref{sec:mainest} for a discussion on why the right scaling is $\de^{2 - 1 / 8}$ and $\eps^{2 - 1/8}$. Once we show that
\[
	X^{\de} \approx Y^{\de,\eps}
\]
as $\de \to 0$ and then $\eps \to 0$, we can lift the convergence as $\de\to0$ of $Y^{\de,\eps}$, obtained from the convergence of cluster $\C^\de_0$, in order to show convergence of $X^{\de}$. 

The choice of the constant $\beta$ is crucial in order to ensure that $X$ and $Y$ are close. The obvious starting point is to choose it so that the first moments of $X$ and $Y$ coincide. That is, we would like that
\begin{equation}\label{eq:beta-intu}
	\de^{2 - 1 / 8} \E_\de\left[ |\C^\de_0 \cap A^\eps| \mid \C^\de_0 \cap A^\eps \neq \emptyset \right]  = (1 + o(1)) \beta,
\end{equation} 
for a `typical' box $A^\eps$ in our tiling. Again, we stress that this choice will come naturally given by the IIC construction in Section~\ref{sec:IIC}, see Proposition \ref{prop:betalimit}. With this choice in hand, one aims to show that
\[
	\E_\de\left[ (X^{\de} - Y^{\de,\eps})^2 \right] = o(1).
\] 
Of course, this is the heart of the proof, and the strategy will be to decompose
\[
	X^{\de} - Y^{\de,\eps} = \sum_{A^\eps} (\de^{2 - 1 / 8} |\C^\de_0 \cap A^\eps| - \beta \eps^{2 - 1 / 8} \i(\C^\de_0 \cap A^\eps) ),
\] 
and show that the terms corresponding to boxes which are far away from each other are approximately independent. Here we  will use Proposition \ref{prop:IIC} again to deduce this spatial decorrelation, arguing that the behaviour of large clusters will concentrate sharply on a mesoscopic scale around the IIC measure. 

Before rigorously setting up the approximation scheme, we deduce the correct parameter $\beta$. For convenience during the proof, we do not yet consider mesh size $\de$, but one can of course rescale everything appropriately thereafter. 

\begin{prop}[Existence of $\beta$]\label{prop:betalimit}
	Let $m \geq n \geq 1$. The limit 
	\[
		\beta = \lim_{\substack{m / n \to \infty, n \to \infty}} n^{1 / 8 - 2} \E_{\L_m}\left[ |\C_0 \cap \L_n| \big| \L_n \overset{\omega}{\longleftrightarrow} \L_m \right] 
	\] 
	exists and lies in $(0,\infty)$.
\end{prop}
\begin{proof}
	Define
	\[
		\beta(n,m) := n^{1 / 8 - 2} \E_{\L_m}\left[ |\C_0 \cap \L_n| \big| \L_n \overset{\omega}{\longleftrightarrow} \L_m \right]. 
	\] 
	Then, Proposition \ref{lemma:Xnestimate} is equivalent to the estimate 
	\begin{equation}\label{eq:betainequality}
		| \beta(n,m) - \beta(n,m') | \leq C \left(\frac{n}{m}\right)^{\alpha'},
	\end{equation} 
	for some constant $C > 0$ and $m' \geq m$. We first prove the relevant limit with the slightly modified version $\beta'(n,k) = \beta(n,kn)$ for $n,k \geq 1$. We will deduce this from two preliminary facts.
	\begin{enumerate}[(1)]
		\item The limit $k \to \infty$ of $\beta'(n,k)$ exists \emph{uniformly in $n$}.
	\end{enumerate}
	This is an immediate consequence of \eqref{eq:betainequality}.
	\begin{enumerate}[(1)]
		\setcounter{enumi}{1}
		\item For each $k \geq 1$, the limit $n\to\infty$ of $\beta'(n,k)$ exists. 
	\end{enumerate}
	To prove this, we can expand $\beta'(n,k)$ as
			\[
				\beta'(n,k) = \frac{ n^{1 / 8 - 2} \E_{\L_{kn}}\left[ |\C_0 \cap \L_n| \right] }{\P_{\L_{kn}}( \L_n {\longleftrightarrow} \partial\L_{kn})}  
			\] 
	The numerator is simply the renormalized one-point function of the Ising model, which converges as $n\to\infty$ by Theorem \ref{thm:CHI1} and a straightforward dominated convergence argument.
			A simple argument (see Lemma \ref{lemma:indicator-conv} and Appendix \ref{app:conv-ext} for details) using the convergence of the clusters of $\omega$ gives that the denominator also converges to the corresponding connection probability in the $\CLE_4$ carpet. That is,
			\[
				\lim_{n \to \infty} \P_{\L_{kn}}(\L_n {\longleftrightarrow}\b  \L_{kn}) =  \P_{S_k}\left( S_1 \cap \CLE_4 \neq \emptyset  \right),
			\] 
			where $\CLE_4$ denotes its carpet in the square domain $S_k$ of side-length $k \geq 1$.
	Using (1) and (2) together, we can apply the Moore-Osgood theorem to the double sequence $(\beta'(n,k))_{n,k \geq 1}$ to see that the limit
	\[
		\beta = \lim_{\substack{ n \to \infty \\ k \to \infty}} \beta'(n,k)
	\] 
	exists. Applying \eqref{eq:betainequality} once more, we deduce the identical limit for $\beta(n,m)$.
\end{proof}

\begin{remark}
	Convergence of the numerator is the only statement where we invoke Theorem~\ref{thm:CHI1}. Hence, as claimed, our arguments give a proof for the convergence of the discrete IMF that does not rely on the convergence of all $n$-point correlation functions. This is also the case in \cite{mag-field}.
\end{remark}

\subsection{Proof of the main estimate}\label{sec:mainest}
We rigorously prove the approximation described in the previous section. Let $\eps_n = 2^{-n}$ be the sequence of mesoscopic scales and let $\A\equiv\A_n$ be a tiling of the plane into boxes of side-length $\eps_n$, in such a way that the different approximations are compatible. We often identify each $A \in \A_n$ with its discrete counterpart $A \cap \Dd$, and assume that each $A \cap \Dd$ is identical (i.e. we ignore lattice effects). Moreover, we use the shorthand notation $A\cap \C^\de_k$ for the event that $A\cap\C^\de_k\neq\emptyset$. 

\begin{theorem}\label{thm:boxcount} Let $D$ be a non-fractal domain. Let $k \geq 0$. For any continuous function $f\in C(\bar D)$, we have that
	\[
		\lim_{n \to \infty} \limsup_{\de \to 0} \Ed \left[ \left( \mu^\de_{\C^\de_k}[f] - \beta\eps_n^{2 - 1 / 8} \sum_{A \in \A_n} f_A \i(A \cap \C^\de_k) \right)^2  \right] = 0,
	\]
	where $f_A$ is the average value of $f$ in the box $A$. 
\end{theorem}
\def\Ck{\C_k}

\noindent Before starting the proof, we briefly justify why the scaling $\eps_n^{2 - 1  / 8 }$ shows up. In particular, we show that both terms are roughly of the same order and uniformly bounded $\de \to 0$ and then $n \to \infty$. Arguing as in the proof of Proposition \ref{prop:tight},
\[
	\E_\de\left[ \mu^\de_{\C^\de_k}[D]^2 \right] \leq  \de^{4 - 1 / 4} \sum_{x,y \in \Dd} \Ed [\sigma_x \sigma_y].
\] 
Similarly, the coarse-grained version of the above inequality is that
\[
	\Ed\left[\left(\beta\eps_n^{2 - 1 / 8} \sum_{A \in \A_n} \i(A \cap \C^\de_k)\right)^2\right]\leq \beta^2 \eps_n^{4 - 1 / 4} \sum_{A,B \in \A_n} \Pd(A \cap \C^\de_k, B \cap \C^\de_k).
\] 
The important but elementary observation is that if both events $A \cap \C^\de_k$ and $B \cap \C^\de_k$ occur, then necessarily $A \longleftrightarrow B$. As such, we may apply\footnote{At least when both boxes are at distance greater than $2\eps_n$ from each other and the boundary, and the remaining boxed can be handled as we will see shortly} Proposition \ref{prop:boxtobox} to see that
\[
	\eps_n^{4 - 1 / 4} \Pd(A \cap \C^\de_k, B \cap \C^\de_k) \leq \eps_n^{4 - 1  / 4} \Pd(A \overset{\omega_\de}{\longleftrightarrow} B)  \lesssim \de^{4 - 1 / 4} \sum_{x \in A, y \in B} \Ed[\sigma_x \sigma_y].
\] 
All together, 
\begin{equation}\label{eq:O1bound}
	\Ed\left[\left(\mu^\de_{\C^\de_k}[D]^2 +  \beta \eps_n^{2 - 1 / 8} \sum_{A \in \A_n}\i(A \cap \C^\de_k) \right)^2  \right] \lesssim \de^{4 - 1 / 4} \sum_{x,y \in \Dd} \Ed [\sigma_x \sigma_y],
\end{equation} 
which is bounded as a consequence of Lemma \ref{lemma:Ising2}.

\def\M{\mathcal{M}}
\begin{proof}
	As shorthand, we write $\eps_n = \eps$ and $\P=\Pd$.
	
	\noindent \textbf{Step 1. Riemann approximation.} 
	We first reduce the claim for continuous functions $f$ to a statement solely about dyadic boxes.
	By decomposing
	\[
		\mu^\de_{\Ck^\de}[f] = \sum_{A \in \A_\eps} \mu^\de_{\Ck^\de}[f \i_A],
	\] 
	we may compare to a Riemann sum approximation and show that
	\begin{equation}\label{eq:contf}
		\E \left[ \left( \mu^\de_{\Ck^\de}[f] - \sum_{A \in \A_\eps} f_A \mu^\de_{\Ck^\de}[A] \right)^2  \right] \leq  \sup_{A \in \A_\eps} \text{osc}_A(f)^2  \E \left[ \mu^\de_{\Ck^\de}[D]^2 \right].  
	\end{equation} 
	By Lemma  \ref{lemma:Ising1}, the second factor is bounded as $\de\to0$, and the first factor vanishes as $\eps\to0$ since $f$ is (absolutely) continuous.
	Thus, it remains to analyse 
	\begin{equation}\label{eq:step1prog}
		\E \left[ \left( \sum_{A \in \A_\eps} f_A \left( \mu^\de_{\Ck^\de}[A] - \beta \eps^{2-1/8} 1(A \cap \Ck^\de) \right) \right)^2  \right]. 
	\end{equation} 
	Since $f$ is bounded, we can disregard the prefactor $f_A$ in what follows.
	
	\noindent \textbf{Step 2. Well-separated boxes.}
	We start with the simple inequality
	\begin{align}
		 (\mu^\de_{\Ck^\de}[A] - \beta \eps^{2 - 1 / 8}\i(A\cap\Ck^\de))(\mu^\de_{\Ck^\de}[B] - \beta\eps^{2 - 1 /8}\i(B\cap\Ck^\de)) \\ \label{eq:trivialestimate}
		 \leq \de^{4 - 1/4} \sum_{x \in A, y \in B} \i(x \overset{\omega_\de}{\longleftrightarrow} y) +  \beta^2 \eps^{4 - 1 / 4} \i( A \overset{\omega_\de}{\longleftrightarrow}  B),
	\end{align}
	where we have neglected the negative cross-terms. Note that such cross-terms can actually be helping us (in particular, the left-hand side above may be negative, but the right-hand side is positive), yet we only rely on these cancellations in the last step of the proof. Let
	\[
		\mathcal G^2_\eps = \left\{ (A,B) \in \A_\eps \times \A_\eps \mid \d(A,\b D), \d(B, \b D), \d(A,B) \geq 10\eps^{1/3} \right\}
	\]
	be the set of pairs of boxes that are far ``enough'' from the boundary and from each other. We now show that we can safely neglect pairs of boxes not belonging to $\mathcal G^2_\eps$, which would be problematic later on. After using \eqref{eq:step1prog}, we can apply  Lemma \ref{lemma:Ising2} and Lemma \ref{lemma:Ising3} to show that
	\[
		 \de^{4 - 1 / 4}\sum_{(A,B) \notin \mathcal G^2_\eps} \sum_{x \in A, y \in B}  \P\left( x \overset{\omega_\de}{\longleftrightarrow} y \right) \longrightarrow 0
	\] 
	uniformly in $\de$ as $\eps \to 0$. For the other term, a very simple box-counting counting argument, using the regularity of $\b D$, is enough to show that
	\[
		\eps^{4 - 1/4}\sum_{(A,B) \not\in \mathcal G^2_\eps} \P(A\overset{\omega_\de}{\longleftrightarrow}  B) \leq \eps^{4 - 1 / 4} | ( \A_\eps \times \A_\eps) \setminus \mathcal G^{2}_\eps | \longrightarrow 0
	\] 
	as $\eps \to 0$.	
	
	\noindent\textbf{Step 3. Boxes with a unique large cluster.}
	For each box $A\in\A_\eps$, we define an event $U_A$, which is a type of \emph{local uniqueness event}, and show that we may further discard any pair $(A, B)\in\mathcal G^2_\eps$ for which at least one box fails the event. Let
	\[
		U_A = \left\{ \text{ there is an $\w_\de$-open circuit in the annulus } A(\eps^{1 / 3}) \setminus A(\eps^{2 / 3}) \right\},
	\] 
	where we write $A(R)$ for the box of side-length $R$ which is concentric to $A$. On the event $U_A$, we take $\G_A$ to be the outermost such circuit. See Figure \ref{fig:L2-proof}.
	
	\begin{figure}[t]
		\centering
		\includegraphics[width=0.8\linewidth]{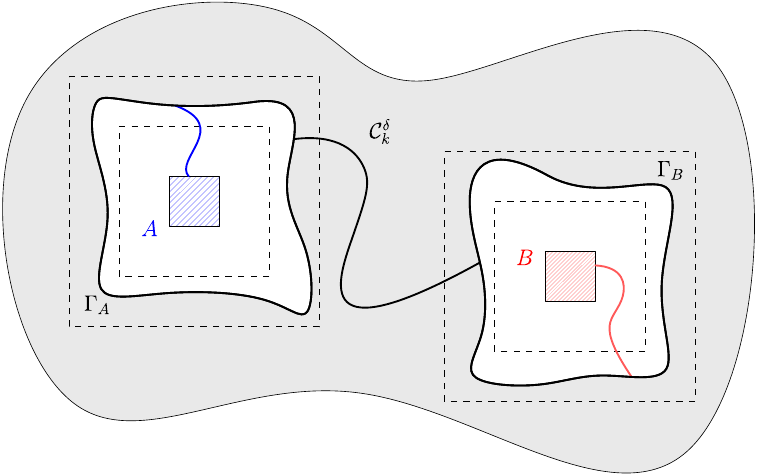}
		\caption{A pair of boxes $(A, B)\in\mathcal G_\eps^2$ satisfying the events $U_A$ and $U_B$. The dashed lines correspond to the annuli of outer radius $\eps^{1/3}$ and inner radius $\eps^{2/3}$. The events $\{A\longleftrightarrow\G_A\}, \{B\longleftrightarrow\G_B\}, \{\G_A\longleftrightarrow\G_B\}$ are also satisfied. The state of the configuration in the grey region is $\mathcal{F}_{AB}$.}
		\label{fig:L2-proof}
	\end{figure}

	\noindent Using \eqref{eq:trivialestimate} once more, 
	\begin{align}
		\E \left[\i(U_A^c)  (\mu^\de_{\Ck^\de}[A] - \beta \eps^{2 - 1 / 8}\i(A\cap\Ck^\de))(\mu^\de_{\Ck^\de}[B] - \beta\eps^{2 - 1 /8}\i(B\cap\Ck^\de))  \right] \\
		 \leq \E \left[ \i(U_A^c) \left( \de^{4 - 1/4}  \sum_{x \in A, y \in B} \i(x \overset{\omega_\de}{\longleftrightarrow} y) +  \beta^2 \eps^{4 - 1 / 4} \i( A \overset{\omega_\de}{\longleftrightarrow} B) \right) \right].
	\end{align}
	We now apply Lemma \ref{lemma:decorrlemma} to decorrelate the event $U_A^c$ with either of the primitive events $\left\{ x \longleftrightarrow y \right\}$ or $\left\{ A \longleftrightarrow B\right\} $, so that the the above expression is upper bounded by
	\begin{equation}\label{eq:postdecorr}
	2 \times \P(\w^-_\de \notin U_A) \times \left( \de^{4 - 1/4}\sum_{x \in A, y \in B} \P(x\overset{\omega_\de}{\longleftrightarrow} y) +  \beta^2 \eps^{4 - 1 / 4} \P( A \overset{\omega_\de}{\longleftrightarrow}  B) \right).
	\end{equation} 
	By Lemma \ref{lemma:dualonearm}, there exist $\w^-_\de$-circuits with high probability, and in particular, 
	\[
		\P(\w^-_\de \notin U_A) = o(1)
	\] 
	uniformly in $\de$ as $\eps\to0$.
	After summing over $(A,B) \in \mathcal G^2_\eps$, it follows from \eqref{eq:O1bound} that the remaining factor in \eqref{eq:postdecorr} is bounded. All together, it remains to consider
	\begin{equation}\label{eq:step3prog}
		\E \left[  \sum_{(AB) \in \mathcal G^2_\eps} \i(U_{A,B}) \left( \mu^\de_{\Ck^\de}[A] - \beta \eps^{2-1/8} \i(A \cap \Ck^\de) \right) \left( \mu^\de_{\Ck^\de}[B] - \beta \eps^{2 - 1 / 8} \i(B \cap \Ck^\de) \right) \right],
	\end{equation} 
	where we write $U_{AB} = U_A \cap U_B$.

	\noindent\textbf{Step 4. Applying the IIC estimate.}
	This constitutes the heart of the proof. The idea is to use the IIC construction to argue that the mesoscopic structure of macroscopic clusters is highly concentrated and decorrelated across large distances. 
	
	\noindent On the event $U_{AB}$, there exists at most one cluster intersecting both $A$ and $B$. Thus, 
	\begin{align}
		\i&(U_{AB}) \left( \mu^\de_{\Ck^\de}[A] - \beta \eps^{2-1/8} \i(A \cap \Ck^\de) \right) \left( \mu^\de_{\Ck^\de}[B] - \beta \eps^{2 - 1 / 8} \i(B \cap \Ck^\de) \right) \nonumber  \\
		& = \i( \G_A \overset{\Ck^\de}{\longleftrightarrow} \G_B, U_{AB}) \times \i(A \overset{\omega_\de}{\longleftrightarrow} \G_A)\left( \mu^\de_{\Ck^\de}[A] - \beta \eps^{2 - 1 / 8} \right) \times \i(B \overset{\omega_\de}{\longleftrightarrow} \G_B) \left( \mu^\de_{\Ck^\de}[B] - \beta \eps^{2 - 1/8} \right). \label{eq:3factors} 
	\end{align}
	We define $\F_{AB}$ to be the information in the restriction of $\w_\de$ to the set outside both circuits $\G_A$ and $\G_B$. Crucially, the first factor in \eqref{eq:3factors} is measurable with respect to $\F_{AB}$. Observe that there is an implicit claim here that the event of being the $k$'th largest cluster is also measurable with respect to $\F_{AB}$, which is indeed true since the diameter of a cluster cannot be affected by its shape inside an open circuit and no other cluster inside can be larger. 
	As the interiors $\overline{\G}_A$, $\overline{\G}_B$ of the circuits are disjoint, the states of $\w_\de|_{\overline{\G}_A}$ and $\w_\de|_{\overline{\G}_A}$ are conditionally independent given $\F_{AB}$. Moreover, they are distributed according to the measures $\P_{\overline{\G}_A}$ and $\P_{\overline{\G}_A}$ respectively. Thus, we may consider the contributions of $A$ and $B$ separately, and we further decompose
	\begin{multline}
		\E \left[ \i(A\overset{\omega_\de}{\longleftrightarrow}\G_A) \left( \mu^\de_{\Ck^\de}[A] - \beta\eps^{2-1 / 8} \right)  \mid \F_{AB} \right] =\\ \P\left[ A \overset{\omega_\de}{\longleftrightarrow} \G_A \mid \F_{AB} \right] \times \E_{ \overline{\G_A}}\left[ \de^{2 - 1 /  8} |A\cap\C^{\G_A}_0 | - \beta \eps^{2 - 1 / 8} \mid A \overset{\omega_\de}{\longleftrightarrow} \G_A \right],
	\end{multline} 
	where $\C^{\G_A}_0$ is used to represent the fact that the cluster $\Ck^\de$ is acting as the \emph{boundary cluster} under the new measure.
	Now, we invoke our carefully chosen $\beta$ from Proposition \ref{prop:betalimit} and also Lemma \ref{lemma:Xnestimate} to deduce that
	\[
		\E_{ \overline{\G_A}}\left[ \de^{2 - 1 /  8} |A\cap\C^{\G_A}_0 | - \beta \eps^{2 - 1 / 8} \big| A \overset{\omega_\de}{\longleftrightarrow}\G_A \right] = o(\eps^{2 - 1/8}),
	\] 
	uniformly in $\de$. An equivalent estimate holds for the box $B$, and after undoing the conditioning, we learn that 
	\begin{align}
		\E \left[  \sum_{(A,B) \in \mathcal G^2_\eps} \i(U_{AB}) \left( \mu^\de_{\Ck^\de}[A] - \beta \eps^{2-1/8} \i(A \cap \Ck^\de) \right) \left( \mu^\de_{\Ck^\de}[B] - \beta \eps^{2 - 1 / 8} \i(B \cap \Ck^\de) \right) \right] \\
		= o(1)\ \eps^{4 - 1 / 4} \E \left[ \sum_{A,B \in \A_\eps} \i(U_{AB}, A \overset{\omega_\de}{\longleftrightarrow} B)  \right].
	\end{align} 
	This concludes the proof, as the non-vanishing factor is bounded by \eqref{eq:O1bound}.
\end{proof}

\begin{remark}\label{rem:fractal-bdy}
	The second step in the above proof showed how, under some very mild regularity assumptions on $\partial D$, one can handle boundary issues from first principles, albeit in a non-optimal way. If the boundary of $D$ is fractal, one can instead obtain vague convergence of the measures. Namely, the same argument as in the proof of Proposition \ref{prop:tight} gives that for any compact subset $K\subset D$,
	\[
	\limsup_{\de\to0}\Ed[ \mu^\de_{\C^\de_k}(K)^2] < \infty,
	\]
	and throughout the proof of Theorem \ref{thm:boxcount} all boundary issues are avoided by taking $f\in C_c(D)$ to be a continuous function with compact support in $D$. 
\end{remark}

\begin{remark}\label{rem:meas-compacts}
	The statement of Theorem \ref{thm:boxcount} also holds for $f=\i_K$, where $K\subset D$ is any compact subset. It suffices to show that
	\[
	\Ed \left[ \left( \sum_{\substack{A \in \A_\eps \\ A \cap \b K}} \de^{2 - 1 / 8} \mu^\de_{\Ck^\de}[A] \right)^2 \right] = o(1),
	\] 
	which follows by a straightforward argument using Lemma \ref{lemma:Ising0}.
\end{remark}

\begin{remark}\label{rem:free-bc}
	For free boundary conditions, one can check that an almost identical (in fact, simpler) proof holds. Indeed, in the appropriate places, any needed estimates for free boundary conditions can be dominated by those for $+$ boundary conditions on a suitably larger domain.
\end{remark}

\subsection{Measurability of the continuum area measures}\label{sec:meas-meas}

Let us put everything together and use Theorem \ref{thm:boxcount} to prove Theorem \ref{thm:convergencearea}. As the convergence of the clusters has already been established in Theorem \ref{thm:DRC} (see also Lemma \ref{lem:bdy-to-hausdorff}), it remains to prove that any subsequential scaling limit of the discrete measures is a measurable function of $\mathcal F_{\C_k}$. The following statements are stated under the assumption of a coupling in which a.s. convergence holds, which we know exists by Skorokhod's representation theorem.
 
\begin{lemma}\label{lemma:indicator-conv}
	Let $\P$ be a coupling such that the collection $(\Cd_k)_{k \geq 0}$ converges to $(\C_k)_{k \geq 0}$ $\mathbb P$-almost surely
	as $\delta \to 0$. Fix $k\geq0$ and a dyadic box $A\in\A$ of size $\eps_n$, for any $n\geq1$. Then,
	\[
	\mathbf 1(\Cd_k \cap A) \overset{}{\longrightarrow} \mathbf 1(\C_k \cap A),
	\] 
	$\P$-almost surely as $\de \to 0$ .
\end{lemma}
\begin{proof}
	See Appendix \ref{app:conv-ext}.
\end{proof}

\begin{prop}\label{prop:meas}
	Let $\P$ be a coupling such that the collection $(\Cd_k)_{k \geq 0}$ converges to $(\C_k)_{k \geq 0}$ $\mathbb P$-almost surely as $\delta \to 0$. Fix $k\geq0$ and $f\in C(\bar D)$. Then,
 	\[
		\mu^\de_{\C^\de_k}[f]\overset{}{\longrightarrow}\mu_{\C_k}[f]
	\] 
	exists as a limit in $L^2(\P)$ as $\de\to0$. Moreover, the limit $\mu_{\C_k}$ is measurable with respect to $\F_{\C_k}$.  
\end{prop}
\begin{proof}
	By Theorem \ref{thm:boxcount}, we have that
	\[
		\mu^\de_{\C^\de_k}[f] - \beta \eps_n^{2 - 1 / 8}\sum_{A\in \A_n}f_A\i(A\cap\Ck^\de) \overset{}{\longrightarrow} 0
	\]		
	in $L^2(\P)$ as $\de \to 0$ and then $n \to \infty$. For $n$ fixed, the sum above is over finitely many boxes, and applying Lemma \ref{lemma:indicator-conv} we see that
	\[
		\beta \eps_{n}^{2 - 1/8} \sum_{A\in \A_n}f_A\i(A\cap\Ck^\de) \longrightarrow \beta \eps_{n}^{2 - 1 / 8}\sum_{A\in \A_n}f_A\i(A\cap\Ck)
	\] 
	almost surely and in $L^2(\P)$ as $\de \to 0$. Moreover, using Theorem \ref{thm:boxcount} again, we see that the sequence 
	\begin{equation} \label{eq:eps-sequence}
		\Big(\ \beta \eps_{n}^{2 - 1 / 8} \sum_{A\in \A_n}f_A\i(A\cap\Ck) \Big)_{n\geq1}	
	\end{equation}
	is Cauchy in $L^2(\P)$, and thus convergent. Together, we deduce that 
	\[
		\mu_{\C_k}[f] := \lim_{\de\to0}\mu^\de_{\C^\de_k}[f] = \lim_{n\to\infty}\beta \eps_n^{2 - 1 / 8} \sum_{A\in \A_n}f_A\i(A\cap\Ck)
	\]
	exist as a limit in $L^2(\P)$. Since the sequence \eqref{eq:eps-sequence} is measurable with respect to $\F_{\C_k}$, the claim follows.  
\end{proof}

The proof of Theorem \ref{thm:convergencearea} now follows from a standard uniqueness argument. We write the full details of this argument as it will be used repeatedly in what follows.

\begin{proof}[Proof of Theorem \ref{thm:convergencearea}]
	Fix $k\geq0$ and drop it from the notation. Combining Theorem \ref{thm:DRC} with Proposition \ref{prop:tight}, we know that the joint law $(\C^{\de}, \mu^{\de}_{\C^\de})$ is tight: there exists a subsequence $(\de_n)$ and a measure $\mu$ such that
	\[
		(\C^{\de_n}, \mu^{\de_n}_{\C^{\de_n}})\overset{(d)}{\longrightarrow}(\C, \mu),
	\]
	where $\C$ is the appropriate $\CLE_4$ carpet described in Theorem \ref{thm:DRC}. Now, it suffices to show that for any other convergent subsequence $(\de_k)$ with limit
	\begin{equation}\label{eq:unique-law}
		(\C^{\de_k}, \mu^{\de_k}_{\C^{\de_k}})\overset{(d)}{\longrightarrow}(\C, \mu^*),
	\end{equation}
	we have that $(\C, \mu^*)\overset{(d)}{=}(\C, \mu)$. Note that here we are using the fact that the limit $\C$ is independent of the subsequence (in fact explicit) by Theorem \ref{thm:DRC}. 
	
	\noindent Consider the limit in \eqref{eq:unique-law}. By Skorokhod's representation theorem, there exists a joint coupling $\P$ such that convergence holds $\P$-a.s. In particular, for any $f\in C(\bar D)$,
	\begin{equation}\label{eq:meas-unique-1}
		\lim_{k\to\infty}\mu^{\de_k}_{\C^{\de_k}}[f] = \mu^*[f]\quad \text{a.s.}
	\end{equation}
	 Moreover, in such a coupling, we can use Proposition $\ref{prop:meas}$ to see that
	\begin{equation}\label{eq:meas-unique-2}
		\lim_{k\to\infty}\mu^{\de_k}_{\C^{\de_k}}[f] = \lim_{\eps\to0}\beta\eps^{2-1/8}\sum_{A\in\Q_f}f_A\i_{\{A\cap\C\}} =: \mu_{\C}[f]
	\end{equation}
	exist as limits in $L^2(\P)$. Combining \eqref{eq:meas-unique-1} and \eqref{eq:meas-unique-2}, we identify
	\[
		\mu^*[f] = \mu_{\C}[f]
	\]
	as a measurable function of $\C$. Since this function is (crucially) independent of the subsequence $(\de_k)$, the uniqueness of the limit law follows.
\end{proof}

\section{Convergence of the IMF and coupling with the GFF}\label{sec:proofs}

We devote this section to proving all of our main results, as stated in Section \ref{sec:intro}. We start with all the statements concerning the coupling with the GFF, along with the joint convergence of the discrete coupling. We then proceed to identify the measures and prove any remaining claims.

\subsection{Couplings and joint convergence}

\def\sobD{H^{s}_{loc}(D)}
\def\sobC{H^{s}(\CC)}

For completion, we state the tightness of the IMF and sketch the proof. For further background on Sobolev spaces, see Appendix \ref{app:spaces}.

\begin{prop}[\cite{tight-field, mag-field}]\label{prop:ising-tight}
	Let $D$ be a Jordan domain with non-fractal boundary. For $s<-1$, the law of $(\IMF_\de)_{\de>0}$ is tight with respect to the topology of $\sobC$. 
\end{prop}
\begin{proof}
	By the Rellich-Kondrachov embedding theorem, it suffices to show that
	\[
		\limsup_{\de\to0}\E_\de[\norm{\IMF_\de}^2_{\sobC}]	<\infty.
	\]
	Since we can expand the expected norm as 
	\[
		\int_{\CC}(1+|\xi|^2)^s\Ed[|\hat{\IMF}^\de(\xi)|^2]d\xi = \int_{\CC}(1+|\xi|^2)^s\Ed[|\IMF_\de(e^{-2\pi i\xi\cdot})|^2]d\xi,
	\]
	and by our choice of $s<-1$, it suffices to show that for any positive function $f\in L^\infty(\CC)$, 
	\[
		\limsup_{\de\to0}\Ed[\Phi^\de[f]^2] <\infty.
	\]
	This follows from Lemma \ref{lemma:Ising1}.
\end{proof}

\begin{remark}
	The exponent $s<-1$ can be improved up to $s<-1/8$, as established in \cite{tight-field}. However, since Dirac masses belong to the Sobolev space $H^{s}(\CC)$ if and only if $s<-1$, 
	we would need to modify our definition of the discrete IMF. To match our approach using measures, and for simplicity, we do not worry about the optimality of the exponent.
\end{remark}

We now prove the convergence of a single IMF along with its Edwards-Sokal representation, and in a joint coupling with the GFF. The joint coupling with four IMFs will follow readily thereafter. All the arguments presented in the proof below are standard when dealing with (the scaling limit of) Edwards-Sokal representations of spin fields \cite{CME, CME-review, mag-field, GFF-excur}. We do not explicitly state the topologies of convergence in what follows, hoping they have become clear by now.

\begin{lemma}\label{lemma:single-IMF}
	Let $D$ be a Jordan domain with non-fractal boundary. Then, as $\de\to0$, 
	\[
		(H_\de, \IMF_\de, (\C_k^\de)_{k=0}^\infty, (\mu^\de_k)_{k=0}^\infty) 
		\overset{(d)}{\longrightarrow}(h/(\pi\sqrt{2}), \IMF, (\C_k)_{k=0}^\infty, (\mu_k)_{k=0}^\infty). 
	\]
	Moreover, 
	\[
		\Phi = \mu_0 + \sum_{k=1}^\infty\xi_k\mu_k,
	\]
	where convergence holds a.s. and in $L^2(\P)$ in the appropriate Sobolev space.
\end{lemma}

\def\diam{\textnormal{diam}}
\begin{proof}
	Fix $N\geq 1$ and index $s<-1$. Combining Theorem \ref{thm:DRC}, Theorem \ref{thm:convergencearea} and Proposition~\ref{prop:ising-tight}, we can pass to a subsequence (denoted the same way) and apply Skorokhod's representation theorem to pass to a common probability space in which
		\[
			(H_\de, \IMF_\de, (\C_k^\de)_{k=0}^\infty, (\mu^\de_k)_{k=0}^\infty, (\xi_k)_{k=0}^\infty) {\longrightarrow}(h/(\pi\sqrt{2}), \IMF^*, (\C_k)_{k=0}^\infty, (\mu_k)_{k=0}^\infty, (\xi_k)_{k=0}^\infty)
		\]
	almost surely. Note that we have (trivially) included the i.i.d. signs $(\xi_k)_{k\geq0}$ independent of everything else. Moreover, all limits obtained are unique except for $\IMF^*$.
	
	\noindent In the discrete, for any $N\geq1$, we can write 
		\[
			\IMF_\de = \mu_0^\de+\sum_{k=1}^N\xi_k\mu^\de_{k} + R^\de_N,
		\]
	where $R^\de_N$ is simply a remainder term. 
	Since $s<-1$, weak convergence of measures readily implies convergence in $\sobC$. Hence, as a finite sum, 
		\[
			\mu_0^\de+\sum_{k=1}^N\xi_k\mu^\de_{k}\ {\longrightarrow}\ \mu_0+\sum_{k=1}^N\xi_k\mu_{k}
		\]
	almost surely and in $L^2(\P)$ as $\de\to0$. As a difference of two converging terms, we can also conclude that $R^{\de}_N{\longrightarrow}R^*_N$ almost surely as $\de\to0$. For every cutoff $\rho>0$ there is an almost surely finite\footnote{This follows from the Aizenman-Burchard criterion \cite{AB}, which is satisfied by the DRC \cite[Theorem~1.1]{DRC2}. A direct proof may be given in the continuum, see \cite[Proposition 6.1]{XOR-exc}. } number $N\equiv N(\rho)$ of clusters with diameter greater than $\rho$. By Lemma \ref{lemma:Ising3}, it then follows that
	\begin{align}\nonumber
		\E[R^*_N[f]^2] & \leq \limsup_{\de\to0}\Ed\bigg[\sum_{k=N+1}^\infty\mu^\de_k[f]^2\bigg] \\\nonumber
		& = \limsup_{\de\to0}\Ed\bigg[\sum_{k:\ \diam(\C^\de_k)<\rho}^\infty\mu^\de_k[f]^2\bigg] \\\nonumber
		& \leq \norm{f}_\infty^2\limsup_{\de\to0}\de^{4-1/4}\ \Ed\bigg[\sum_{\substack{x,y \in \Dd \\ \d(x,y) \leq \rho}}\sigma_x\sigma_y\bigg]\\ \label{eq:error-rho}
		& \leq C\norm{f}_\infty^2\rho^{2-1/4} \longrightarrow 0\quad \text{as}\quad \rho\to0,
	\end{align}
	for any $f\in H^{-s}(\CC)$. As the signs $(\xi_k)_{k\geq0}$ are i.i.d. and symmetric,	
	\[
		\E[\IMF^*[f]^2] = \sum_{k=1}^N\E[\mu_k[f]^2] + 	\E[R^*_N[f]^2],
	\]
	so \eqref{eq:error-rho} implies that 
	\[
		\sum_{k=0}^\infty\E[\mu_k[f]^2] < \infty.
	\]
	In particular, 
	\[
	\mu_0[f]+\sum_{k=1}^N\xi_k\mu_{k}[f]\ {\longrightarrow}\ \IMF^*[f]
	\]
	almost surely and in $L^2(\P)$ as $N\to\infty$. A standard argument can be used to raise convergence against a fixed test function to convergence in $\sobC$.
	
	\noindent Finally, since we have written $\IMF^*$ as a measurable function of $h$ and the independent signs $(\xi_k)_{k\geq0}$, and this function is the same for any subsequential limit, we conclude joint convergence as in the proof of Theorem \ref{thm:convergencearea}.
\end{proof}

\begin{remark}
	Similarly to Remark \ref{rem:fractal-bdy}, whenever $D$ has fractal boundary one readily obtains convergence with respect to the topology of $\sobD$ when $s<-1$.
\end{remark}

\def\f{\textup{f}}
\def\A{\mathbb{A}}

Before proving the main Theorems \ref{thm:main-coupling}-\ref{thm:main-decomp}-\ref{thm:joint}, let us describe in detail how the spins $(\tau_k^+)_{k\geq1}$ and $(\tau_k^\f)_{k\geq1}$ are defined. In order to use the results of \cite{DRC, DRC2} and match them to the values of the discrete XOR-Ising spins on the respective clusters, we define them in terms of the labels of the two-valued sets. Writing $\La_{-a,b}^+$ to denote the collection of loops of $\La_{-a,b}$ with boundary value $b$, the iterations are as follows:

\paragraph{$(+)$ For $+$ boundary conditions:}\hypertarget{pos}{}
\begin{enumerate}[align = left, labelwidth=\parindent, labelsep = 0pt]
	\item[(1.$+$)]\ The boundary cluster is $\C_0^+=\A_{-2\lambda, 2\lambda}(\GFF)$. The label of this cluster is $c^+_{\C^+_0}=-1$.
	\item[(2.$+$)]\ Let $\gamma\in\La_{-2\lambda, 2\lambda}(\GFF)$.
	\begin{enumerate}
		\item\ If $\gamma\in\La_{-2\lambda, 2\lambda}^+(h)$, every loop $\ell\in\La_{-2\lambda, (2\sqrt{2}-2)\lambda}(\GFF^\gamma)$ is given a cluster \\
		$\C^+_\ell=\A_{-2\lambda, 2\lambda}(\GFF^\ell)$. The label of the cluster is
		\begin{equation}\label{eq:tau-1}
			c^+_{\C^+_\ell} = \begin{cases}
				+1\quad \text{if}\ \ell\in\La_{-2\lambda, (2\sqrt{2}-2)\lambda}^+(\GFF^\gamma), \\
				-1\quad \text{if}\ \ell\in\La_{-2\lambda, (2\sqrt{2}-2)\lambda}^-(\GFF^\gamma).
			\end{cases}
		\end{equation}
		\item\ If $\gamma\in\La_{-2\lambda, 2\lambda}^-(h)$, every loop $\ell\in\La_{-(2\sqrt{2}-2)\lambda, 2\lambda}(\GFF^\gamma)$ is given a cluster\\
		 $\C^+_\ell=\A_{-2\lambda, 2\lambda}(\GFF^\ell)$. The label of the cluster is
		\begin{equation}\label{eq:tau-2}
			c^+_{\C^+_\ell} = \begin{cases}
				+1\quad \text{if}\ \ell\in\La_{-(2\sqrt{2}-2)\lambda, 2\lambda}^-(\GFF^\gamma), \\
				-1\quad \text{if}\ \ell\in\La_{-(2\sqrt{2}-2)\lambda, 2\lambda}^+(\GFF^\gamma).
			\end{cases}
		\end{equation}
	\end{enumerate}
	\item[(3.$+$)]\  Iteratively, let $\C^+_\ell=\A_{-2\lambda, 2\lambda}(\GFF^\ell)$ for some loop $\ell$. Apply (2.$+$) to every $\gamma\in\mathcal L_{-2\lambda, 2\lambda}(\GFF^\ell)$. Define the spin
	\[
	\tau_{\C^+_\ell}^+ = \prod_{j}(-c^+_{\C^+_j}),
	\]
	where the product runs over all clusters with an outer boundary enclosing $\C^+_\ell$, including the cluster $\C^+_\ell$ itself.
\end{enumerate}

\paragraph{$(\text{free})$ For free boundary conditions:}\hypertarget{free}{}
\begin{enumerate}[align = left, labelwidth=\parindent, labelsep = 0pt]	
	\item[(1.\textup{f})]\ Every loop $\ell \in \La_{-\sqrt{2}\lambda, \sqrt{2}\lambda}(\GFF)$ is given a cluster $\C^{\textrm{f}}_\ell=\A_{-2\lambda, 2\lambda}(\GFF^\ell)$. The label is
	\[
	c^{\textup{f}}_{\C^{\textup{f}}_\ell} = \begin{cases}
		+1\quad \text{if}\ \ell\in\La_{-\sqrt{2}\lambda, \sqrt{2}\lambda}^+(h), \\
		-1\quad \text{if}\ \ell\in\La_{-\sqrt{2}\lambda,\sqrt{2}\lambda}^-(h).
	\end{cases}
	\]
	\item[(2.\textup{f})]\ Iterate as in (3.$+$). Define $\tau^\f_{\C^\f}$ in the same way.
\end{enumerate}

\begin{figure}[htbp]
	\renewcommand*\thesubfigure{\arabic{subfigure}}
	
	\begin{subfigure}[t]{0.45\textwidth}
		\includegraphics[width=\textwidth]{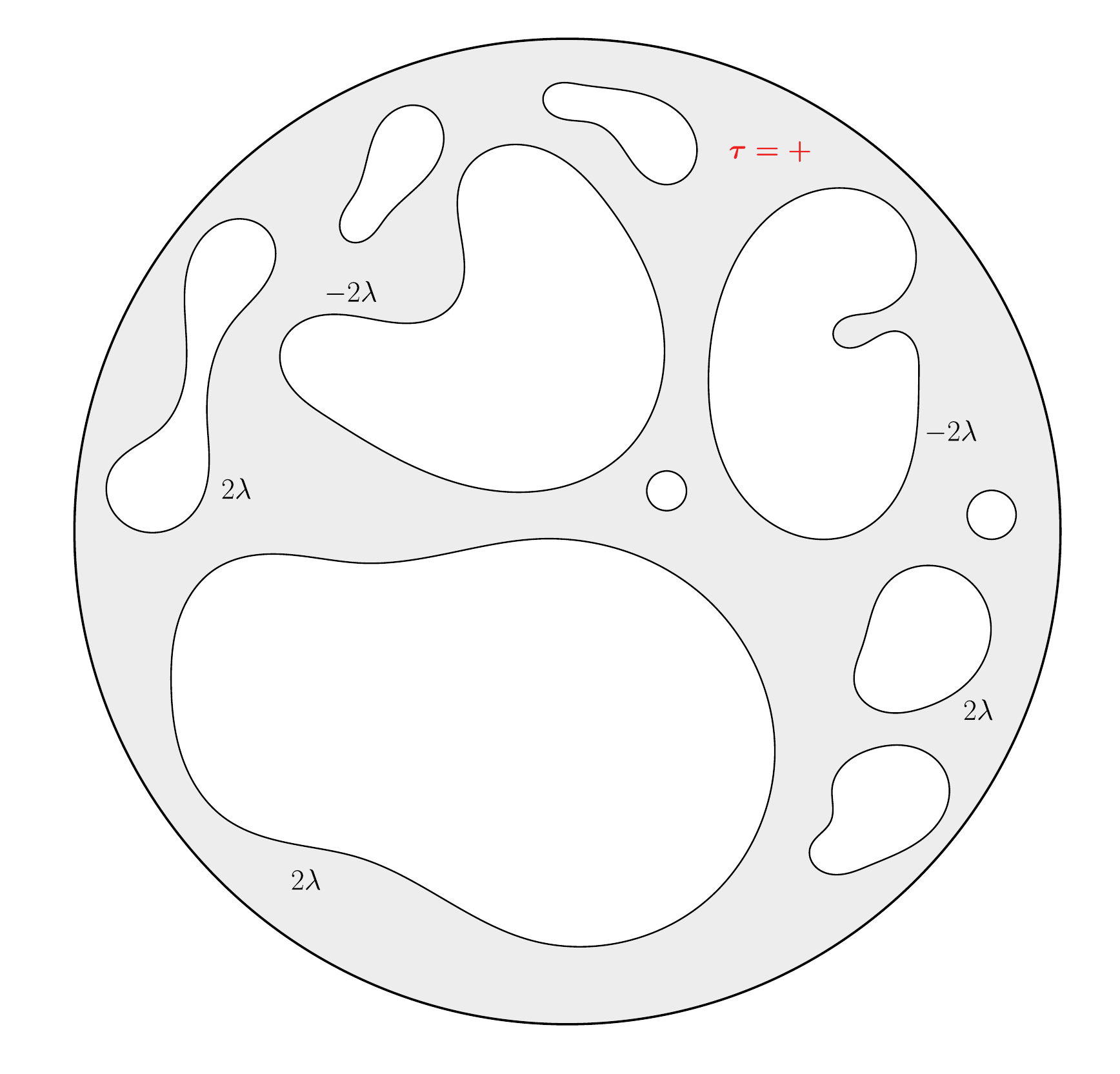}
		\caption[]{}
	\end{subfigure}\hspace{\fill}
	\begin{subfigure}[t]{0.45\textwidth}
		\includegraphics[width=\textwidth]{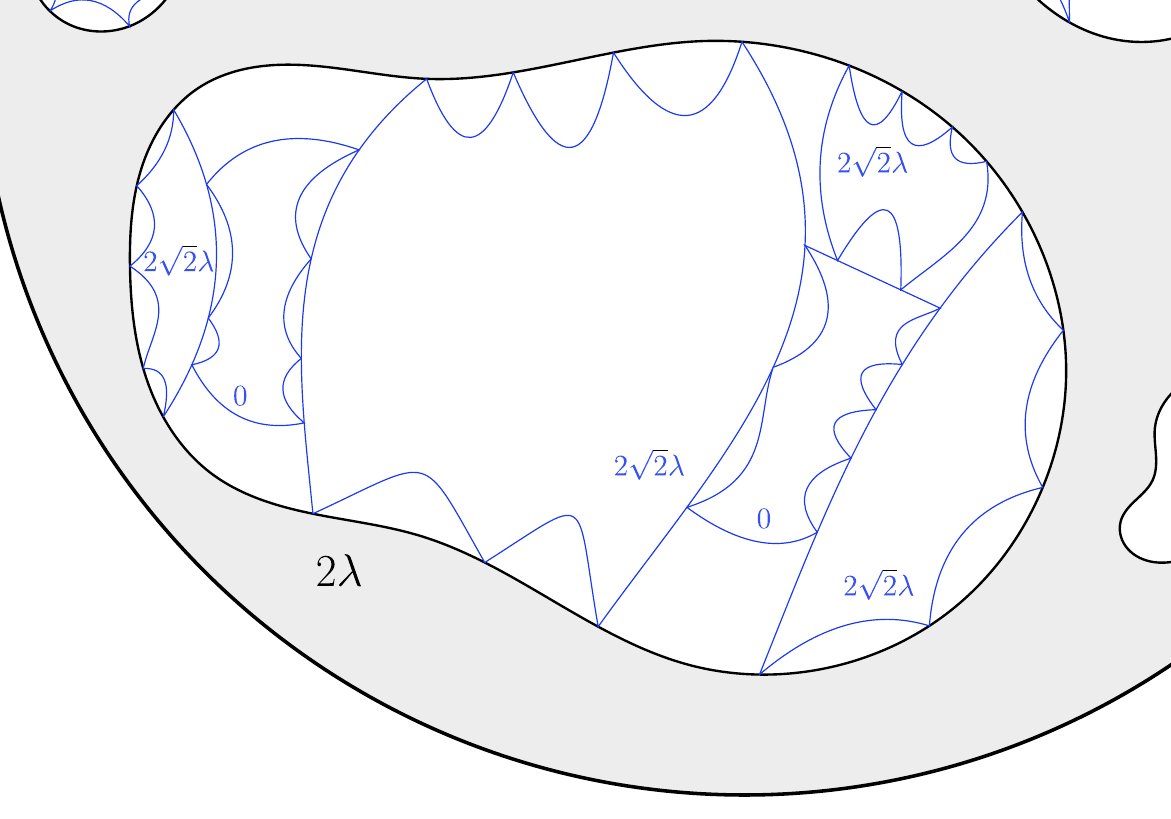}
		\caption[]{}		
	\end{subfigure}\bigskip\bigskip\bigskip

	\begin{subfigure}[t]{0.45\textwidth}
		\includegraphics[width=\textwidth]{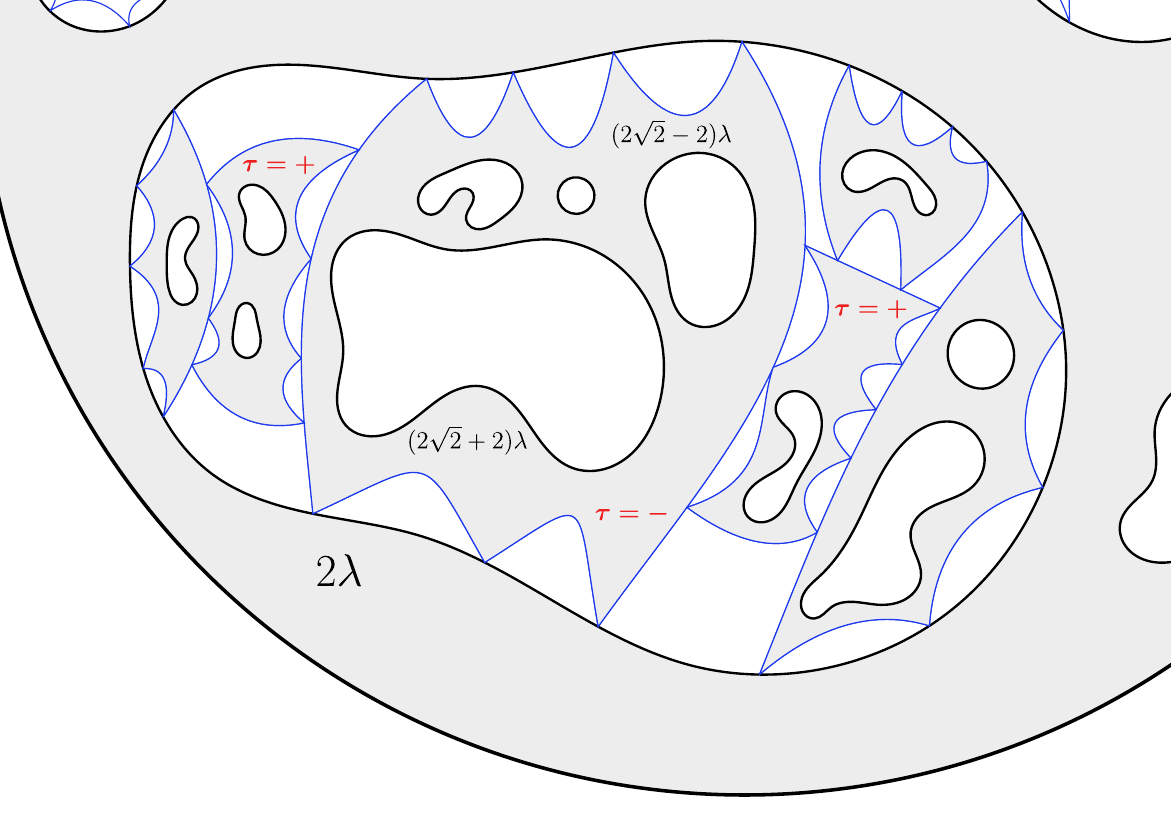}
		\caption[]{}
	\end{subfigure}\hspace{\fill}
	\begin{subfigure}[t]{0.45\textwidth}
		\includegraphics[width=\textwidth]{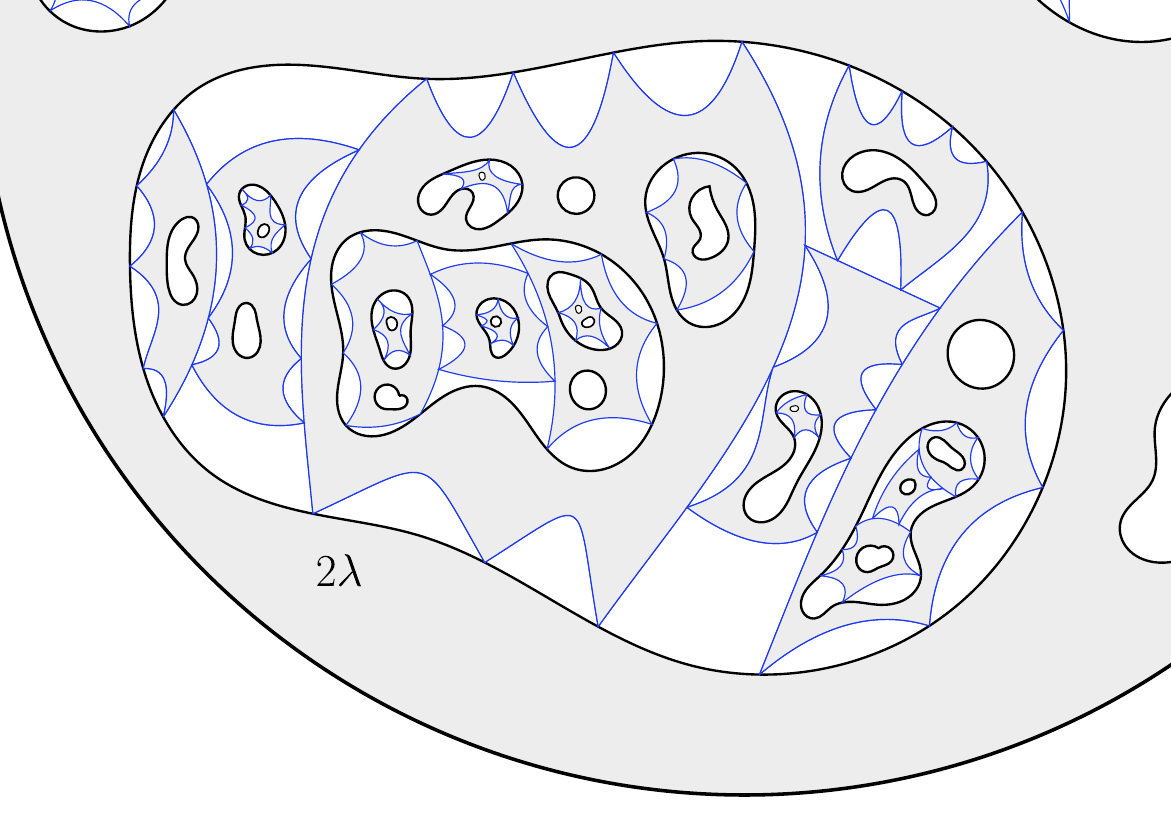}
		\caption[]{}	
	\end{subfigure}\bigskip

	\caption{
		Step-by-step illustration of the exploration of the clusters $(\C^+_k)_{k \geq 0}$, shaded in grey. Whenever a cluster is discovered for the first time, its $\tau^+$--spin is shown in red. (1) The boundary cluster is $\C^+_0=\A_{-2\lambda,2\lambda}(h)$. (2) We focus on a single loop $\gamma\in\La^+_{-2\lambda,2\lambda}(h)$ with boundary value $2\lambda$. Every loop $\ell\in\La_{-2\lambda, (2\sqrt{2}-2)\lambda}(h^\gamma)$, coloured in blue, will be associated a cluster. Note that the boundary values of these loops must be as shown, and no even loop can touch the boundary. (3) The clusters inside each blue loop $\ell\in\La_{-2\lambda, (2\sqrt{2}-2)\lambda}(h^\gamma)$ are discovered, where each cluster is of the form $\A_{-2\lambda, 2\lambda}(h^\ell)$. Mind the correspondence between the $\tau^+$--spins and the labels. (4) Nested clusters are defined iteratively within each loop (although not all iterations are shown here).
		We highlight that the collection $(\tau_k^+)_{k\geq1}$ is, in fact, measurable with respect to \emph{only} the geometry of the clusters, see Remark \ref{rem:labels}.
		\label{fig:full-iter}
		}
\end{figure}

\noindent We can now begin to prove our main results.

\begin{proof}[Proof of Theorems \ref{thm:main-coupling}-\ref{thm:main-decomp}-\ref{thm:joint}]
	It suffices to prove the joint convergence stated in Theorem \ref{thm:joint}, along with the explicit representations \eqref{eq:IMF}.
	
	\noindent Denote by $(\IMF_\de^{+}, \tilde\IMF_\de^{+})$ the pair of discrete IMFs with $+$ boundary conditions on $\Dd$. For any $k\geq1$, the outer boundary of the cluster $\C_k^{\de}$ is given by an inner boundary of the coupled DRC on the dual graph. As introduced in \cite{DRC}, each such inner boundary has an associated label $c_k^{\delta}$, defined in terms of the parity of the dual DRC. In particular, it is done in such a way that
	\[
		\tau_{k}^{\de, +} = - \prod_{j}c^+_{\C^+_j},
	\]
	where, again, the product runs over all clusters with an outer boundary enclosing $\C^{\de, +}_k$. The joint convergence of both fields, along with their decompositions, follow from the following observations:
	\begin{enumerate}
		\item By \cite[Theorems 6.2 - 6.4]{DRC}, the law of the labels $(c_k^{\delta, +})_{k\geq1}$ converges (jointly) to the law of $(c_k^+)_{k\geq1}$, as defined above.
		\item The analogous decompositions to \eqref{eq:IMF} hold for the discrete IMFs. And both fields are independent of each other.
		\item The law of $(\xi_k\tau_k^+)_{k\geq1}$ is still that of i.i.d. coin tosses, hence Lemma \ref{lemma:single-IMF} applies to $\tilde\Phi_\delta^+$.
	\end{enumerate}
	The joint convergence of all four IMFs is proved analogously, i.e. combining the joint convergence of Theorem \ref{thm:DRC} with Lemma \ref{lemma:single-IMF}, noting Remark \ref{rem:free-bc}.
\end{proof}

\begin{remark}\label{rem:labels}
	The labels $(c_k^+)_{k\geq0}$, and thus the spins  $(\tau_k^+)_{k\geq0}$, are not only a measurable function of $h$, but also a measurable function of the collection $(\C_k^+)_{k\geq0}$. Indeed, by {\cite[Theorem 4.1]{AruSep}}, two spins $\tau_k^+$ and $\tau_{k'}^+$ must differ whenever they correspond to two clusters that touch each other, and all clusters touching the preceding $\CLE_4$ must have the same spin. In the discrete, however, it is \emph{not} true that the parity labels are measurable functions of \emph{only} the geometry of the clusters. As mentioned in Remark \ref{remark:antiferro}, the law of the spins $\tau^{\de,+}$ given the clusters of $\w_\de$ is that of an antiferromagnetic Ising model (with an interaction proportional to the size of the common boundary), and so it is reasonable that in the limit surviving macroscopic clusters which touch will never have the same label.
\end{remark}

\subsection{Identification of the measures and conformal covariance}\label{sec:ident-meas} 

\def\Cont{\textnormal{Cont}}
\def\ContU{\textnormal{Cont}_U}
\def\ContL{\textnormal{Cont}_L}
\def\Box{\textnormal{Box}}
\def\BoxU{\textnormal{Box}_U}
\def\BoxL{\textnormal{Box}_L}
\def\Int{\mathrm{Int}}

\def\Leb{\mathrm{Leb}}

Let us first draw some comparisons between our measures and the $d$-dimensional Minkowski content of the cluster $\C$. We follow the definitions as in \cite{zhan}. For any closed set
$S\subset\CC$, the $d$-dimensional Minkowski content of $S$ is defined as 
	\begin{equation} \label{eq:Mink-cont}
		\Cont_d(S) := \lim_{\eps\to0}\eps^{d-2} \Leb(B(S, \eps)),
	\end{equation}
provided that the limit exists. Here $\Leb$ is the two-dimensional Lebesgue measure and $B(S, \eps)$ is the $\eps$-enlargement of the set $S$. One should think about this definition as the generalisation of $d$-dimensional area. Moreover, a measure  $\mathcal M$ is the $d$-dimensional Minkowski content measure of $S$ in some domain $D$ if it is a measure supported in $D\cap S$ and for every compact set $K\subset D$, 
	\begin{equation}
		\Cont(K\cap S) =\mathcal M_S(K)<\infty.
	\end{equation}
Recall that the Minowski dimension of $S$ is defined as
	\[
		d:= \dim_M(S) := \lim_{\eps\to0}\frac{\log(N_\eps(S))}{\log(1/\eps)}
	\]
provided the limit exists, where $N_\eps(S)$ denote the number of $\eps$-boxes intersecting $S$. In analogy to \eqref{eq:Mink-cont}, it is then natural to define the $d$-dimensional ``box'' content as
	\begin{equation}\label{eq:box-cont}
		\Box_d(S) = \lim_{\eps\to0}\eps^{d}N_\eps(S).
	\end{equation}
Indeed, this definition captures the idea of $d$-dimensional area in the same spirit as the Minkowski content does (and is tied more closely to the box-counting aspect of the Minkowski dimension). Note, however, that the existence of either limit in \eqref{eq:Mink-cont}-\eqref{eq:box-cont} does not guarantee the existence of the other. 

In the previous section, we have shown that the ``box'' content measure of the cluster $\C$ exists and is given by a constant multiple of $\mu_\C$. That is, the continuum measures are indeed ``area'' measures. While this is a very explicit identification, we choose to further verify the uniqueness axioms posed in \cite{CLE-meas} for measures of this type.

\begin{proof}[Proof of Theorem \ref{thm:unique-meas}]
	The existence of the box-counting limit is the content of Proposition \ref{prop:meas}, along with Remark \ref{rem:meas-compacts}. To conclude uniqueness, we verify the axioms in Theorem \ref{thm:Miller-Schoug}.
	\begin{enumerate}[label=(\roman*)]
		\item Immediately follows\footnote{Observe that any locally finite measure on $\CC$ is automatically inner and outer regular, so we can already extract some nice properties.} from Theorem \ref{thm:boxcount}, or simply Proposition \ref{prop:tight}.
		\item We use the same notation as in the statement of Theorem \ref{thm:Miller-Schoug}. For any $f\in C_c(V)$,
		\[
			\mu_{\C}[f] = \lim_{\eps\to0}\beta\eps^{2-1/8}\sum_{A\in\A_\eps}f_A\i(A\cap\C) = \lim_{\eps\to0}\beta\eps^{2-1/8}\sum_{A\in\A_\eps}f_A\i(A\cap\C_V).
		\]
		Moreover, given $\C\setminus\C_V$, the law of $\C_V$ is that of a $\CLE_4$ on $V$. In particular, the right-hand side is precisely the definition of the measure in the domain $V$.
		\item We prove a slightly stronger claim, as it will be needed later. Fix two Jordan domains $D$ and $\tilde D$, with clusters $\C$ and $\tilde\C$ respectively. Given our (continuum) decomposition of the IMF, 
		\begin{equation}\label{eq:meas-CE}
			\mu_\C[f] = \E[\IMF^D[f] \mid \F_{\C}]
		\end{equation}
		for any $f\in C_c(D)$. Observe that, in order for this argument to apply to measures other than that of the boundary cluster, we allow for arbitrary Jordan domains (e.g. those defined by the inner boundaries of a cluster). In particular, by Theorem \ref{thm:CHI1},
		\[
			\E[\mu_\C[f]] = \E[\IMF^D[f]] = 2^{1/4}\mathfrak{C}\int_{D}\CR(z, D)^{-1/8}f(z)dz.
		\]
		For any conformal $\varphi: D\to\tilde\D$ with $\varphi(\C)=\tilde\C$, 
		\begin{align*}
			\E[\mu_{\tilde\C}[f\circ\varphi^{-1}]] & = 2^{1/4}\mathfrak{C} \int_{\tilde\D}\CR(z, \tilde D)^{-1/8}f(\varphi^{-1}(z))dz \\
			& = 2^{1/4}\mathfrak{C}\int_D\CR(\varphi(w), \tilde D)^{-1/8}f(w)|\varphi'(w)|^{2}dw \\
			& = 2^{1/4}\mathfrak{C} \int_D\CR(w, D)^{-1/8}f(w)|\varphi'(w)|^{2-1/8}dw \\
			& = \int_D f(w)|\varphi'(w)|^{2-1/8}d\E[\mu_\C](z),
		\end{align*}
		as required.
\end{enumerate}
\end{proof}

\begin{remark}
	A direct proof of the conformal covariance of the measures follows from the same arguments as in {\cite[Section 6]{GPS}}. Indeed, it is not hard to see that our $L^2$ approximation is unaffected by rotating or translating the $\eps$-lattice. In fact, we observe that both arguments are not too far from each other: the $L^2$ approximation tells us that the measures are very close to its expectation, and thus the conformal covariance of the latter should be sufficient. 
\end{remark}

\begin{remark}
	The proof of conformal covariance implicitly contains the following Markov property statement: given some cluster $\C$,
	\[
		\IMF^D[f] = \IMF^{O(\C)}[f]
	\]
	for all $f\in C_c(O(\C))$, where $O(\C)$ is the domain encircled by the outer boundary of $\C$.
\end{remark}

Requiring only conformal covariance of the intensity measure in Theorem \ref{thm:Miller-Schoug} actually allows one to prove the conformal covariance of the IMF only from the conformal covariance of its one-point function (rather than requiring the covariance of all moments). Indeed, Theorem \ref{thm:Miller-Schoug} implies that the axioms (1)-(4) are enough to conclude the conformal covariance of the measure. From here, it follows that
\begin{align}\nonumber
		\IMF^{\tilde\D}[f\circ\varphi^{-1}] & = \mu_0^{\tilde\D}[f\circ\varphi^{-1}] + \sum_{k=1}^N\xi_k\mu_k^{\tilde\D}[f\circ\varphi^{-1}] + R_N^{\tilde\D}[f\circ\varphi^{-1}] \\ \nonumber
		& = \mu_0^D[f|\varphi'|^{2-1/8}] + \sum_{k=1}^N\xi_k\mu_k^D[f|\varphi'|^{2-1/8}] + R_N^{\tilde\D}[f\circ\varphi^{-1}].
\end{align}
But by the Koebe quarter theorem, the diameter of the clusters in $D$ is comparable to that of the clusters in $\tilde\D$. In particular, the tail of the sum on the second line still vanishes almost surely as $N\to\infty$. And the law of the limit is that of $\Phi^D[f|\varphi'|^{2-1/8}]$, as required.

\subsection{The Ashkin-Teller magnetisation field}

We devote this last, short section to proving that the conjectured scaling limit of the AT magnetisation field is well-defined. 

\begin{proof}[Proof of Proposition \ref{prop:AT-mag-field}]
	Let $g\neq4$. Let $(\C_{k}^{g})_{k\geq0}$ be the clusters under $+$ boundary conditions, which we recall are defined in the almost verbatim analog of Theorem \ref{thm:main-decomp} where any explicit occurrence of $\sqrt{2}$ is replaced by $\sqrt{g}$. Let $(\mu_{k}^g)_{k\geq0}$ be the box-counting measures on these clusters, as defined in Theorem \ref{thm:unique-meas}. We prove that the infinite sum
	\begin{equation}\label{eq:L2-sum}
		\Psi^{g} = \mu_0^g + \sum_{k=1}^\infty\xi^+_k\mu_k^{g}
	\end{equation}
	converges almost surely and in $L^2$. The proof of convergence for the decompositions of the remaining fields is analogous.
	
	\noindent We recall that the ordering of the clusters is by decreasing size of diameter. Moreover, for any $g$, the collection of clusters is locally finite by \cite[Proposition 6.1]{XOR-exc}. To prove the desired convergence of \eqref{eq:L2-sum} it suffices to show that for any $f\in C_c(D)$, 
	\[
	 	\sum_{k=N}^M\E[\mu_k^{g}[f]^2]\longrightarrow0
	\]
	as $M\to\infty$ and then $N\to\infty$. In fact, as this is a non-negative and monotone sequence, it is enough to show that
	\[
	\sum_{k=N}^M\E[\mu_k^{g}[f]^2\mid\F_{O_N}]\longrightarrow0 \quad \text{a.s.},
	\]
	where $\F_{O_N}$ is the $\sigma$-algebra generated by the outermost outer boundaries $O_k$ of the clusters $C_k^g$ for $k\geq N$. One should think of conditioning on the iteration of two-valued sets stopped when discovering loops with diameter smaller than some $\rho=\rho(N)>0$. By \eqref{eq:meas-CE}, 
	\[
		\E[\mu_k^{g}[f]^2\mid \F_{O_N}] = \E[\E[\IMF^{O_k}[f]\mid \F_{\C_k^g}]^2\mid \F_{O_N}] \leq \E[\IMF^{O_k}[f]^2 \mid \F_{O_N}]
	\]
	where $\IMF^{O_k}$ is an IMF in the domain encircled by $O_k$, and independent of everything else.	Now, recall that the outer boundaries $O_k$ are all given by loops of two-valued sets of height gap $2\sqrt{g}\lambda$. Crucially, by \cite{dim-TVS}, these sets have Hausdorff dimension a.s. equal to
	\[
	2-\frac{1}{2g} < 2 - \frac{1}{8}
	\]
	for $g<4$. 
	Since we have fixed the conditioning in the domains, we can conclude from the following (deterministic) claim applied to $D_N = \cup_{k\geq N}O_k$.

	\noindent \textbf{Claim:} Let $D_N\subset D$ be a decreasing sequence of domains (not necessarily connected) such that the maximal diameter over its connected components goes to zero as $N\to\infty$. 
	Moreover, assume that the dimension of $\partial D_N$ is strictly smaller than $2-1/8$. 
	Then, the-two point function \cite{CHI1} of the Ising model on $D_N$ is such that, as $N\to\infty$,
	\[
		\int_{D_N}\int_{D_N}\CR(z, D_N)^{-1/8}\CR(w, D_N)^{-1/8}\cosh((1/2)G_{D_N}(z,w))^{1/2}dzdw \longrightarrow 0.
	\]
	To prove this claim, we write $G_{D_N}(z,w) = -\log(|z-w|) + g_{D_N}(z,w)$ where $g_{D_N}$ is a bounded function. Up to constants, the integral above is then upper bounded by 
	\begin{align}\nonumber
		\int_{D_N}\int_{D_N}\CR(z, D_N)^{-1/8}&\CR(w, D_N)^{-1/8}|z-w|^{-1/4}dzdw \\ \nonumber
		& + \int_{D_N}\int_{D_N}\CR(z, D_N)^{-1/8}\CR(w, D_N)^{-1/8}dzdw.
	\end{align}
	A direct computation shows that, under the assumption on the dimension of $\partial D_N$, the desired convergence holds.
\end{proof}

\begin{remark}
	It is clear from the proof that the case $g=4$, or equivalently $J=U$, requires much more care. Recall this is precisely the point at which the magnetisation field has the same law as the polarisation field, conjecturally that of $\cos((1/2)h)$. A continuum decomposition of this field, and more generally of $\cos(\alpha h)$ for any $\alpha\in(0,1)$, are studied in \cite{XOR-exc}.
\end{remark}

\appendix
\section{Estimates for Ising correlations}\label{app:ising-est}
Let $D$ be a Jordan domain with non-fractal boundary as per Definition \ref{def:non-fractal}. Let $\Dd\subset\de\Z^2$ be a discrete domain approximation as in \eqref{eq:dom-approx}. To derive the correlation inequalities used throughout the paper, one may take the following Onsager-type inequalities as the starting point. Both bounds can be derived from a standard FK-Ising argument and the results in \cite[Section 5]{RSW}. All constants in this section depend only on the domain $D$.

\begin{prop}[{\cite[Proposition 3.10]{tight-field}}, {\cite[Lemma 4.4]{JSW}}]\label{lemma:Ising0}
	There exists a constant $C> 0$ such that for any $x,y \in \Dd$, 
	\begin{enumerate}
		\item[(i)] $\de^{-1/8}\Ed^+[\sigma_x]\leq C \d(x,\b\Dd)^{-1/8}$,
		\item[(ii)] $\de^{-1/4}\Ed^+[\sigma_x\sigma_y]\leq C \left( |x-y| \wedge \d(x,\b\Dd) \right)^{-1/8} \left( |x-y| \wedge \d(y,\b\Dd) \right)^{-1/8}$.
	\end{enumerate}
\end{prop}	

\noindent These bounds readily yield the upper bounds throughout Lemma \ref{lemma:Ising1}--\ref{lemma:Ising3}. The lower bounds in Lemma \ref{lemma:Ising1} can be proved using \cite[Lemma 5.4, Proposition 5.5]{RSW}. We stress that these estimates are well-known, and e.g. appear throughout the proofs in \cite{mag-field}.

\begin{lemma}\label{lemma:Ising1} There exists $C,c > 0$ such that	
	\[
	c \leq \de^{2 - 1 / 8}\ \Ed^+\bigg[ \sum_{x\in\Dd}\sigma_x\bigg] \leq C.
	\] 
	and 
	\[
	c \leq \de^{4 - 1 / 4}\ \Ed^+\bigg[ \sum_{x,y\in\Dd}\sigma_x\sigma_y\bigg] \leq C.
	\] 
\end{lemma}

\begin{lemma}\label{lemma:Ising2} There exists $C > 0$ such that, for all $\eps > 0$ and $\de$ sufficiently small,
	\[
	\de^{4 - 1 / 4}\ \Ed^+\bigg[\sum_{\substack{x,y \in \Dd \\ \d(x,\b\Dd) \leq \eps}} \sigma_x\sigma_y\bigg] \leq C \eps^{1 - 1/8}.
	\] 
\end{lemma}

\begin{lemma}\label{lemma:Ising3} There exists $C > 0$ such that, for all $\eps > 0$ and $\de$ sufficiently small,
	\[
	\de^{4 - 1 / 4}\ \Ed^+\bigg[\sum_{\substack{x,y \in \Dd \\ \d(x,y) \leq \eps}} \sigma_x \sigma_y \bigg] \leq C \eps^{2- 1 / 4}.
	\] 
\end{lemma}

\section{Spaces of random variables}\label{app:spaces}

The goal of this section is to rigorously define the various topologies with respect to which we take convergence in the scaling limit $\de\to0$, and give some background for them. 

Let $D$ be a Jordan domain. Following \cite{DRC, DRC2}, we say that $D_\de\subset \de\Z^2$ is a discrete domain approximation if 
\begin{equation}\label{eq:dom-approx}
	d_{\ell}(\partial D^\de, \partial D)\longrightarrow0
\end{equation}
as $\de\to0$. Here, the distance $d_{\ell}$ between any two loops $\gamma_1$, $\gamma_2$ is defined as
\begin{equation}\label{eq:dist-loops}
	d_\ell(\gamma_1, \gamma_2) := \inf\sup_{t\in\mathbb S^1}|\gamma_1(t)-\gamma_2(t)|,
\end{equation}
where the infimum is taken over all continuous bijective parametrizations of the loops by $\mathbb S^1$. 

The topologies of convergence are given by either of the following spaces. Again, we stress that these are standard definitions, and can be found in e.g \cite{GPS, CME, JSW}.

\noindent \textbf{(1) Collections of loops:}  Let $\mathsf L=\mathsf{L}(D)$ be the space of locally finite\footnote{Meaning that, for every $\eps>0$, only finitely many loops in the collection have diameter greater than $\eps$.} collections of non-self-touching loops in $D$ that do not cross each other. This space can be equipped with a Hausdroff-type metric induced by the metric $d_\ell$ in \eqref{eq:dist-loops}. That is, we define
\[
d_{\mathsf L}(\La_1, \La_2) := \inf\{\eps>0:\ \forall\gamma_1\in\La_1^\eps\ \ \exists\gamma_2\in\La_2 \ \ \text{s.t.}\ \ d_{\ell}(\gamma_1, \gamma_2)<\eps\ \text{and viceversa}\},
\]
where $\La^\epsilon$ is the (finite) collection of loops in $\La$ with diameter greater than $\eps>0$. This is the same metric as in \cite{DRC, DRC2}. The space  $(\mathsf{L}, d_{\mathsf L})$ is complete and separable \cite{AB}. Moreover, a sequence $(\La_n)_{n\geq1}$ of random collections of loops in $\mathsf{L}$ satisfying the Aizenman-Burchard criterion, denoted (\textbf{H1}) in \cite{AB}, is tight with respect to the metric $d_{\mathsf L}$, and in particular converges to a locally finite collection of loops.

\noindent \textbf{(2) Collections of closed sets:} Let $\mathsf C=\mathsf{C}(D)$ be the space of locally finite collections of closed subsets of $D$. This space can be equipped with a metric analogous to $d_{\mathsf L}$ defined with respect to the standard Hausdorff metric $d_H$ on closed subsets of $D$. That is, 
\[
d_{\mathsf C}(\C_1, \C_2) := \inf\{\eps>0:\ \forall C_1\in\C_1^\eps\ \ \exists C_2\in\C_2 \ \ \text{s.t.}\ \ d_{H}(C_1, C_2)<\eps\ \text{and viceversa}\}.
\]
Again, the space $(\mathsf C, d_{\mathsf C})$ is complete and separable. See Lemma \ref{lem:bdy-to-hausdorff} for a relation between the metrics $d_{\mathsf L}$ and $d_{\mathsf C}$.

\noindent\textbf{(3) Convergence of measures}: Let $\mathsf M=\mathsf M(D)$ be the space of finite measures supported in $D$. A sequence of measures $(\mu_n)_{n\geq1}$ converges to $\mu$ with respect to the weak topology if 
\[
\mu_n[f]\longrightarrow\mu[f]
\]
for every continuous, bounded function $f\in C_b(D)$. This topology is metrizable (under e.g. the Prokhorov metric) and turns $\mathsf M$ into a complete and separable metric space \cite[Lemma 4.5]{kallenberg}. A sequence of measures $(\mu_n)_{n\geq1}$ converges to $\mu$ with respect to the vague topology if 
\[
\mu_n[g]\longrightarrow\mu[g]
\]
for every continuous function $g\in C_c(D)$ with compact support in $D$. Again, this topology is metrizable in a way that turns $\mathsf M$ into a complete and separable metric space \cite[Lemma 4.6]{kallenberg}. It is a standard fact that weak convergence is equivalent to weak convergence and tightness of the sequence. For sequences of (countable) collections of measures, convergence is taken to be coordinate-wise.

\noindent\textbf{(4) Sobolev Spaces:} We closely follow the definitions as found in \cite{JSW}. Let $s\in\R$. The Sobolev space $\sobC$ is defined as the space of tempered distributions $h$ such that
\[\label{eq:def-sobolev}
\norm{h}^2_{H^{s}(\CC)}:= \int_\CC(1+|\xi|^2)^{s}|\hat h(\xi)|^2d\xi<\infty,
\]
where $\hat h$ is the Fourier transform of $h$. The space $H^{s}(\CC)$ is a separable Hilbert space. For every $s>0$, the space $H^{-s}(\CC)$ is the dual of $H^{s}(\CC)$ under the standard $L^2(\CC)$ dual pairing. Moreover, the space $H^{-s}(\CC)$ is a subspace of the space of distributions in $\CC$, i.e. the space of continuous linear functions acting on $C^\infty_c(\CC)$. We highlight that
\begin{enumerate}
	\item [(i)] The Dirac distribution $\boldsymbol{\updelta}_x$ is an element of $H^{s}(\CC)$ if and only if $s<-1$. The same holds for any compactly supported Borel measure (so, in particular, our discrete area measures).
	\item [(ii)] There exists a continuous embedding of $H^{s}(\CC)$ into $C_c(\CC)$ if and only if $s>1$.
\end{enumerate}
For these reasons, all of our fields are taken to converge in Sobolev spaces with regularity $s<-1$.

For any domain $D\subset\CC$, the space $H^{s}(D)$ is defined via restriction. That is, $f\in H^{s}(D)$ if there exists $h\in\sobC$ such that $f=h|_D$. The local Sobolev space $\sobD$ is defined as the space of distributions $h$ such that $h\psi\in\sobC$ for every $\psi\in C_c^\infty(D)$. For further background on Sobolev spaces we refer to \cite{adams}, see also \cite{BerPow, JSW, SLE-GFF, tight-field} for their usage in similar contexts.

\section{Convergence results in \cite{DRC}}\label{app:conv-ext}
\def\L{\mathcal{L}}

The purpose of this section is to tackle some remaining technicalities concerning the convergence of the clusters, when viewed as closed sets, with respect to the Hausdorff distance. In particular, we give the proof of Lemma \ref{lemma:indicator-conv}. 

For a collection of loops $\L\in \mathsf L$, we write $\C(\L)\in\mathsf C$ for the collection of carpets induced by $\L$. That is, the closed sets defined by viewing the nested loops in $\L$ as alternating between being inner and outer boundaries of some carpet.

\begin{lemma}\label{lem:bdy-to-hausdorff}
	Let $(\L_\de)_{\de>0}$ be a sequence of collections of loops in $D$ converging to $\L$ with respect to the metric $d_{\mathsf L}$. Suppose that the loops in $\L$ are pairwise disjoint. Then, $(\C(\La_\de))_{\de>0}$ converges to $\C(\L)$ with respect to the metric $d_{\mathsf C}$. 
\end{lemma}
\begin{proof}
	This follows from \cite[Theorem 3.5]{loop-soup-CLE}.
\end{proof}

\noindent A set $C$ is said to be tangent to $A$ if $C\cap A\neq\emptyset$ and $C\cap A^\circ=\emptyset$, where $A^\circ$ is the interior of $A$.

\begin{prop}\label{prop:intersection}
	Let $(C_\de)_{\de>0}$ be a sequence of closed sets converging to $C$ as $\de\to0$ with respect to the Hausdorff metric. Fix a box $A$ and suppose that $C$ is not tangent to $A$. Then,
	\[
	\i(C_\de\cap A\neq\emptyset)\longrightarrow\i(C\cap A\neq\emptyset)\quad \text{as}\ \ \de\to0.
	\]
\end{prop}
\begin{proof}
	First, suppose that $C\cap A=\emptyset$. Then, $d=d_H(C, A)>0$ and for small enough $\de$ we have that $d_H(C_\de, C) < d/2$. Thus $d_H(C_\de,A)> d/2$ eventually, and the claim follows. Otherwise, we must have that $C\cap A^\circ\neq\emptyset$. Choose $z\in C\cap A^\circ$ and $r>0$ such that $B(z, r)\subset A$. For small enough $\de$, there exists $z_\de\in C_\de$ such that $|z_\de-z| < r/2$. Hence $z_\de\in B(z, r)\subset A$ and $C_\de\cap A\neq\emptyset$. 
\end{proof}

\def\SLE{\textup{SLE}}

\begin{lemma}\label{lem:SLE4-touching}
	Let $A\subset D$ be a fixed (deterministic) box. Let $\C$ be the carpet of a $\CLE_4$ in $D$. Then, $\C$ is not tangent to $A$ almost surely.
\end{lemma}
\begin{proof}
	We work with the construction of the $\CLE_4$ in terms of $\SLE_4(-1, -1)$ and $\SLE_4(-1)$ as introduced in \cite[Section 6]{ASW}. First, we note that if $A\cap\C\neq\emptyset$, then necessarily $A\cap(D\setminus \C)\neq\emptyset$. In particular $A$ must intersect some level line $\eta$ used in the construction of the $\CLE_4$.
	
	\noindent By \cite[Theorem 1.3]{MilShe} (see also \cite[Appendix A.3]{BerPow}), we know that $\eta$ is a continuous curve and, by the above, there exists a finite stopping time $\tau$ such that $\eta_\tau\in\partial A$. Moreover, by \cite[Theorem 2.2]{MilShe}, the underlying driving system of SDEs gives rise to an almost surely continuous Markov process. By applying the strong Markov property at time $\tau$, we can invoke Blumenthal's 0-1 law to see that almost surely $\eta([\tau, \infty])\cap A^\circ \neq\emptyset$, which concludes the proof.
\end{proof}

\begin{proof}[Proof of Lemma \ref{lemma:indicator-conv}]
	Immediate by combining Proposition \ref{prop:intersection} and Lemma \ref{lem:SLE4-touching}.
\end{proof}
\pagebreak

\bibliography{refs.bib}{}
\bibliographystyle{plain}

\end{document}